\documentclass[11pt]{amsart}
\usepackage{enumerate}
\usepackage{hyperref}
\usepackage{amsmath,amssymb,amsthm}
\usepackage{tikz}
\usepackage{algorithm}%
\usepackage{algpseudocode}
\usepackage[letterpaper, margin=1in]{geometry}
\usepackage{todonotes}

\usepackage{thmtools}
\usepackage{thm-restate}

\hypersetup{
    pdftitle = {Erd\H{o}s-P\'osa property of chordless cycles},
   pdfauthor=  {Eun Jung Kim and O-joung Kwon}
   }

\newcommand\abs[1]{\lvert #1\rvert}

\newtheorem{THM}{Theorem}[section]
\newtheorem{LEM}[THM]{Lemma}

\newtheorem*{THMMAIN2}{Theorem \ref{thm:main2}}

\newtheorem{COR}[THM]{Corollary}
\newtheorem{PROP}[THM]{Proposition}

\newtheorem{CLAIM}{Claim}

\theoremstyle{remark}

\newenvironment{proofofclaim}{\noindent \textsc{Proof of the Claim:}}{\hfill$\Diamond$\medskip}
\newenvironment{proofof}[1]{\noindent{\bfseries Proof of #1.}}{\qed}

\theoremstyle{definition}

\newcommand\dist{\operatorname{dist}}

\newcommand\pack{\textsf{pack}}
\newcommand\cover{\textsf{cover}}

\newcommand\sle\precsim
\newcommand\sge\succsim
\newcommand\seq\sim

\newcommand\tge\succcurlyeq
\newcommand\teq\cong
\newcommand\ple\sqsubseteq
\newcommand\pge\sqsupseteq

\newcommand{\spp}{\operatorname{\textsf{sp}}}

\begin{document}
\title{Erd\H{o}s-P\'osa property of chordless cycles and its applications}
\author{Eun Jung Kim}
\author{O-joung Kwon}
\address[Kim]{CNRS-Universit\'{e} Paris-Dauphine, Place du Marechal de Lattre de Tassigny, 75775 Paris cedex 16, France}
\address[Kwon]{Department of Mathematics, Incheon National University, Incheon, South Korea}
\address[Kwon]{Discrete Mathematics Group, Institute~for~Basic~Science~(IBS), Daejeon,~South~Korea
}\email{eunjungkim78@gmail.com}
\email{ojoungkwon@gmail.com}
\thanks{The corresponding author is O-joung Kwon. The first author was supported by the French National Research Agency under the PRC programme (ANR grant ESIGMA, ANR-17-CE40-0028). The second author was supported by the European Research Council (ERC) under the European Union's Horizon 2020 research and innovation programme (ERC consolidator grant DISTRUCT, agreement No. 648527), supported by the National Research Foundation of Korea (NRF) grant funded by the Ministry of Education (No. NRF-2018R1D1A1B07050294), and supported by  the Institute for Basic Science (IBS-R029-C1).}
\date{\today}
\begin{abstract}
A chordless cycle, or equivalently a hole, in a graph $G$ is an induced subgraph of $G$ which is a cycle of length at least $4$. 
We prove that the Erd\H{o}s-P\'osa property holds for chordless cycles, which resolves the major open question concerning the Erd\H{o}s-P\'osa property.
Our proof for chordless cycles is constructive: in polynomial time, one can find either $k+1$ vertex-disjoint chordless cycles, or $c_1k^2 \log k+c_2$ vertices 
hitting every chordless cycle for some constants $c_1$ and $c_2$. 
It immediately implies an approximation algorithm of factor $\mathcal{O}(\sf{opt}\log {\sf opt})$ for {\sc Chordal Vertex Deletion}.
We complement our main result by showing that chordless cycles of length at least $\ell$ for any fixed $\ell\ge 5$ do not have the Erd\H{o}s-P\'osa property.

\end{abstract}
\maketitle

\section{Introduction}\label{sec:introduction}

All graphs in this paper are finite and have neither loops nor parallel edges. We denote by $\mathbb{N}$ the set of positive integers.
A class $\mathcal{C}$ of graphs is said to have the \emph{Erd\H{o}s-P\'osa property} if there exists a function $f: \mathbb{N} \rightarrow \mathbb{N}$, called a \emph{gap function}, such that
for every graph $G$ and every positive integer $k$, $G$ contains either
\begin{itemize}
\item $k+1$ pairwise vertex-disjoint subgraphs in $\mathcal{C}$, or
\item a vertex set $T$ of $G$ such that $\abs{T}\le f(k)$ and $G- T$ has no subgraphs in $\mathcal{C}$.
\end{itemize} 
Erd\H{o}s and P\'osa~\cite{ErdosP1965} showed that the class of all cycles has this property with a gap function $\mathcal{O}(k\log k)$.
This breakthrough result sparked an extensive research on finding min-max dualities of packing and covering for various graph families and combinatorial objects. 
Erd\H{o}s and P\'osa also showed that the gap function cannot be improved to $o(k\log k)$.
The result of Erd\H{o}s and P\'osa has been strengthened for cycles with additional constraints; for example, long cycles~\cite{RobertsonS1986, BirmeleBR2007, FioriniH2014, MoussetNSW2016, BruhnJS2017}, directed cycles~\cite{ReedRST1996, KakimuraK2012}, cycles with modularity constraints~\cite{Thomassen1988, HuyneJW2017}, or cycles intersecting a prescribed vertex set~\cite{KakimuraKM2011, PontecorviW2012, BruhnJS2017, HuyneJW2017}. 
Not every variant of cycles has the Erd\H{o}s-P\'osa property; for example, Reed~\cite{Reed1999} showed that  the class of odd cycles does not satisfy the Erd\H{o}s-P\'osa property. 
We generally say that a graph class $\mathcal{C}$ has the  \emph{Erd\H{o}s-P\'osa property} under a graph containment relation $\le_{\star}$ 
if there exists a gap function $f: \mathbb{N} \rightarrow \mathbb{N}$ such that
for every graph $G$ and every positive integer $k$, $G$ contains either
\begin{itemize}
\item $k+1$ pairwise vertex-disjoint subsets $Z_1,\ldots , Z_k$ such that each subgraph of $G$ induced by $Z_i$ contains a member of $\mathcal{C}$ under $\le_{\star}$, or
\item a vertex set $T$ of $G$ such that $\abs{T}\le f(k)$  and $G- T$ contains no member of $\mathcal{C}$ under $\le_{\star}$.
\end{itemize} 
Here, $\le_{\star}$ can be a graph containment relation such as subgraph, induced subgraph, minor, topological minor, induced minor, or induced subdivision. 
An edge version and directed version of the Erd\H{o}s-P\'osa property can be similarly defined. 
In this setting, the Erd\H{o}s-P\'osa properties of diverse undirected and directed graph families have been studied for graph containment relations such as minors~\cite{RobertsonS1986}, immersions~\cite{GianKRT2016, NaonoriK2018}, and (directed) butterfly minors~\cite{AmiriKKW2016}. It is known that the edge-version of the Erd\H{o}s-P\'osa property also holds for cycles~\cite{diestel2012}. Raymond and Thilikos~\cite{RaymondT16} provided an up-to-date overview on the Erd\H{o}s-P\'osa properties for a range of graph families.

In this paper, we study the Erd\H{o}s-P\'osa property for cycles of length at least $4$ under the induced subgraph relation.
An induced cycle of length at least $4$ in a graph $G$ is called a \emph{hole} or a \emph{chordless cycle}.
Considering the extensive study on the topic, it is somewhat surprising that 
whether the Erd\H{o}s-P\'osa property holds for cycles of length at least $4$ under the induced subgraph relation has been left open till now. 
This question was explicitly asked by Jansen and Pilipczuk~\cite{JansenP17} in their study of the \textsc{Chordal Vertex Deletion} problem,
and  was also asked by Raymond and Thilikos~\cite{RaymondT16} in their survey. 
We answer this question positively.

\begin{THM}\label{thm:main}
There exist constants $c_1, c_2$ and a polynomial-time algorithm 
which, given a graph $G$ and a positive integer $k$,   finds either $k+1$ vertex-disjoint holes or a vertex set of size at most $c_1k^2\log k + c_2$ hitting every hole of $G$.
\end{THM}

One might ask whether Theorem~\ref{thm:main} can be extended to the class of cycles of length at least $\ell$ for fixed $\ell\ge 5$.
We present a complementary result that for every fixed $\ell\ge 5$, the class of cycles of length at least $\ell$ does not satisfy the Erd\H{o}s-P\'osa property under the induced subgraph relation.

\begin{THM}\label{thm:main2}
Let $\ell\ge 5$ be an integer. 
Then the class of cycles of length at least $\ell$ does not have the Erd\H{o}s-P\'osa property under the induced subgraph relation.
\end{THM}
In Section~\ref{sec:lowerbound}, 
we additionally argue that 
the class of cycles of length at least $\ell$ for fixed $\ell\ge 5$ does not have the \emph{$1/ \alpha$-integral Erd\H{o}s-P\'osa property} under the induced subgraph relation, for all integers $\alpha\ge 2$. 

Theorem~\ref{thm:main} is closely related to the \textsc{Chordal Vertex Deletion} problem.
The \textsc{Chordal Vertex Deletion} problem asks whether, for a given graph $G$ and a positive integer $k$, there exists a vertex set $S$ of size at most $k$ such that $G-S$ has no holes; 
	in other words, $G-S$ is a chordal graph.
	In  parameterized complexity, whether or not \textsc{Chordal Vertex Deletion} admits a polynomial kernelization was one of major open problems since it was first mentioned by Marx~\cite{Marx10}.
	 A \emph{polynomial kernelization} of a parameterized problem is a polynomial-time algorithm 
	 that takes an instance $(x,k)$ and outputs an instance $(x', k')$ such that 
	 (1) $(x,k)$ is a \textsc{Yes}-instance if and only if $(x', k')$ is a \textsc{Yes}-instance, and
	 (2) $k'\le k$, and $\abs{x'}\le g(k)$ for some polynomial function $g$.

Jansen and Pilipczuk~\cite{JansenP17}, and independently Agrawal \emph{et~al.}~\cite{AgrawalLMSZ17}, presented polynomial kernelizations for the \textsc{Chordal Vertex Deletion} problem. In both works, an approximation algorithm for the optimization version of this problem emerges as an important subroutine. Jansen and Pilipczuk~\cite{JansenP17} obtained an approximation algorithm of factor $\mathcal{O}({\sf opt}^2 \log {\sf opt} \log n)$ using iterative decomposition of the input graph and linear programming. Agrawal \emph{et~al.}~\cite{AgrawalLMSZ17} obtained an algorithm of factor $\mathcal{O}({\sf opt}\log^2 n)$ based on divide-and-conquer. 
As one might expect, the factor of an approximation algorithm for the \textsc{Chordal Vertex Deletion} is intrinsically linked to the quality of  the polynomial kernels 
attained in \cite{JansenP17} and~\cite{AgrawalLMSZ17}.
We point out that the polynomial-time algorithm of Theorem of~\ref{thm:main} can be easily converted into an approximation algorithm of  factor  $O({\sf opt}\log {\sf opt})$.

\begin{THM}\label{thm:main3}
There is an approximation algorithm of factor $O({\sf opt}\log {\sf opt})$ for {\sc Chordal Vertex Deletion}. 
\end{THM}

It should be noted that an $\mathcal{O}(\log^2 n)$-factor approximation algorithm 
was presented recently by Agrawal \emph{et~al.}~\cite{AgrawalLMSZ17b}, which outperforms 
the approximation algorithm of Theorem~\ref{thm:main3}, when applied to the kernelization algorithm.

Our result has another application on packing and covering for weighted cycles.
For a graph $G$ and a non-negative weight function $w:V(G)\rightarrow \mathbb{N}\cup \{0\}$, 
let $\pack(G, w)$ be the maximum number of cycles (repetition is allowed) such that each vertex $v$ is used at most $w(v)$ times, and
let $\cover(G, w)$ be the minimum value $\sum_{v\in X} w(v)$ where $X$ hits all cycles in $G$.
Ding and Zang~\cite{DingZ2002} characterized \emph{cycle Mengerian graphs} $G$, which satisfy the property that for all non-negative weight function $w$, $\pack(G,w)=\cover(G,w)$.
Up to our best knowledge, it was not previously known whether $\cover(G,w)$ can be bounded by a function of $\pack(G,w)$.

As a corollary of Theorem~\ref{thm:main}, we show the following.

\begin{COR}
For a graph $G$ and a non-negative weight function $w:V(G)\rightarrow \mathbb{N}\cup \{0\}$, 
$\cover(G,w)\le \mathcal{O}(k^2\log k)$, where $k=\pack(G,w)$.
\end{COR}

The paper is organized as follows. Section~\ref{sec:prelim} provides basic  notations and previous results that are relevant to our result. 
In Section~\ref{sec:overview}, we explain how to reduce the proof of Theorem~\ref{thm:main} to a proof under a specific premise, in which  
we are given a shortest hole $C$ of $G$ 
such that $C$ has length more than $d_1k\log k + d_2$ for some constants $d_1, d_2$ and $G-V(C)$ is chordal.
In this setting, we introduce further technical notations and terminology.  An outline of our proof will be also given in this section.
We present some structural properties of a shortest hole $C$ and its neighborhood in Section~\ref{sec:lemmas}.
In Sections~\ref{sec:hittingsunflower} and \ref{sec:tulip}, we prove the Erd\H{o}s-P\'osa property for different types of holes intersecting $C$ step by step, and we conclude Theorem~\ref{thm:main} at the end of Section~\ref{sec:tulip}.
Section~\ref{sec:lowerbound} demonstrates that the class of cycles of length at least $\ell$, for every fixed $\ell\ge 5$, does not have the Erd\H{o}s-P\'osa property under the induced subgraph relation. 
Section~\ref{sec:applications} illustrates the implications of Theorem~\ref{thm:main} to weighted cycles and to the \textsc{Chordal Vertex Deletion} problem.

\section{Preliminaries}\label{sec:prelim}

For a graph $G$, we denote by $V(G)$ and $E(G)$ the vertex set and the edge set of $G$, respectively.
Let $G$ be a graph. 
For a vertex set $S$ of $G$, let $G[S]$ denote the subgraph of $G$ induced by $S$, and 
let $G-S$ denote the subgraph of $G$ obtained by removing all vertices in $S$.
For $v\in V(G)$, we let $G-v:=G-\{v\}$.
If $uv\in E(G)$,  we say that $u$ is a \emph{neighbor} of $v$. 
The set of neighbors of a vertex $v$ is denoted by $N_G(v)$, and the \emph{degree} of $v$ is defined as the size of $N_G(v)$.
The \emph{open neighborhood} of  a vertex set $A\subseteq V(G)$ in $G$, denoted by $N_G(A)$, is the set of vertices in $V(G)\setminus A$ having a neighbor in $A$. The set $N_G(A)\cup A$ is called the \emph{closed neighborhood} of $A$, and denoted by $N_G[A]$. 
For convenience, we define these neighborhood operations for subgraphs as well; that is, for a subgraph $H$ of $G$, 
let $N_G(H):=N_G(V(H))$ and $N_G[H]:=N_G[V(H)]$.
When the underlying graph is clear from the context, we drop the subscript $G$. 
A vertex set $S$ of a graph is a \emph{clique} if every pair of vertices in $S$ is adjacent, and  
it is an \emph{independent set} if every pair of  vertices in $S$ is non-adjacent.
For two subgraphs $H$ and $F$ of $G$, the \emph{restriction} of $F$ on $H$ is defined as the graph $F[V(F)\cap V(H)]$.

A \emph{walk} is a non-empty alternating sequence of vertices and edges of the form $(x_0,e_0,\ldots  ,e_{\ell-1},x_{\ell})$, beginning and ending with vertices, such that for every $0\leq i\leq \ell-1$,  $x_{i}$ and $x_{i+1}$ are endpoints of $e_i$. 
 A \emph{path} is a walk in which vertices are pairwise distinct. 
For a path $P$ on vertices $x_0,\ldots , x_{\ell}$ with edges $x_ix_{i+1}$ for $i=0,1, \ldots , \ell-1$, we write $P=x_0x_1 \cdots x_{\ell}$.
It is also called an $(x_0, x_{\ell})$-path.
We say $x_{i}$ is the $i$-th neighbor of $x_0$, and similarly, $x_{\ell-i}$ is the $i$-th neighbor of $x_{\ell}$ in $P$.
A \emph{cycle} is a walk $(x_0, e_0, \ldots, e_{\ell-1}, x_{\ell})$ in which vertices are pairwise distinct except $x_0=x_{\ell}$. 
For a cycle $C$ on $x_0,x_1, \ldots , x_{\ell}$ with edges $x_ix_{i+1}$ for $i=0,1, \ldots , \ell-1$ and $x_{\ell}x_0$, 
we write $C=x_0x_1 \cdots x_{\ell}x_0$. If a cycle or a path $H$ is an induced subgraph of the given graph $G$,
then we say that $H$ is an induced cycle or an induced path in $G$, respectively. 

A subpath of a path $P$ starting at $x$ and ending at $y$ is denoted as $xPy$. 
In the notation $xPy$, we may replace $x$ or $y$ with $\mathring{x}$ or $\mathring{y}$, to obtain a subpath starting from the neighbor of $x$ in $P$ closer to $y$ or ending at the neighbor of $y$ in $P$ closer to $x$, respectively.
For instance, $xP\mathring{y}$ refers to the subpath of $P$ starting at $x$ and ending at the neighbor of $y$ in $P$ closer to $x$. 
Given two walks $P=(v_0,e_0,\ldots  ,e_{p-1},v_{p})$ and $Q=(u_{0},f_{0},\ldots  ,f_{q-1},u_{q})$ such that $v_p=u_0$, the \emph{concatenation} of $P$ and $Q$ is  
defined as the walk $(v_0,e_0,\ldots  ,e_{p-1},v_{p}(=u_0),f_{0},\ldots  ,f_{q-1},u_{q})$, which we denote as $P\odot Q$.
Note that for two internally vertex-disjoint paths $P_1$ and $P_2$ from $v$ to $w$, 
$vP_1w\odot wP_2v$ denotes the cycle passing through $P_1$ and $P_2$.

Given a graph $G$, the \emph{distance} between two vertices $x$ and $y$ in $G$ is defined as the length of a shortest $(x,y)$-path and denoted as $\dist_G(x,y)$. If $x=y$, then we define $\dist_G(x,y)=0$, and $\dist_G(x,y)=\infty$ if there is no $(x,y)$-path in $G$. The distance between two vertex sets $X,Y\subseteq V(G)$, written as $\dist_G(X,Y)$, is the minimum $\dist_G(x,y)$ over all $x\in X$ and $y\in Y$. If $X=\{x\}$, then we write $\dist_G(X,Y)$ as $\dist_G(x,Y)$. For a vertex subset $S$ of $G$, a vertex set $U$ is the \emph{$r$-neighborhood} of $S$ in $G$ if it is the set of all vertices $w$ such that $\dist_G(w, S)\le r$. 
We denote the $r$-neighborhood of $S$ in $G$ as $N^r_G[S]$. When the underlying graph $G$ is clear from the context, we omit the subscript $G$.

Given a cycle $C=x_0x_1 \cdots x_{\ell}x_0$, an edge $e$ of $G$ is a \emph{chord} of $C$ if both endpoints of $e$ are contained in $V(C)$ 
but $e$ is not an edge of $C$. 
A \emph{hole} in a graph $G$ is an induced cycle of length at least 4 in $G$. 
A hole is also called as a \emph{chordless cycle}.
A graph is \emph{chordal} if it has no holes.
A vertex set $T$ of a graph $G$ is called a \emph{chordal deletion set} if $G-T$ is chordal.

Given a vertex set $S\subseteq V(G)$, a path $P$ is called an \emph{$S$-path} if the endpoints of $P$ are vertices of $S$ and all internal vertices are contained in $V(G)\setminus S$. An $S$-path is called \emph{proper} if it has at least one internal vertex.  An $(A,B)$-path of a graph $G$ is a path $v_0v_1\cdots v_{\ell}$ such that $v_0\in A$, $v_{\ell}\in B$ and all internal vertices are in $V(G)\setminus (A\cup B)$. Observe that every path from $A$ to $B$ contains an $(A,B)$-path. If $A$ or $B$ is a singleton, then we omit the bracket from the set notation. A vertex set $S$ is an \emph{$(A,B)$-separator} if $S$ disconnects all $(A,B)$-paths in $G$.

We recall Menger's Theorem.
\begin{THM}[Menger's Theorem; See for instance \cite{diestel2012}]\label{thm:menger}
Let $G$ be a  graph and $A,B\subseteq V(G)$. Then the size of a minimum $(A,B)$-separator in $G$  equals the maximum number of vertex-disjoint $(A,B)$-paths in $G$. 
Furthermore, one can output either one of them in polynomial time.
\end{THM}
A \emph{bipartite graph} is a graph $G$ with a vertex bipartition $(A,B)$ in which each of $G[A]$ and $G[B]$ is edgeless.
A set $F$ of edges in a graph is a \emph{matching} if no two edges in $F$ have a common endpoint.
A vertex set $S$ of a graph $G$ is a \emph{vertex cover} if $G-S$ has no edges.
By Theorem~\ref{thm:menger}, 
given a bipartite graph with a bipartition $(A,B)$, 
one can find a maximum matching or a minimum vertex cover in polynomial time.

The following result is useful to find many vertex-disjoint cycles in a graph of maximum degree $3$. All logarithms in this paper are taken to base 2. We define $s_k$ for $k\in \mathbb{N}$ as
\begin{align*}
s_k=
\begin{cases}
4k(\log k + \log \log k +4) \quad &\text{if } k\geq 2\\
2 &\text{if } k=1.
\end{cases}
\end{align*}
\begin{THM}[Simonovitz~\cite{Simonovits1967}]\label{thm:simonovitz}
Let $G$ be a graph  all of whose vertices have degree $3$ and let $k$ be a positive integer. If $\abs{V(G)}\geq s_k$, then $G$ contains at least $k$ vertex-disjoint cycles. Furthermore, such $k$ cycles can be found in polynomial time.
\end{THM}

Lastly, we present lemmas which are useful for detecting a hole.

\begin{LEM}\label{lem:twopaths}
Let $H$ be a graph and $x,y\in V(H)$ be  two distinct vertices. Let $P$ and $Q$ be internally vertex-disjoint $(x,y)$-paths such that $Q$ contains an internal vertex $w$ having no neighbor in $V(P)\setminus \{x,y\}$. If $Q$ is an induced path, then $H[V(P)\cup V(Q)]$ has a hole containing $w$.
\end{LEM}
\begin{proof}
Let $x_1$ and $x_2$ be the neighbors of $w$ in $Q$. As $xPy\odot yQx$ is a cycle, $H[V(P)\cup V(Q)]-w$ is connected.
Let $R$ be a shortest $(x_1, x_2)$-path in $H[V(P)\cup V(Q)]-w$. As the only neighbors of $w$  contained in $(V(P)\cup V(Q)) \setminus \{w\}$ are $x_1$ and $x_2$, 
$w$ has no neighbors in the internal vertices of $R$. Note that $R$ has length at least $2$ since $x_1,x_2\in V(Q)$ and $Q$ is an induced path.
Therefore, $x_1Rx_2\odot x_2wx_1$ is a hole containing $w$, as required.
\end{proof}

A special case of Lemma~\ref{lem:twopaths} 
is when there is a vertex $w$ in a cycle $C$ such that $w$ has no neighbors in the internal vertices of $C-w$ and the 
neighbors of $w$ on $C$ are non-adjacent. 
In this case, $C$ has a hole containing $w$ by Lemma~\ref{lem:twopaths}.

One can test in polynomial time whether a graph contains a hole or not.
\begin{LEM}\label{lem:detectinghole}
Given a graph $G$, one can test in polynomial time whether it has a hole or not.
Furthermore, one can find in polynomial time a shortest hole of $G$, if one exists.
\end{LEM}
\begin{proof}
We guess three vertices $v,w,z$ where $vw, wz\in E(G)$ and $vz\notin E(G)$, and 
test whether there is a path from $v$ to $z$ in $G-(N_G[w]\setminus \{v,z\})$.
If there is such a path, then we choose a shortest path $P$ from $v$ to $z$. As $w$ has no neighbors in the set of internal vertices of $P$, 
$V(P)\cup \{w\}$ induces a hole. Clearly if $G$ has a hole, then we can find one by the above procedure.

To find a shortest one, for every such a guessed tuple $(v,w,z)$, we keep the length of the obtained hole.
Then it is sufficient to output a hole with minimum length among all obtained holes.
\end{proof}

\section{Terminology and a  proof overview}\label{sec:overview} 

The proof of Theorem~\ref{thm:main} begins by finding a sequence of shortest holes. Let $G$ be the input graph and let $G_1=G$. 
For each $i=1,2,\ldots$, we iteratively find a shortest hole $C_i$ in $G_i$ and set $G_{i+1}:=G_i - V(C_i)$. If the procedure fails to find a hole at $j$-th iteration, then $G_j$ is a chordal graph. 
This iterative procedure leads us to the following theorem, which is the core component of our  result.

For $k\in \mathbb{N}$, we define
$\mu_k=76s_{k+1}+3217k+1985$.

\begin{THM}\label{thm:core}
Let $G$ be a graph, $k$ be a positive integer, and $C$ be a shortest hole of $G$
such that $C$ has length strictly larger than $\mu_k$ and $G-V(C)$ is chordal. Given such $G$, $k$, and $C$, one can find in polynomial time either $k+1$ vertex-disjoint holes or a vertex set $X\subseteq V(G)$ of size at most $\mu_{k}$ hitting every hole of $G$.
\end{THM}

It is easy to derive our main result from Theorem~\ref{thm:core}.

\smallskip

\begin{proofof}{Theorem~\ref{thm:main}}
We construct sequences $G_1,\ldots, G_{\ell+1}$ and $C_1,\ldots , C_{\ell}$ 
such that 
\begin{itemize}
\item $G_1=G$, 
\item for each $i\in \{1, 2, \ldots, \ell\}$, $C_i$ is a shortest hole of $G_i$,  
\item for each $i\in \{1, 2, \ldots, \ell\}$, $G_{i+1}=G_i-V(C_i)$, and
\item $G_{\ell+1}$ is chordal.
\end{itemize}
Such a sequence can be constructed in polynomial time repeatedly applying Lemma~\ref{lem:detectinghole} to find a shortest hole.
If $\ell \geq k+1$, then we have found a  packing of $k+1$ holes. 
Hence, we assume that $\ell \leq k$. 

We prove the following claim for $j=\ell+1$ down to $j=1$. 
\begin{quote}
One can find in polynomial time either $k+1$ vertex-disjoint holes, or 
a chordal deletion set $T_{j}$ of $G_{j}$ of size at most $(\ell+1-j)\mu_k$. 
\end{quote}

The claim trivially holds for $j=\ell+1$ with $T_{\ell+1}=\emptyset$ because $G_{\ell+1}$ is chordal.
Let us  assume that for some $j\leq \ell$, we obtained a chordal deletion set $T_{j+1}$ of $G_{j+1}$ of size at most $(\ell-j)\mu_k$.
Then in $G_{j}-T_{j+1}$, $C_{j}$ is a shortest hole, and $\left( G_{j}-T_{j+1} \right)-V(C_{j})$ is chordal.
If $C_{j}$ has length at most $\mu_k$, then we set $T_{j}:=T_{j+1}\cup V(C_{j})$. Clearly, $\abs{T_j}\leq (\ell-j+1)\mu_k$.
Otherwise, by applying Theorem~\ref{thm:core} to $G_{j}-T_{j+1}$ and $C_{j}$, 
one can find in polynomial time either $k+1$ vertex-disjoint holes or a chordal deletion set $X$ of size at most $\mu_k$ of $G_{j}-T_{j+1}$. 
In the former case, we output $k+1$ vertex-disjoint holes, and we are done. If we obtain a chordal deletion set $X$, then we set $T_{j}:=T_{j+1}\cup X$. Observe that the set $T_{j}$ is a chordal deletion set of $G_{j}$ and $\abs{T_{j}}\le (\ell-j+1)\mu_k$ as claimed. 

From the claim with $j=1$, we conclude that 
in polynomial time, one can find either $k+1$ vertex-disjoint holes, 
or a chordal deletion set of $G_1=G$ of size at most $\ell\mu_k\le k\mu_k$. 
So, in the latter case, there exist constants $c_1, c_2$ such that $G$ admits a chordal deletion set of size $c_1k^2\log k +c_2$ ($c_2$ is necessary for $k=1$).  
\end{proofof}

\begin{figure}
  \centering
  \begin{tikzpicture}[scale=0.7]
  \tikzstyle{w}=[circle,draw,fill=black!50,inner sep=0pt,minimum width=4pt]

   \draw (-2,0)--(11,0);
	\draw(-2, 0)--(-3,-0.5);
	\draw(11, 0)--(12,-0.5);
      \draw (-3,-.5) node [w] {};
       \draw (12,-.5) node [w] {};
 	\draw[dashed](13, -1)--(12,-0.5);
	\draw[dashed](-4,-1)--(-3,-0.5);
 
 \foreach \y in {-2,...,11}{
      \draw (\y,0) node [w] (a\y) {};
     }
 \foreach \y in {2,3,4}{
	\draw (3,2)--(a\y);
    }
 \foreach \y in {4,5}{
	\draw (4.5,2)--(a\y);
    }
    \draw(3,2)--(4.5,2);

 \foreach \y in {7,8,9}{
	\draw (8,2)--(a\y);
    }
       \draw (3,2) node [w] {};
       \draw (8,2) node [w] {};
       \draw (4.5,2) node [w] {};

\draw[dashed, rounded corners] (4, 2.3)--(5, 2.3)--(4, -.5)--(2.5, 2.3)--(4, 2.3);

\draw[rounded corners] (-3,3)--(-3,2)--(0,2)--(0,4)--(-3,4)--(-3,3);
 \foreach \y in {-2,-1.5,-1, -.5,0}{
	\draw (-1.5,2.5)--(\y, 1.5);
    }
       \draw (-1.5, 2.5) node [w] {};

     \node at (-1.5, 3) {$D$};
     \node at (-2, -1) {$C$};
     \node at (4, 2.7) {$Z_v$};
     \node at (4, -.8) {$v$};

   \end{tikzpicture}     \caption{The set of vertices adjacent to all vertices of $C$ is denoted by $D$, and for each $v\in V(C)$, $Z_v$ 
   denotes the set $\{v\}\cup (N(v)\setminus V(C)\setminus D)$. Using the fact that $C$ is chosen as a shortest hole and it is long, we will prove in Lemma~\ref{lem:consecutive} that each vertex in $N(C)\setminus D$ has at most $3$ neighbors on $C$ and they are consecutive in $C$. }\label{fig:setting}
\end{figure}
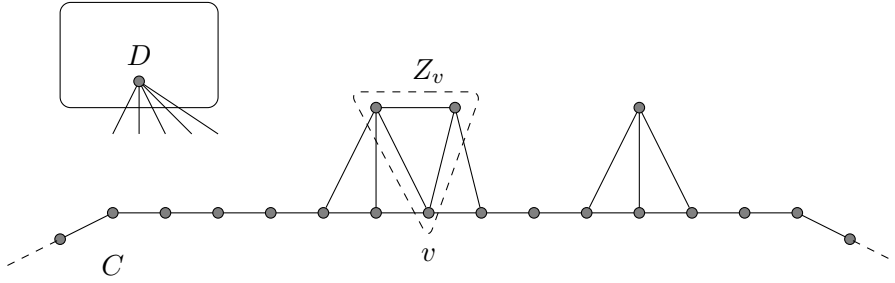

\medskip

The rest of this section and Sections~\ref{sec:lemmas}-\ref{sec:tulip} are devoted to establish Theorem~\ref{thm:core}. 
Throughout these sections, we shall consider the input tuple $(G,k, C)$ of Theorem~\ref{thm:core} as fixed. 

Let us introduce the notations that are frequently used (see Figure~\ref{fig:setting}).
A vertex $v\in N(C)$ is \emph{$C$-dominating} if $v$ is adjacent to every vertex on $C$.
We reserve $D$ to denote the set of all $C$-dominating vertices.
For each vertex $v$ in $C$, we denote by $Z_v:=\{v\}\cup (N(v)\setminus V(C)\setminus D )$, and 
for a subset $S$ of $V(C)$, we denote by $Z_S:=\bigcup_{v\in S}Z_v$.
We also define
\begin{itemize}
\item $G_{deldom}:=G-D$ and $G_{nbd}:=G[N[C]\setminus D]$. 
\end{itemize}
For a subpath $Q$ of $C$, the subgraph of $G$ induced by $Z_{V(Q)}$ is called a \emph{$Q$-tunnel}. 
By definition of $Z_{V(Q)}$, a $Q$-tunnel is an induced subgraph of $G_{nbd}$. 
When $q, q'$ are endpoints of $Q$, we say that $Z_{q}$ and $Z_{q'}$ are \emph{entrances} of the $Q$-tunnel.

For a subgraph $H$ of $G$, the \emph{support} of $H$, denoted by $\spp (H)$, is the set of all vertices  $v\in V(C)$ such that $(Z_v\cup D)\cap V(H)\neq\emptyset$. 
Observe that if $H$ contains a vertex of $D$, then trivially $\spp(H)=V(C)$.

We distinguish between two types of holes, namely \emph{sunflowers} and \emph{tulips}. 
A hole $H$ is said to be a \emph{sunflower} if $V(H)\subseteq N[C]$, that is, its entire vertex set is placed within the closed neighborhood of $C$.
A hole that is not a sunflower is called a \emph{tulip}. Every tulip contains at least one vertex not contained in $N[C]$.
Also, we classify holes depending on whether one contains a $C$-dominating vertex or not.
A hole is \emph{$D$-traversing} it contains a $C$-dominating vertex (which is a vertex of $D$), and \emph{$D$-avoiding} otherwise.

In the remainder of this section, we present a proof outline of Theorem~\ref{thm:core}. 
Here are three basic observations, necessary to give the ideas of our proofs.
\begin{itemize}
\item (Lemma~\ref{lem:consecutive})
For every vertex $v$ of $N(C)$, either it has at most $3$ neighbors in $C$ and they are consecutive in $C$, or it is $C$-dominating.
\item (Lemma~\ref{lem:farnonadj})
Let $x,y$ be two vertices in $C$ such that $\dist_C(x,y)\geq 4$. Then there is no edge between $Z_x$ and $Z_y$. See Figure~\ref{fig:distance} for an illustration.
\item (Lemma~\ref{lem:connectedsupport})
Let $H$ be a connected subgraph in $G_{nbd}$. Then $C[\spp(H)]$ is connected.
\item (Lemma~\ref{lem:dominating})
$D$ is a clique.
\end{itemize}

\begin{figure}
  \centering
  \begin{tikzpicture}[scale=0.7]
  \tikzstyle{w}=[circle,draw,fill=black!50,inner sep=0pt,minimum width=4pt]

   \draw (-2,0)--(11,0);
	\draw(-2, 0)--(-3,-0.5);
	\draw(11, 0)--(12,-0.5);
      \draw (-3,-.5) node [w] {};
       \draw (12,-.5) node [w] {};
 	\draw[dashed](13, -1)--(12,-0.5);
	\draw[dashed](-4,-1)--(-3,-0.5);
 
 \foreach \y in {-2,...,11}{
      \draw (\y,0) node [w] (a\y) {};
     }

      \draw[dashed](8,2) [in=30,out=150] to (3,2);

 \foreach \y in {2,3,4}{
	\draw (3,2)--(a\y);
    }
 
 \foreach \y in {7,8,9}{
	\draw (8,2)--(a\y);
    }
       \draw (3,2) node [w] {};
       \draw (8,2) node [w] {};

\draw[rounded corners] (-3,3)--(-3,2)--(0,2)--(0,4)--(-3,4)--(-3,3);
 \foreach \y in {-2,-1.5,-1, -.5,0}{
	\draw (-1.5,2.5)--(\y, 1.5);
    }
       \draw (-1.5, 2.5) node [w] {};

     \node at (-1.5, 3) {$D$};
     \node at (-2, -1) {$C$};
     \node at (2, -.8) {$x$};
	\node at (9, -.8) {$y$};
    \node at (2.6, 2.4) {$v$};
    \node at (8.4, 2.4) {$w$};
 
   \end{tikzpicture} \caption{Illustration of Lemma~\ref{lem:farnonadj}: if $\dist_C(x,y)\ge 4$, then there are no edges between $Z_x$ and $Z_y$.
   For instance, suppose $v$ and $w$ are adjacent. 
   Note that $v$ and $w$ have at most $3$ consecutive neighbors in $C$.
   If the distance from $N(v)\cap V(C)$ to $N(w)\cap V(C)$ in $C$ is at least $1$, then we can find a hole shorter than $C$ using the shortest path from $N(v)\cap V(C)$ to $N(w)\cap V(C)$ in $C$.
   Otherwise, we have $\dist_C(x,y)=4$ and $\abs{N(v)\cap V(C)}=\abs{N(w)\cap V(C)}=3$ and thus, the longer path between $N(v)\cap V(C)$ to $N(w)\cap V(C)$ in $C$ creates a hole shorter than $C$.}\label{fig:distance}
\end{figure}
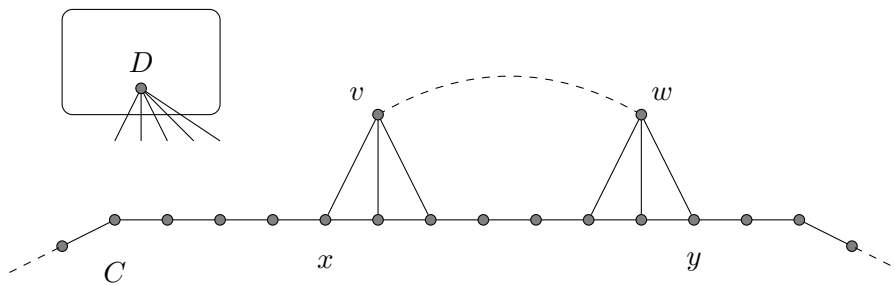

\subsection{$D$-avoiding sunflowers.} 

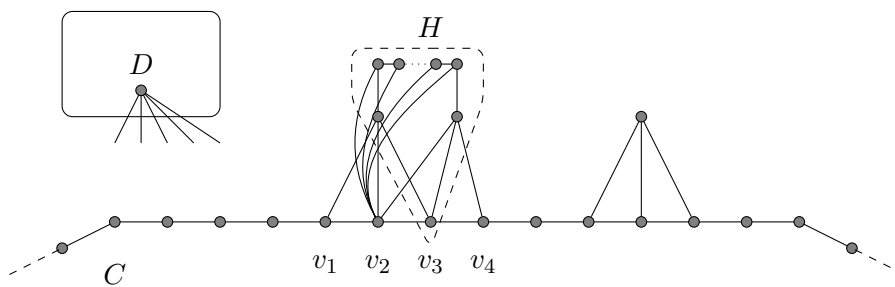
\begin{figure}
  \centering
  \begin{tikzpicture}[scale=0.7]
  \tikzstyle{w}=[circle,draw,fill=black!50,inner sep=0pt,minimum width=4pt]

	     \draw(3,3) [in=120,out=-120] to (3,0);
	     \draw(3.4,3) [in=120,out=-120] to (3,0);
	     \draw(4.1,3) [in=120,out=-140] to (3,0);
	     \draw(4.5,3) [in=120,out=-140] to (3,0);
	 \draw(4.5,2) -- (3,0);

   \draw (-2,0)--(11,0);
	\draw(-2, 0)--(-3,-0.5);
	\draw(11, 0)--(12,-0.5);
      \draw (-3,-.5) node [w] {};
       \draw (12,-.5) node [w] {};
 	\draw[dashed](13, -1)--(12,-0.5);
	\draw[dashed](-4,-1)--(-3,-0.5);
 
 \foreach \y in {-2,...,11}{
      \draw (\y,0) node [w] (a\y) {};
     }
 \foreach \y in {2,3,4}{
	\draw (3,2)--(a\y);
    }
 \foreach \y in {4,5}{
	\draw (4.5,2)--(a\y);
    }
    \draw (3,2)--(3,3);
    \draw(4.5,3)--(4.5,2);
	\draw(3,3)--(3.4,3);\draw[dotted](3.4,3)--(4.1,3);\draw(4.1,3)--(4.5,3);
	     
 \foreach \y in {7,8,9}{
	\draw (8,2)--(a\y);
    }
       \draw (3,2) node [w] {};
       \draw (3,3) node [w] {};
       \draw (8,2) node [w] {};
       \draw (4.5,2) node [w] {};
       \draw (4.5,3) node [w] {};

	   \draw (3.4,3) node [w] {};
       \draw (4.1,3) node [w] {};
       
\draw[dashed, rounded corners] (4, 3.3)--(5,3.3)--(5, 2.3)--(4, -.5)--(2.5, 2.3)--(2.5,3.3)--(4, 3.3);

\draw[rounded corners] (-3,3)--(-3,2)--(0,2)--(0,4)--(-3,4)--(-3,3);
 \foreach \y in {-2,-1.5,-1, -.5,0}{
	\draw (-1.5,2.5)--(\y, 1.5);
    }
       \draw (-1.5, 2.5) node [w] {};

     \node at (-1.5, 3) {$D$};
     \node at (-2, -1) {$C$};
     \node at (4, 3.7) {$H$};

     \node at (2, -.8) {$v_1$};
     \node at (3, -.8) {$v_2$};
     \node at (4, -.8) {$v_3$};
     \node at (5, -.8) {$v_4$};

   \end{tikzpicture}     \caption{The cycle $H$ is a petal having the support $\{v_1, v_2, v_3, v_4\}$. A petal can be arbitrarily long.}\label{fig:petal}
\end{figure}

The set of $D$-avoiding sunflowers is categorized into two subgroups, \emph{petals} and \emph{full sunflowers}.
Note that by Lemma~\ref{lem:connectedsupport}, the support of every $D$-avoiding sunflower consists of consecutive vertices of $C$. 
For a $D$-avoiding sunflower $H$, we say that 
\begin{itemize}
\item it is a \emph{petal} if $\abs{\spp(H)}\leq 7$,  and 
\item it is \emph{full} if $\spp(H)=V(C)$.
\end{itemize}
See Figure~\ref{fig:petal} for an illustration of a petal.
\medskip

\noindent {[Subsection~\ref{subsec:petal}.]} We first obtain a small hitting set of petals, unless $G$ has  $k+1$ vertex-disjoint holes.
For this, we greedily pack petals and mark their supports on $C$.
Clearly, if there are $k+1$ petals whose supports are pairwise disjoint,  then we can find $k+1$ vertex-disjoint holes. 
Thus, we can assume that there are at most $k$ petals whose supports are pairwise disjoint.  
We take the union of all those supports and call it $T_1$. By construction, for every petal $H$, $\spp(H)\cap T_1\neq\emptyset$.
Then we take the $6$-neighborhood of $T_1$ in $C$ and call it $T_{petal}$.
It turns out that 
\begin{itemize}
\item[($\ast$)] for every petal $H$, $\spp (H)$ is fully contained in $T_{petal}$, 
\end{itemize} 
and in particular, $V(H)\cap T_{petal}\neq \emptyset$. The size of $T_{petal}$ is at most $19k$.
\medskip

\noindent {[Subsection~\ref{subsec:allisfull}.]} 
Somewhat surprisingly, we show that every $D$-avoiding sunflower that does not intersect $T_{petal}$ is a full sunflower.  
It is possible that there is a sunflower with support of size at least $8$ and less than $\abs{V(C)}$.
We argue that if such a sunflower $H$ exists, then there is a vertex $v\in V(C)\cap V(H)$ and a petal whose support contains $v$. 
But the property $(\ast)$ of $T_{petal}$ implies that $T_{petal}$ contains $v$, which implies that such a sunflower should be hit by $T_{petal}$.
Therefore, it is sufficient to hit full sunflowers for hitting all remaining $D$-avoiding sunflowers.
\medskip

\noindent {[Subsection~\ref{subsec:sunflowerwithout}.]} 
We obtain a small hitting set of full sunflowers, when $G$ has no $k+1$ vertex-disjoint holes.
For this, we consider two vertex sets $Z_v$ and $Z_w$ for some $v$ and $w$ on $C$, 
and apply Menger's theorem for two directions, say `north' and `south' hemispheres around $C$, between $Z_v$ and $Z_w$ in the graph $G_{nbd}$.
We want to argue that if there are many paths in both directions, then we can find many vertex-disjoint holes. 
However, it is not clear how to link  two families of paths.

To handle this issue, we find two families of paths whose supports cross on constant number of vertices. 
Since $C$ is much larger than the obtained hitting set $T_{petal}$ for petals, we can find $25$ consecutive vertices that contain no vertices in $T_{petal}$.
Let $v_{-2}, v_{-1}, v_0, \ldots, v_{22}$ be such a set of consecutive vertices.
Let $\mathcal{P}$ be the family of vertex-disjoint paths from $Z_{v_0}$ to $Z_{v_{20}}$ whose supports are contained in $\{v_{-2}, v_{-1}, v_0, v_1, \ldots, v_{20}, v_{21}, v_{22}\}$, and let $\mathcal{Q}$ be the family of vertex-disjoint paths from $Z_{v_5}$ to $Z_{v_{16}}$ whose supports do not contain $v_8$ and $v_{13}$. 
Non-existence of petals with support intersecting $\{v_{-2}, v_{-1}, \ldots, v_5\}$ implies that 
paths in $\mathcal{P}$ and $\mathcal{Q}$ are `well-linked' at $Z_{v_0}$ except for few paths, and a symmetric argument holds at $Z_{v_{20}}$. 
This allows us to link any pair of paths from $\mathcal{P}$ and $\mathcal{Q}$. 
If one of $\mathcal{P}$ and $\mathcal{Q}$ is small, then we can output a hitting 
set of full sunflowers using Menger's theorem. The size of the obtained set $T_{full}$ will be at most $3k+14$.

\subsection{$D$-traversing sunflowers}
[Subsection~\ref{subsec:sunflowerwith}.] 
It is easy to see that every $D$-traversing hole $H$ contains exactly one vertex of $D$ (since $D$ is a clique and $H$ contains a vertex of $C$), and every vertex of $V(C)\cap V(H)$ is a neighbor  of the vertex in $D\cap V(H)$. Let $v\in V(C)\cap V(H)$ and $d\in D\cap V(H)$ be such an adjacent pair.
We argue that $H$ contains a subpath $Q$ in $N[C]\setminus D$ that starts from $v$ and is contained in $Z_{\{v, v_2, v_3\}}$ for some three consecutive vertices $v, v_2, v_3$ of $C$, such that
\begin{itemize}
\item $G[V(Q)\cup \{v, v_2, v_3, d\}]$ contains a $D$-traversing sunflower containing $d$ and $v$.
\end{itemize}
 In other words, even if $H-d$ has large support, we can find another $D$-traversing sunflower $H'$ 
 containing $d$ and $v$ where $H'-d$ has support on small number of vertices. 
 The fact that $H$ and $H'$ share $v$ is important as we will take one of $d$ and $v$ as a hitting set for such $H'$, and this will hit $H$ as well.

To this end, we create an auxiliary bipartite graph in which one part is $D$ and the other part consists of sets of $3$ consecutive vertices $v_1, v_2, v_3$ of $C$, and we add an edge 
between $d\in D$ and $\{v_1, v_2, v_3\}$ if $G[Z_{\{v_1, v_2, v_3\}}\cup \{d\}]$ contains a $D$-traversing hole.
We argue that if this bipartite graph has a large matching, then we can find many vertex-disjoint holes, 
and otherwise, we have a small vertex cover. The union of all vertices involved in the vertex cover suffices 
to cover all $D$-traversing sunflowers. The hitting set $T_{trav:sunf}$ will have size at most $15k+9$.

\subsection{$D$-avoiding tulips.} 
We follow the approach of constructing a subgraph of maximum degree $3$ used in proving the Erd\H{o}s-P\'osa property for various types of cycles: roughly speaking, if there is a  cycle after removing the vertices of degree $3$ in the subgraph constructed so far, we augment the construction by adding some path or cycle.
Simonovitz~\cite{Simonovits1967} first proposed this idea and proved that if the number of degree $3$ vertices is $s_{k+1}$, then there are $k+1$ vertex-disjoint cycles. 
If a maximal construction has less than $s_{k+1}$ vertices of degree 3, then we can hit all cycles of the input graph by taking all vertices of degree $3$ and a few more vertices.
\medskip 

\noindent [Subsection~\ref{subsec:tuliphive}.] The major obstacle for  employing Simonovitz' approach is that for our purpose, we need to guarantee that every cycle of such a construction gives a hole.
For this reason, we will carefully add a path so that every cycle in a construction contains some hole.
We arrive at a notion of an \emph{$F$-extension}, which is a path to be added iteratively with $C$ at the beginning.
By adding $F$-extensions recursively, we will construct a subgraph such that all vertices have degree $2$ or $3$ and it contains $C$.
For a subgraph $F$ of $G_{deldom}$ such that all vertices have degree $2$ or $3$ and it contains $C$, 
an $F$-extension is a proper $V(F)$-path $P$ in $G_{deldom}$, but has additional properties that 
\begin{enumerate}[(i)]
\item both endpoints of $P$ are vertices of degree $2$ in $F$,
\item one of its endpoints should be in $C$, and 
\item $P$ has at least one endpoint in $C$ whose  second neighbor on $P$ has no neighbors in $F$, 
and the path obtained from $P$ by removing its endpoints is induced.
\end{enumerate}
An almost $F$-extension is a similar object, but its endpoints on $F$ are the same. Note that an almost $F$-extension is a cycle and is not an $F$-extension.
We depict an (almost) $F$-extension in Figure~\ref{fig:wextension}.  
When we recursively choose an $F$-extension to add, we apply two priority rules:
\begin{itemize}
\item We choose a shortest $F$-extension among all possible $F$-extensions. 
\item We choose an $F$-extension $Q$ with maximum $\abs{V(Q)\cap V(C)}$ among all shortest $F$-extensions. 
\end{itemize}
Following these rules, we recursively add $F$-extensions until there are no $F$-extensions. 

Let $W$ be the resulting graph.
The properties (ii), (iii) together with Lemma~\ref{lem:twopaths} guarantee that the subgraph induced by the vertex set of every cycle of $W$ contains a hole. Therefore, in case when $W$ contains $s_{k+1}$ vertices of degree $3$, we can find $k+1$ vertex-disjoint holes by Theorem~\ref{thm:simonovitz}. 
We may assume that it contains less than $s_{k+1}$ vertices of degree $3$. 
Let $T_{branch}$ be the set of degree $3$ vertices in $W$.
We also separately argue that we can hit all of almost $W$-extensions by at most $5k+4$ vertices.
Let $T_{almost}$ be the hitting set for almost $W$-extensions.

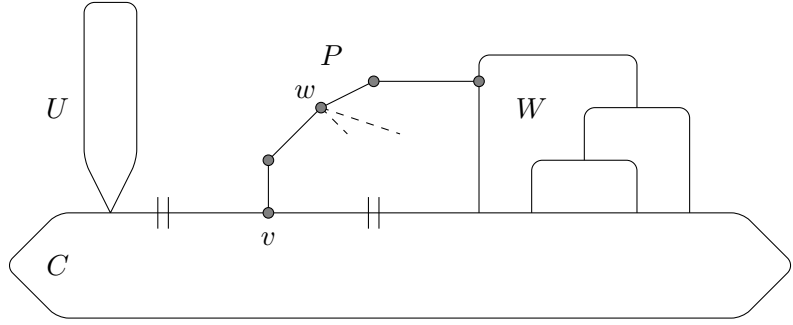
\begin{figure}
  \centering
  \begin{tikzpicture}[scale=0.7]
  \tikzstyle{w}=[circle,draw,fill=black!50,inner sep=0pt,minimum width=4pt]

   \draw[rounded corners] (6,0)--(11,0)--(12,-1)--(11,-2)--(-2,-2)--(-3,-1)--(-2,0)--(6,0);
   
   \draw[rounded corners] (7,0)--(7,1)--(9,1)--(9,0);
   \draw[rounded corners] (8,1)--(8,2)--(10,2)--(10,0);
   \draw[rounded corners] (6,0)--(6,3)--(9,3)--(9,2);

 \foreach \y in {3.5,4.5}{
	\draw[dashed] (3,2)--(\y, 1.5);
    }
   
   \draw (-.1,0.3)--(-.1,-0.3);
   \draw (.1,0.3)--(.1,-0.3);
   \draw (4-.1,0.3)--(4-.1,-0.3);
   \draw (4.1,0.3)--(4.1,-0.3);
   
   \draw (2,0)--(2,1)--(3,2)--(4,2.5);
   \draw (4,2.5)--(6,2.5);
   \draw (2,0) node [w] {};
   \draw (2,1) node [w] {};
   \draw (3,2) node [w] {};
   \draw (4,2.5) node [w] {};
   \draw (6,2.5) node [w] {};
   
     \node at (7, 2) {$W$};
  \node at (-2, 2) {$U$};

    \node at (2, -.5) {$v$};
    \node at (2.7, 2.3) {$w$};
 
	\draw[rounded corners] (-1,0)--(-1.5, 1)--(-1.5, 4)--(-.5,4)--(-.5,1)--(-1,0);
     \node at (3.2, 3) {$P$};

     \node at (-2, -1) {$C$};

   \end{tikzpicture}     \caption{A brief description of the construction $W$. 
   Each extension contains at least one endpoint in $C$ whose second neighbor in the extension 
   has no neighbor in $W$ hitherto constructed.
   For instance, $P$ is a $W$-extension, and $v$ is the vertex in $V(C)\cap V(P)$, and its second neighbor $w$ in $P$ has no neighbors in $W$.
   The subgraph $U$ depicts an almost $W$-extension.}\label{fig:wextension}
\end{figure}

Now, let $T_{ext}$ be the union of 
\[T_{petal}\cup T_{full}\cup T_{trav:sunf}\cup T_{branch}\cup T_{almost}\] and 
\[N^{20}_C[(T_{petal}\cup T_{full}\cup T_{trav:sunf}\cup T_{branch}\cup T_{almost})\cap V(C)].\] 
Note that 
\[ \abs{T_{ext}}\le 41(19k+(3k+14)+(15k+9)+(s_{k+1}-1)+(5k+4))\le 41(s_{k+1}+42k+26). \]

Furthermore, $C-(T_{petal}\cup T_{full}\cup T_{trav:sunf}\cup T_{branch}\cup T_{almost})$ contains at most $s_{k+1}+42k+26$ connected components, and thus $C-T_{ext}$ does as well.
\medskip

\noindent [Subsection~\ref{subsec:cfragment}] We discuss the patterns of the remaining tulips in $G_{deldom}-T_{ext}$.
Since we will add all components of $C-T_{ext}$ having at most $35$ vertices to the deletion set for $D$-avoiding tulips, 
we focus on components of $C-T_{ext}$ with at least $36$ vertices.
Let $H$ be a $D$-avoiding tulip in $G_{deldom}-T_{ext}$.
Let $Q=q_1q_2 \cdots q_m$ be a connected component of $C-T_{ext}$, and we consider the $Q$-tunnel $R$.

We argue that there is no edge $vw$ in $H$ such that 
$v$ is in the $Q$-tunnel, and $w$ is not in $N[C]$. 
See Figure~\ref{fig:qtunnel} for an illustration.
Suppose there is such a pair, and let $q\in V(Q)$ be a neighbor of $v$. We mainly prove that since $q$ is sufficiently far from degree $3$ vertices of $W$ in $C$, $w$ has no neighbors in $W$ 
(this is why we take the $20$-neighborhood of $V(C)\cap (T_{petal}\cup T_{full}\cup T_{trav:sunf}\cup T_{branch}\cup T_{almost})$ in $C$).
This is formulated in Lemma~\ref{lem:distance2}. 
Note that $qvw$ is a path where $q\in V(C)$ and $w$ has no neighbors in $W$, and furthermore, $H$ contains a vertex in $V(C)\setminus T_{ext}$ which is a vertex of degree $2$ in $W$.
Therefore, if we traverse in $H$ following the direction from $v$ to $w$, we meet some vertex having a neighbor which is a vertex of degree $2$ in $W$. This gives a $W$-extension or an almost $W$-extension. But it is a contradiction as there is no $W$-extension, and $T_{ext}$ hits all of almost $W$-extensions.
So, there are no such edges $vw$.

\begin{figure}
  \centering
  \begin{tikzpicture}[scale=0.7]
  \tikzstyle{w}=[circle,draw,fill=black!50,inner sep=0pt,minimum width=4pt]

   \draw[rounded corners] (6,0)--(11,0)--(12,-1)--(11,-2)--(-2,-2)--(-3,-1)--(-2,0)--(6,0);
   
   \draw[rounded corners] (7,0)--(7,1)--(9,1)--(9,0);
   \draw[rounded corners] (8,1)--(8,2)--(10,2)--(10,0);
   \draw[rounded corners] (6,0)--(6,3)--(9,3)--(9,2);
   \draw[rounded corners] (6,2.5)--(-1,2.5)--(-1,0);

	\draw[fill=blue!20, rounded corners] (-2,0)--(-2,.3)--(1,.3)--(1,-.3)--(-2,-.3)--(-2, 0);
	\draw[fill=blue!20, rounded corners] (4,0)--(4,.3)--(7,.3)--(7,-.3)--(4,-.3)--(4, 0);
	
	\draw[red, dotted, very thick, rounded corners] (1,0)--(1,1)--(4,1)--(4,0);
	
	\draw[rounded corners] (0, .5)--(2,.5)--(2.3, -.2)--(2.6,.5)--(3.5,.5)--(3.5,1.5)--(4, 2);
	   \draw (3.45,.5) node [w] {};
      \draw (3.5,1.5) node [w] {};
      \node at (3.45, 0) {$v$};
      \node at (3, 1.5) {$w$};

     \node at (7, 2) {$W$};

     \node at (-2, -1) {$C$};
 \node at (2.5, -1) {$Q$};
 \node at (5.5, -1) {$T_{ext}$};

\draw[arrows=<->](1, -.5)--(4, -.5);

   \end{tikzpicture} 
   \caption{A description of the set $T_{ext}$ and a component $Q$ of $C-T_{ext}$. 
   When there is an edge $vw$ where $v\in Z_{V(Q)}\setminus V(Q)$ and $w\notin N[C]$, 
   we will prove in Lemma~\ref{lem:tunnellemma1} that  $w$ has no neighbors in $W$.
   In particular, if there is a $D$-avoiding tulip containing such an edge, then we can find a $W$-extension or an almost $W$-extension starting with $qvw$ for some $q\in V(Q)$.
   }\label{fig:qtunnel}
\end{figure}
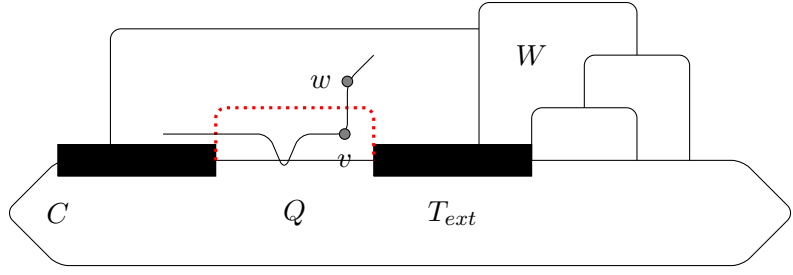

This argument leads to an observation that 
if $H$ contains some vertex $q_i$ with $6\le i\le m-5$, then the restriction of $H$ on $R$ should be a path from $Z_{q_1}$ to $Z_{q_m}$. Let $P$ be the restriction of $H$ on $R$. 
We additionally remove $\{q_j:1\le j\le 15, m-14\le j\le m\}$ and assume $H$ is not removed.
Then there are two ways that $P$ can be placed inside the $Q$-tunnel $R$: either
the endpoints of $P$ are in the same connected component of $R-(T_{ext}\cup V(Q))$ or not.
In the former case, we could reroute this path so that this part does not contain a vertex of $Q$.
So, we could obtain a $D$-avoiding tulip containing less vertices of $C$.
However, since $G-V(C)$ is chordal, 
there should be some subpath $Q'$ of $C-T_{ext}$ such that 
the restriction of $H$ on the $Q'$-tunnel is of the second type.
We will show that such a path can be hit by removing $5$ more vertices in $Q'$.
This will give a vertex set $T_{avoid:tulip}$ of size at most $35(s_{k+1}+42k+26)$  hitting all the remaining $D$-avoiding tulips.

\subsection{$D$-traversing tulips}
[Subsection~\ref{subsec:Dtulip}.]
This case can be handled similarly as the case of $D$-traversing sunflowers. 
It turns out that $T_{ext}$ hits every $D$-traversing tulip that contains precisely two vertices of $C$. 
Using a  matching argument between $D$ and the set of three consecutive vertices of $C$, we show that an additional set $T_{trav:tulip}$ of size at most $25k+9$ 
plus $T_{avoid:tulip}$ 
hits all the remaining $D$-traversing tulips unless $G$ contains $k+1$ vertex-disjoint holes.

In total, we can output in polynomial time either $k+1$ vertex-disjoint holes in $G$, or a vertex set of size at most 
\begin{align*}
&\abs{T_{ext}\cup T_{avoid:tulip}\cup T_{trav:tulip}} \\
&\le 41(s_{k+1}+42k+26)+ 35(s_{k+1}+42k+26)+25k+9 \\
										&\le 76(s_{k+1}+42k+26)+25k+9= 76s_{k+1}+3217k+1985
										\end{align*}
										 hitting all holes.

\section{Structural properties of $G$}\label{sec:lemmas}

In this section, we present structural properties of a graph $G$ with a shortest hole $C$.
In Subsection~\ref{subsec:distance}, we derive a relationship between the distance between $Z_v$ and $Z_w$ in $G_{deldom}$ for two vertices $v, w\in V(C)$ and the distance between $v$ and $w$ in $C$.
Briefly, we show that the distance between $Z_v$ and $Z_w$ in $G_{deldom}$ is at least some constant times the distance between $v$ and $w$ in $C$. 
We also prove that every connected subgraph in $G_{nbd}$ has a connected support.
In Subsection~\ref{subsec:dominating}, we obtain some basic properties of $C$-dominating vertices.
Recall that we assume that the length of $C$ exceeds $\mu_k$.

\subsection{Distance lemmas}\label{subsec:distance}

The following lemma classifies vertices in $N(C)$ with respect to the number of neighbors in $C$.

\begin{LEM}\label{lem:consecutive}
For every vertex $v$ of $N(C)$, either it has at most $3$ neighbors in $C$ and these vertices are consecutive in $C$, or it is $C$-dominating. 
\end{LEM}
\begin{proof}
Let us  write $N_i:=\{v\in N(C):\abs{N(v)\cap V(C)}=i\}$ for $i\geq 1$.
We first show that $N(C)=N_1\uplus N_2\uplus N_3\uplus D$.
Let $v\in N(C)\setminus D$.

We claim that $v$ has no two neighbors $w_1$ and $w_2$ in $C$ 
such that 
\begin{itemize}
\item[($\ast$)] there is a $(w_1, w_2)$-subpath $Q$ of $C$ where $Q$ has length at least $2$ and at most $\abs{V(C)}-3$ and $v$ has no neighbor in the internal vertices of $Q$.
\end{itemize}
If there is such a path $Q$, then $w_1vw_2\odot w_2Qw_1$ is a hole of length at most $\abs{V(C)}-1<\abs{V(C)}$, which contradicts the assumption that $C$ is a shortest hole. So the claim holds.

This implies that $v$ has no neighbors $z_1$ and $z_2$ with $\dist_C(z_1, z_2)\ge 3$. Indeed, if such neighbors exist, then let $Q$ be a $(z_1,z_2)$-subpath of $C$ containing at least one internal vertex non-adjacent to $v$. Since $v\notin D$, such $Q$ exists. Note that the length of $Q$ is at least 3 and at most $\abs{V(C)}-3$. Then there exist two neighbors $w_1$ and $w_2$ of $v$ in $V(Q)$ satisfying $(\ast)$, a contradiction. 
Therefore, the neighbors of $v$ in $C$ are contained in three consecutive vertices of $C$. 
This implies that $v\in N_1\uplus N_2\uplus N_3$.

Furthermore, if $v\in N(C)\setminus D$ has exactly two neighbors with distance $2$ in $C$, then $G$ contains a hole of length $4$, a contradiction. 
Therefore, such a vertex has at most $3$ neighbors in $C$ that are consecutive in $C$, as required.
\end{proof}

The next lemma is illustrated in Figure~\ref{fig:distance}.
\begin{LEM}\label{lem:farnonadj}
Let $x$ and $y$ be two vertices in $C$ such that $\dist_C(x,y)\geq 4$. Then there is no edge between $Z_x$ and $Z_y$. 
\end{LEM}
\begin{proof}
By Lemma~\ref{lem:consecutive}, the neighbors of any vertex in $N(C)\setminus D$ lie within distance at most $2$, and thus $x$ has no neighbors in $Z_y$ and $y$ has no neighbors in $Z_x$. 
Suppose $v\in Z_x\setminus \{x\}$ and $w\in Z_y\setminus \{y\}$ are adjacent. 
Let $P$ and $Q$ be $(x,y)$-subpaths of $C$ such that the length of $P$ is not greater than the length of $Q$. 
Since $C$ has length at least $9$, we may assume that $Q$ has length at least $5$.

Since the length of $Q$ is at least $5$, $N(v)\cap V(Q)$ is included in $N_C^2(x)\cap V(Q)$ and likewise we have 
$N(w)\cap V(Q)\subseteq N_C^2(y)\cap V(Q)$. Thus, $Q$ does not contain a common neighbor of $v$ and $w$. 

Since the length of $Q$ is at most $\abs{V(C)}-4$, a shortest $(v,w)$-path $Q'$ in $G[\{v,w\}\cup V(Q)]-vw$ has length at most $\abs{V(C)}-2$. 
Moreover, $Q'$ has length at least three due to the above assumption. 
Therefore, $vQ'w\odot wv$ is a hole strictly shorter than $C$, a contradiction.
\end{proof}

We prove a generalization of Lemma~\ref{lem:farnonadj}.
\begin{LEM}\label{lem:generalfarnonadj}
Let $m$ be a positive integer, and let $P$ be a $V(C)$-path in $G_{deldom}$ with endpoints $x$ and $y$.
If $P$ has length at most $m+2$, then $\dist_C(x,y)\leq  4m-1$.
\end{LEM}
\begin{proof}
We prove by induction on $m$. Lemma~\ref{lem:farnonadj} settles the case when $m=1$.
Let us assume $m\ge 2$.
Let $P=p_1p_2 \cdots p_n$ be a $V(C)$-path of length at most $m+2$ from $p_1=x$ and $p_n=y$ such that 
all of $p_2, \ldots, p_{n-1}$ are contained in $V(G_{deldom})\setminus V(C)$, and suppose that  $\dist_C(x,y)\geq 4m$.
For $\dist_C(x,y)\ge 4m\ge 4$, Lemma~\ref{lem:farnonadj} implies that $p_2$ is not adjacent to $p_{n-1}$.
Therefore, $P$ contains at least $5$ vertices.
We distinguish cases depending on whether $\{p_3, \ldots, p_{n-2}\}$ contains a vertex in $N(C)$ or not.

\medskip
\noindent {\bf Case 1.} $\{p_3, \ldots, p_{n-2}\}$ contains a vertex in $N(C)$. \\
We choose an integer $i\in \{3, \ldots, n-2\}$ such that $p_i\in N(C)$, and
choose a neighbor $z$ of $p_i$ in $C$.
Since there is a $V(C)$-path from $p_1$ to $z$ of length $i$, by induction hypothesis, $\dist_C(p_1, z)< 4(i-2)$.
By the same reason, we have $\dist_C(z, p_n)<4(n-i+1-2)=4(n-i-1)$.
Therefore, we have 
\[\dist_C(p_1, p_n)\le \dist_C(p_1, z)+\dist_C(z, p_n)<4(n-3)\le 4m,\]
a contradiction.

\medskip
\noindent {\bf Case 2.} $\{p_3, \ldots, p_{n-2}\}$ contains no vertices in $N(C)$.
\\
Let $Q$ be a shortest path from $N(p_2)\cap V(C)$ to $N(p_{n-1})\cap V(C)$ in $C$, and let $q, q'$ be its endpoints.
Observe that $p_2Pp_{n-1}$ and $p_2q\odot qQq' \odot q'p_{n-1}$ are two paths from $p_2$ to $p_{n-1}$ where there are no edges between their internal vertices.
Therefore, $G[V(Q)\cup (V(P)\setminus \{p_1, p_n\})]$ is a hole.

Since $\dist_C(x,y)\le \frac{\abs{V(C)}}{2}$, we have $m\le  \frac{\abs{V(C)}}{8}$. Therefore, the hole
$G[V(Q)\cup (V(P)\setminus \{p_1, p_n\})]$ has length at most 
\[ \abs{V(Q)}+\abs{V(P)}\le \frac{\abs{V(C)}}{2}+  \frac{\abs{V(C)}}{8} +1 < \abs{V(C)}.\]
This contradicts the assumption that $C$ is a shortest hole of $G$.

\medskip
This concludes the proof.
\end{proof}

Next, we show that every connected subgraph in $G_{nbd}$ has a connected support.
The following observation is useful.

\begin{LEM}\label{lem:threeconsecutive}
Let $a,b$ be vertices of $C$ with $\dist_C(a,b)\in \{2,3\}$ and 
let $S$ be the set of internal vertices of the shortest $(a,b)$-path of $C$.
Then there is no edge between $Z_a\setminus Z_S$ and $Z_b\setminus Z_S$.
\end{LEM}
\begin{proof}
Suppose there is an edge between $x\in Z_a\setminus Z_S$ and $y\in Z_b\setminus Z_S$.
If $x$ is adjacent to $b$, then $x\neq a$, and by Lemma~\ref{lem:consecutive}, $x$ has a neighbor in $S$, contradicting the assumption that $x\notin Z_S$.
Therefore, $x$ is not adjacent to $b$.
For the same reason, $y$ is not adjacent to $a$.
Therefore, the distance between $N(x)\cap V(C)$ and $N(y)\cap V(C)$ in $C$ is $2$ or $3$, 
and the vertex set of the shortest path from $N(x)\cap V(C)$ to $N(y)\cap V(C)$ in $C$ with $\{x,y\}$ induces a hole of length $5$ or $6$. 
This contradicts the assumption that $C$ is a shortest hole in $G$ and it has length greater than $6$.
\end{proof}

\begin{LEM}\label{lem:connectedsupport}
Let $H$ be a connected subgraph in $G_{nbd}$.
Then $C[\spp(H)]$ is connected.
\end{LEM}
\begin{proof}
Suppose $\spp(H)$ is not connected.
Then $H$ contains an edge $xy$ such that $\spp(G[\{x\}])$ and $\spp(G[\{y\}])$ are contained in distinct components of $C[\spp(H)]$.
Notice that $\spp(G[\{x\}])=N(x)\cap V(C)$ and $\spp(G[\{y\}])=N(y)\cap V(C)$, and it follows that $x,y\notin V(C)$ by Lemma~\ref{lem:consecutive}. 
We choose $a\in \spp(G[\{x\}])$ and $b\in \spp(G[\{y\}])$ with minimum $\dist_C(a,b)$.
By Lemma~\ref{lem:farnonadj}, we have $\dist_C(a,b)\le 3$, and since 
$\spp(G[\{x\}])$ and $\spp(G[\{y\}])$ are contained in distinct components of $C[\spp(H)]$, 
we have $\dist_C(a,b)\ge 2$.
Let $S$ be the set of internal vertices of the shortest $(a,b)$-path in $C$.
By the choice, $x\in Z_a\setminus Z_S$ and $y\in Z_b\setminus Z_S$.
Then by Lemma~\ref{lem:threeconsecutive}, 
there is no edge between $Z_a\setminus Z_S$ and $Z_b\setminus Z_S$.
This contradicts the assumption that $x$ is adjacent to $y$.
\end{proof}

The following lemma provides a structure of a $(Z_x, Z_y)$-path in $G_{nbd}$ for two vertices $x,y\in V(C)$.

\begin{LEM}\label{lem:pathsupport}
Let $x,y$ be two distinct vertices in $C$ and $P_1$, $P_2$ be two $(x,y)$-paths in $C$.
Let $Q$ be a $(Z_x, Z_y)$-path in $G_{nbd}$ such that $\spp(Q)\neq V(C)$.
Then either $Q$ is contained in $Z_{V(P_1)}$ or $Z_{V(P_2)}$.
\end{LEM}
\begin{proof}
By Lemma~\ref{lem:connectedsupport}, 
$\spp(Q)$ is connected, and since $\spp(Q)\neq V(C)$, $\spp(Q)$ contains either $V(P_1)$ or $V(P_2)$.
Without loss of generality, we assume that $\spp(Q)$ contains $V(P_1)$.
By the definition of a $(Z_x, Z_y)$-path, $Q$ contains no vertex of $Z_{\{x,y\}}$ as an internal vertex.
Let $s$ and $t$ be the two endpoints of $Q$ contained in $Z_x$ and $Z_y$, respectively.

We claim that $Q$ contains no vertex of $V(G_{nbd})\setminus Z_{V(P_1)}$, which immediately implies the statement.
Suppose for contradiction that $Q$ contains a vertex $v\in V(G_{nbd})\setminus Z_{V(P_1)}$. Clearly, we have $v\neq s$ and $v\neq t$.
Let $u$ be a vertex in $C$ such that $v\in Z_u$. Observe that $u\neq x$ and $u\neq y$, as
$Q$ contains no vertex of $Z_{\{x,y\}}$ as an internal vertex.
Let $Q_s$ and $Q_t$ be the $(s,v)$- and $(t,v)$-subpath of $Q$, respectively.

By Lemma~\ref{lem:connectedsupport}, $\spp(Q_s)$ contains an $(x,u)$-subpath of $C$.
Assume $\spp(Q_s)$ contains the $(x,u)$-subpath of $C$ containing $y$.
This means that $Q_s$ contains a vertex of $Z_y$, other than $t$, contradicting the fact that $Q$ contains no vertex of $Z_{\{x,y\}}$ as an internal vertex.
Therefore, 
$\spp(Q_s)$ contains the vertex set of the $(x,u)$-subpath of $C$ avoiding $y$. 
Similarly, $\spp(Q_t)$ contains the vertex set of the $(y,u)$-subpath of $C$ avoiding $x$. 
Now, observe that $\spp(Q)=\spp(Q_s)\cup \spp(Q_t) \supseteq V(P_2)$ and also by assumption, we have $V(P_1)\subseteq \spp(Q)$. 
Consequently, we have $\spp(Q)=V(C)$, a contradiction. This completes the proof.
\end{proof}

The following lemma is useful to find a hole with a small support.
\begin{LEM}\label{lem:overlay}
Let $P$ and $Q$ be two vertex-disjoint induced paths of $G_{nbd}$ such that
\begin{itemize}
\item there are no edges between $V(P)$ and $V(Q)$, and 
\item $\spp(P)\neq V(C)$ and $\spp(Q)\neq V(C)$.
\end{itemize}
If $\abs{\spp(P)\cap \spp(Q)}\geq 3$ and $x,y,z\in \spp(P)\cap \spp(Q)$ are three consecutive vertices on $C$, then $Z_{\{x,y,z\}}$ contains a hole. 
\end{LEM}
\begin{proof}
Since $x,y,z\in \spp(P)\cap \spp(Q)$ and there is no edge between $V(P)$ and $V(Q)$, $P$ and $Q$ contains no vertex of $\{x,y,z\}$.
Let $P'$ be a shortest $(Z_x,Z_z)$-subpath of $P$, and 
let $Q'$ be a shortest $(Z_x,Z_z)$-subpath of $Q$.
As $\spp(P)\neq V(C)$ and $\spp(Q)\neq V(C)$, 
$V(P')$ and $V(Q')$ are contained in $Z_{\{x,y,z\}}$ by Lemma~\ref{lem:pathsupport}.
By the preconditions, $P'$ and $Q'$ are vertex-disjoint and there are no edges between $P'$ and $Q'$. 
Thus, by Lemma~\ref{lem:twopaths}, $G[V(P')\cup V(Q')\cup \{x,z\}]$ contains a hole, which is in $Z_{\{x,y,z\}}$.
\end{proof}

\subsection{$C$-dominating vertices}\label{subsec:dominating}

We recall that $D$ is the set of $C$-dominating vertices.
We observe that $D$ is a clique because $G$ does not contain a hole of length 4.

\begin{LEM}\label{lem:dominating}
The set $D$ is  a clique. Furthermore, every hole contains at most one vertex of $D$.
\end{LEM}
\begin{proof}
Note that $G$ contains no hole of length $4$. This implies that any two vertices of $D$ are adjacent, which proves the first statement. To see the second statement, suppose that  $H$ is a hole containing two distinct vertices $u,v$ of $D$ and let $x\in V(H)\cap V(C)$ (there are no holes in $G-V(C)$). Then $\{x,u,v\}$ forms a triangle, contradicting the assumption that $H$ is a hole. 
\end{proof}

\begin{LEM}\label{lem:neighborofdominating}
If $H$ is a $D$-traversing hole, then it contains at most two vertices of $C$. Furthermore, every vertex of $V(H)\cap V(C)$ is adjacent to the unique $C$-dominating vertex on $H$. 
\end{LEM}
\begin{proof}
By Lemma~\ref{lem:dominating}, $H$ contains exactly one vertex of $D$, say $v$. 
If there is a vertex $x\in V(C)\cap V(H)$, then $x$ is adjacent to $v$ as $v$ is $C$-dominating. Therefore, any vertex of $V(C)\cap V(H)$ is adjacent to $v$ on $H$. Since $H$ is a cycle, $H$ contains at most two vertices of $C$.
\end{proof}

\section{Hitting all sunflowers}\label{sec:hittingsunflower}

In this section, we obtain a hitting set for sunflowers, unless $G$ contains $k+1$ vertex-disjoint holes.
Like in the previous section, we assume that $(G,k,C)$ is given as an input such that $C$ is a shortest hole of $G$ of length strictly greater than $\mu_k$, 
and $G-V(C)$ is chordal.

\subsection{Hitting all petals.} \label{subsec:petal}
We recall that a $D$-avoiding sunflower $H$ is a petal if $\abs{\spp(H)}\leq 7$.
By Lemma~\ref{lem:connectedsupport}, the vertices in the support of a petal is consecutive in $C$.

\begin{LEM}\label{lem:petalcover}
There is a polynomial-time algorithm which finds either $k+1$ vertex-disjoint holes in $G$ or a vertex set $T_{petal}\subseteq V(C)$ of at most $19k$ vertices such that 
\begin{itemize}
\item for every petal $H$, we have $\spp(H)\subseteq T_{petal}$.
\end{itemize}  
\end{LEM}
\begin{proof}
Set $X:=\emptyset$, $\mathcal{C}=\emptyset$, and $counter:=0$ at the beginning. We recursively do the following until the counter reaches $k+1$. For every set of nine consecutive vertices $v_0, v_1, v_2, \ldots, v_7, v_8$ of $C$ with $\{v_1, v_2, \ldots, v_7\}\cap X=\emptyset$, we test if $G[Z_{\{v_1, v_2, \ldots, v_7\}}\setminus Z_{\{v_0, v_8\}}]$ contains a hole $H$, and if so, add vertices in $\{v_1, v_2, \ldots, v_7\}$ to $X$, and add $H$ to $\mathcal{C}$ and increase the counter by 1. If the counter reaches $k+1$, then we stop. 
If the counter does not reach $k+1$, then
we have $\abs{X}\leq 7k$. In this case, we set $T_{petal}$ as the $6$-neighborhood of $X$ in $C$. So, $\abs{T_{petal}}\le (7+12)k=19k$.

By construction, any hole $H \in \mathcal{C}$ has a support that is fully contained in the considered set $\{v_1, v_2, \ldots, v_7\}$. 
Observe that we choose this set to be disjoint from $X$ constructed thus far.  
Therefore, holes in $\mathcal{C}$ are pairwise vertex-disjoint; otherwise, their supports have a common vertex. 
This implies that if the counter reaches $k+1$, then we can output $k+1$ vertex-disjoint holes.

Assume the counter does not reach $k+1$. In this case, we claim that for every petal $H$, $\spp(H)\subseteq T_{petal}$.
Let $H$ be a petal. By the definition of a petal and by Lemma~\ref{lem:connectedsupport}, there is a set of $7$ consecutive vertices $w_1, w_2, \ldots, w_7$ in 
$C$ such that $\spp(H)\subseteq \{w_1, w_2, \ldots, w_7\}$. If the set $\{w_1, w_2, \ldots, w_7\}$ is disjoint from $X$, then the above procedure must have considered this set and added it to $X$, a contradiction. 
Therefore, $\{w_1, w_2, \ldots, w_7\}\cap X\neq \emptyset$. Then during the step of adding 6-neighborhood of $X$ to $T_{petal}$, 
$\{w_1, w_2, \ldots, w_7\}$ is added to $T_{petal}$, and thus we have $\spp(H)\subseteq \{w_1, w_2, \ldots, w_7\} \subseteq T_{petal}$ as claimed.
\end{proof}

In what follows, we reserve $T_{petal}$ to denote a vertex subset of $V(C)$ that contains the support of every petal. 

\subsection{Polarization of $D$-avoiding sunflowers.}\label{subsec:allisfull}

We show that every $D$-avoiding sunflower in $G-T_{petal}$ is full. 
This will imply that, in order to hit every $D$-avoiding sunflower it is sufficient to find a hitting set for full sunflowers. We illustrate Lemma~\ref{lem:sppdichotomy} in Figure~\ref{fig:sppdichotomy}.

\begin{figure}
  \centering
  \begin{tikzpicture}[scale=0.7]
  \tikzstyle{w}=[circle,draw,fill=black!50,inner sep=0pt,minimum width=4pt]

   \draw (-2,0)--(5,0);\draw[thick, dotted](5,0)--(6,0);
   \draw (6,0)--(13,0);
	\draw(-2, 0)--(-3,-0.5);
	\draw(13, 0)--(14,-0.5);
      \draw (-3,-.5) node [w] {};
       \draw (14,-.5) node [w] {};
 	\draw[dashed](15, -1)--(14,-0.5);
	\draw[dashed](-4,-1)--(-3,-0.5);
 
 \foreach \y in {-1, 12}{
\draw[dashed, rounded corners] (\y, 3.3)--(\y+.5, 3.3)--(\y, -.5)--(\y-.5, 3.3)--(\y, 3.3);
}

     \node at (-1, 4) {$Z_{v_1}$};
     \node at (12, 4) {$Z_{v_{\ell}}$};
     \node at (-1, -.8) {$v_1$};
     \node at (2, -.8) {$v_4$};
     \node at (12, -.8) {$v_\ell$};

     \node at (-2, -1) {$C$};

     \node at (9, 3.6) {$P$};
     \node at (9, 2.2) {$Q$};

     \node at (-2, 2.4) {$x$};
     \node at (13, 2.4) {$y$};

	 \draw (-1,2.4)--(3, 3.5)--(8, 3.5)-- (12,2.4);
	 \draw (-1,2.4)--(-.5, 2.3) edge [bend right] (0, 0);
	 \draw (0,0) edge [bend right] (0.5, 2.1);
	 \draw (0.5, 2.1)--(3, 1.6);

	 \draw (3, 1.6)--(8, 1.6)-- (12,2.4);
       \draw (-1,2.4) node [w] () {};
     \draw (12,2.4) node [w] () {};
 
          \path[-] (2, 0) edge [bend right] (1.5, 1.9);
         \path[-] (2, 0) edge [bend right] (1.5, 3.1);

 \foreach \y in {-2,...,13}{
      \draw (\y,0) node [w] (a\y) {};
     }

   \end{tikzpicture}     \caption{A $D$-avoiding sunflower $H$ in $G-T_{petal}$ with $\spp (H)\neq V(C)$ in Lemma~\ref{lem:sppdichotomy}, which is not a petal. This hole $H$ has to intersect one of $\{v_1, v_2, v_3\}$ and $\{v_{\ell-2}, v_{\ell-1}, v_{\ell} \}$. If $H$ intersects $\{v_1, v_2, v_3\}$ as in the figure, then there is a petal with support contained in $\{v_1, \ldots, v_6\}$. The existence of the intersection on $\{v_1, v_2, v_3\}$ implies that $T_{petal}$ does not contain one of $v_1, v_2$, and $v_3$. This contradicts the fact that $T_{petal}$ contains the support of every petal. }\label{fig:sppdichotomy}
\end{figure}

\begin{LEM}\label{lem:sppdichotomy}
Every $D$-avoiding sunflower $H$ in $G-T_{petal}$ is full, that is, $\spp(H)=V(C)$. 
\end{LEM}
\begin{proof}
Suppose $H$ is a $D$-avoiding sunflower in $G-T_{petal}$ such that $8\le \abs{\spp(H)}< \abs{V(C)}$. 
By Lemma~\ref{lem:connectedsupport}, $C[\spp(H)]$ is a subpath of $C$. 
Let $C[\spp(H)]=v_1v_2 \cdots v_{\ell}$. Choose $x\in V(H)\cap Z_{v_1}$ and $y\in V(H)\cap Z_{v_{\ell}}$, and let $P$ and $Q$ be the two $(x,y)$-paths on $H$. 
As each of $C[\spp(P)]$ and $C[\spp(Q)]$ is connected by Lemma~\ref{lem:connectedsupport}, 
we have $\spp(P)=\spp(Q)= \spp(H)$.

Recall that $V(H)\cap V(C)\neq \emptyset$ and $V(H)\cap V(C)$ must be contained in $\spp(H)$. We argue that any $v_i\in \spp(H)$ with $i\in \{4, 5, \ldots, \ell-3\}$ does not lie on $H$. Suppose $v_i\in V(H)\cap V(C)$ for some $4\leq i\leq \ell-3$. Notice that both $x$ and $y$ are distinct from $v_i$. Therefore, $v_i$ belongs to exactly one of $P$ and $Q$. Without loss of generality, we assume $v_i\in V(P)$. Since $v_i\in \spp(Q)$, $Z_{v_i}\cap V(Q)\neq \emptyset$ and thus we can choose a vertex $v'_i$ from the set $Z_{v_i}\cap V(Q)$. Lemma~\ref{lem:consecutive} and $v'_i\notin D$ imply that $v'_i$ is not adjacent to $v_1$ or $v_{\ell}$, and thus $v'_i\notin Z_{v_1}$ and $v'_i\notin Z_{v_{\ell}}$. This means that $v'_i$ is distinct from $x$ and $y$, especially $v'_i$ is an internal vertex of $Q$. However, $v_iv'_i\in E(G)$ is a chord of $H$, a contradiction. 

Therefore, $v_i \notin V(H)$ for every $4\leq i\leq \ell-3$. At least one of $\{v_1,v_2,v_3\}$ and $\{v_{\ell-2},v_{\ell-1},v_{\ell}\}$ intersects with $V(H)$, and we assume that $\{v_1,v_2,v_3\}$ intersects with $V(H)$ without loss of generality (a symmetric argument works in the other case). From each of $P$ and $Q$, choose the first vertex (starting from $x$) that lies in $Z_{v_4}$ and call them $p\in V(P)$ and $q\in V(Q)$ respectively; the existence of such vertices follows from $v_4\in \spp(P)=\spp(Q)$. Let $H'$ be the cycle $pPx\odot xQq\odot qv_4p$.

Since $p$, $q$ are the first vertices contained in $Z_{v_4}$, 
$v_4$ has no neighbors in $(V(xPp)\cup V(xQq))\setminus \{p, q\}$, and thus $H'$ is a hole.
By Lemma~\ref{lem:consecutive}, we have $\spp(H')\subseteq \{v_1, v_2, \ldots, v_6\}$, that is, $H'$ is a petal. Because $T_{petal}$ contains the support of every petal, $\{v_1, v_2, v_3, v_4\}\subseteq \spp(H')\subseteq T_{petal}$.
On the other hand, $\{v_1,v_2,v_3\}\cap V(H)\neq \emptyset$ and $V(H)\cap T_{petal}=\emptyset$ implies  $\{v_1,v_2,v_3\}\setminus T_{petal}\neq \emptyset$. This is a contradiction. This completes the proof. 
\end{proof}

\subsection{Hitting all $D$-avoiding sunflowers.}\label{subsec:sunflowerwithout}
In this subsection we focus on full sunflowers.

\begin{PROP}\label{prop:hitsunflower}
There is a polynomial-time algorithm which finds either $k+1$ vertex-disjoint holes in $G$ or a vertex set $T_{full}\subseteq V(G)\setminus T_{petal}$ of at most $3k+14$ vertices such that $T_{petal}\cup T_{full}$ hits all full sunflowers.  
\end{PROP}

Our strategy is to find two collections of many vertex-disjoint paths so that  we can link the paths to obtain many vertex-disjoint holes.
The following lemma explains how to do this.
Note that since $C$ has length greater than $\mu_k$, $V(C)\setminus T_{petal}$ contains $25$ vertices that are consecutive in $C$.

\begin{LEM}\label{lem:connectingpaths}
Let $v_{-2}v_1v_0v_1 \cdots v_{20}v_{21}v_{22}$ be a subpath of $C$ that does not intersect $T_{petal}$, 
and $X$ be the $(v_0, v_{20})$-subpath of $C$ containing $v_1$ 
and $Y$ be the $(v_5, v_{15})$-subpath of $C$ containing $v_4$.
Let $\mathcal{P}$ be a collection of vertex-disjoint $(Z_{v_0},Z_{v_{20}})$-paths in $G[Z_{V(X)}]$, and  
let $\mathcal{Q}$ be a collection of vertex-disjoint $(Z_{v_5},Z_{v_{15}})$-paths in $G[Z_{V(Y)}]$. 
Given such $\mathcal{P}$ and $\mathcal{Q}$, if $\abs{\mathcal{P}}\ge k+13$ and $\abs{\mathcal{Q}}\ge  3k+15$, then one can output $k+1$ vertex-disjoint holes in polynomial time.
\end{LEM}
\begin{proof}
We begin with the observation that there is no petal whose support contains a vertex of $\{v_{-2}, v_{-1}, \ldots, v_{22}\}$. 
This is because $\{v_{-2}, v_{-1}, \ldots, v_{22}\}\cap T_{petal}=\emptyset$ by assumption and  
$T_{petal}$ contains the support of every petal of $G$.
We may assume that every path in $\mathcal{P}$ is induced, 
and similarly every path in $\mathcal{Q}$ is induced.

We take a subset $\mathcal{P}_1$ of $\mathcal{P}$ with $\abs{\mathcal{P}_1}=k+1$ that consists of paths containing no vertices of $\{v_0, v_1, \ldots, v_5\}\cup \{v_{15}, v_{16}, \ldots, v_{20}\}$. Such a collection $\mathcal{P}_1$ exists because the paths of $\mathcal{P}$ are vertex-disjoint and at most 12 of them 
intersect with $\{v_0, v_1, \ldots, v_5\}\cup \{v_{15}, v_{16}, \ldots, v_{20}\}$. 
Similarly we take a subset $\mathcal{Q}_1$ of $\mathcal{Q}$ with $\abs{\mathcal{Q}_1}=3k+3$ 
that consists of paths containing no vertices of $\{v_0, v_1, \ldots, v_5\} \cup \{v_{15}, v_{16}, \ldots, v_{20}\}$. 

For each $P\in \mathcal{P}_1$, let $\ell(P)$ and $r(P)$  be the endpoints of $P$ contained in $Z_{v_0}$ and  $Z_{v_{20}}$, respectively.
For each $Q\in \mathcal{Q}_1$, let $a(Q)$ be the vertex of $Z_{v_0}\cap V(Q)$ that is closest to $Z_{v_5}\cap V(Q)$ in $Q$, and 
let $b(Q)$ be the vertex in $Z_{v_{20}}\cap V(Q)$  that is closest to $Z_{v_{15}}\cap V(Q)$ in $Q$.
By definition, no internal vertex of the subpath of $Q$ from $a(Q)$ to the vertex in $V(Q)\cap Z_{v_5}$ contains a neighbor of $v_0$, and 
similarly no internal vertex of the subpath of $Q$ from $b(Q)$ to the vertex in $V(Q)\cap Z_{v_{15}}$ contains a neighbor of $v_{20}$.  See Figure~\ref{fig:pathsystems} for an illustration.

\begin{figure}
  \centering
  \begin{tikzpicture}[scale=0.7]
  \tikzstyle{w}=[circle,draw,fill=black!50,inner sep=0pt,minimum width=4pt]

   \draw (-2,0)--(5,0);\draw[thick, dotted](5,0)--(6,0);
   \draw (6,0)--(13,0);
	\draw(-2, 0)--(-3,-0.5);
	\draw(13, 0)--(14,-0.5);
      \draw (-3,-.5) node [w] {};
       \draw (14,-.5) node [w] {};
 	\draw[dashed](15, -1)--(14,-0.5);
	\draw[dashed](-4,-1)--(-3,-0.5);
 
 \foreach \y in {-2,...,13}{
      \draw (\y,0) node [w] (a\y) {};
     }

 \foreach \y in {-1, 4, 7, 12}{
\draw[dashed, rounded corners] (\y, 3.3)--(\y+.5, 3.3)--(\y, -.5)--(\y-.5, 3.3)--(\y, 3.3);
}

     \node at (-1, 4) {$Z_{v_0}$};
     \node at (4, 4) {$Z_{v_5}$};
     \node at (7, 4) {$Z_{v_{15}}$};
     \node at (12, 4) {$Z_{v_{20}}$};
     \node at (-1, -.8) {$v_0$};
     \node at (4, -.8) {$v_5$};
     \node at (7, -.8) {$v_{15}$};
     \node at (12, -.8) {$v_{20}$};

     \node at (-2, -1) {$C$};

     \node at (1.5, 2.8) {$P$};
     \node at (1.5, 1.6) {$Q$};

     \node at (-2, 2.4) {$\ell(P)$};
     \node at (13, 2.4) {$r(P)$};

	 \node at (-1.8, 0.8) {$a(Q)$};
     \node at (12.8, 0.8) {$b(Q)$};

	 \draw (-1,2.4)--(12,2.4);
     \draw[rounded corners] (4,1.2)--(-2,1.2)--(-3,0.7);\draw[dashed] (-3, 0.7)--(-4,0.2);
     \draw[rounded corners] (7,1.2)--(13,1.2)--(14,0.7);\draw[dashed] (14, 0.7)--(15,0.2);
     \draw (-1,2.4) node [w] () {};
     \draw (4,1.2) node [w] () {};
     \draw (7,1.2) node [w] () {};
     \draw (12,2.4) node [w] () {};
     \draw (-1,1.2) node [w] () {};
     \draw (12,1.2) node [w] () {};

   \end{tikzpicture}     \caption{Paths $P\in \mathcal{P}_1$ and $Q\in \mathcal{Q}_1$ in Lemma~\ref{lem:connectingpaths}. The vertices $\ell(P)$ and $a(Q)$ for $P\in \mathcal{P}_1$ and $Q\in \mathcal{Q}_1$ are adjacent, otherwise, we can find a hole $Z_{\{v_0, v_1, v_2\}}$, which is a petal. For the same reason, $r(P)$ is adjacent to $b(Q)$ for $P\in \mathcal{P}_1$ and $Q\in \mathcal{Q}_1$.}\label{fig:pathsystems}
\end{figure}

\begin{CLAIM}\label{claim:ef}
Let $w\in \{\ell(P):P\in \mathcal{P}_1\}$ and $z\in \{a(Q):Q\in \mathcal{Q}_1\}$. If $w\neq z$, then $wz\in E(G)$.
Let $w\in \{r(P):P\in \mathcal{P}_1\}$ and $z\in \{b(Q):Q\in \mathcal{Q}_1\}$. If $w\neq z$, then $wz\in E(G)$.
\end{CLAIM}
\begin{proofofclaim}
Suppose $w\in \{\ell(P):P\in \mathcal{P}_1\}$ and $z\in \{a(Q):Q\in \mathcal{Q}_1\}$ such that $w\neq z$ and they are not adjacent. 
Let $P_w\in \mathcal{P}_1$ and $Q_z\in \mathcal{Q}_1$ such that 
$\ell(P_w)=w$ and $a(Q_z)=z$. Let $Q_z'$ be the subpath of $Q_z$ from $z$ to the vertex in $Z_{v_5}$.
Note that $v_0$ has no neighbors in $V(P_w)\setminus \{w\}$ and $V(Q_z')\setminus \{z\}$.
In case when $P_w$ and $Q_z'$ meet somewhere in $\bigcup_{i\in \{1, 2\}}Z_{v_i}$, 
we obtain a hole contained in $Z_{\{v_0, v_1, v_2\}}$ by Lemma~\ref{lem:twopaths}.
When $P_w$ and $Q_z'$ do not meet in $\bigcup_{i\in \{1, 2\}}Z_{v_i}$, 
there is a hole contained in $Z_{\{v_0, v_1, v_2\}}$  by Lemma~\ref{lem:twopaths} since $v_2$ has a neighbor in both $P_w$ and $Q_z'$. 
In both cases, there is a petal with support contained in $\{v_{i}: -2\le i\le 4\}$, a contradiction. 
We conclude that $wz\in E(G)$.
The proof of the latter statement is symmetric.
\end{proofofclaim}

For every $P\in \mathcal{P}_1$, $\ell(P)$ is the unique vertex of $Z_{v_0}\cap V(P)$. Therefore, for fixed $P\in \mathcal{P}_1$,
there is at most one path $Q\in \mathcal{Q}_1$ such that $V(Q)\cap V(P)\cap Z_{v_0}\neq \emptyset$. 
Similarly, there is at most one path of $\mathcal{Q}_1$ intersecting with $P$ at a vertex of $Z_{v_{20}}$. 
We construct a new collection $\mathcal{Q}_2$ so that
\begin{quote}
for every $Q\in \mathcal{Q}_1$, $\mathcal{Q}_2$ contains the subpath $a(Q)Qb(Q)$ if and only if 
$Q$ does not intersect with any $P\in \mathcal{P}_1$ at a vertex of $Z_{v_0}\cup Z_{v_{20}}$.
\end{quote}
Observe that $\mathcal{Q}_2$ contains at least $k+1$ paths because each path of $\mathcal{P}_1$ can make  
at most two paths of $\mathcal{Q}_1$ drop out. For our purpose, taking precisely $k+1$ paths is sufficient. 
Let $\mathcal{P}_1=\{P_1,\ldots , P_{k+1}\}$ and  $\mathcal{Q}_2=\{Q_1, Q_2, \ldots, Q_{k+1}\}$.
For each $i\in \{1, 2, \ldots, k+1\}$, we create a cycle $C_i$ from the disjoint union of $P_i\in \mathcal{P}_1$ and $Q_i\in \mathcal{Q}_2$ 
by adding two edges $a(Q_i)\ell(P_i)$ and $b(Q_i)r(P_i)$. 
Such edges exist by Claim~\ref{claim:ef}.

We observe that each $C_i$ contains a hole. To see this, take $v\in Z_{v_{10}}\cap V(P_i)$. 
As $Q_i\in \mathcal{Q}_2$ is a path of  $G[Z_{V(Y)}]$, Lemma~\ref{lem:consecutive} implies that $v$ is not adjacent to any vertex of $Q_i$. 
Note that $v$ is an internal vertex of 
the induced path $P_i$. Therefore, $G[V(C_i)]$ contains a hole  by Lemma~\ref{lem:twopaths}. 

Lastly, we verify that two holes contained in distinct cycles of $\{C_i:1\le i\le k+1\}$ are vertex-disjoint. To prove this, it is sufficient to show that 
for two integers $a,b\in \{1, 2, \ldots, k+1\}$, no internal vertex of $P_a\in \mathcal{P}_1$ is an internal vertex of $Q_b\in \mathcal{Q}_2$. 
Suppose the contrary, that is, $w$ is an internal vertex of $P_a\in \mathcal{P}_1$ and $Q_b\in \mathcal{Q}_2$ simultaneously for some $a,b$.
Since $Q_b$ is a path of $G[Z_{Y}]$ and $w\notin Z_{v_0}\cup Z_{v_{20}}$ is an internal vertex of $P_a$, we have $w\in Z_{\{v_1, v_2, v_3, v_4,v_5\}}\cup Z_{\{v_{15}, v_{16}, v_{17}, v_{18},v_{19}\}}$. 
Without loss of generality, we assume $w\in Z_{\{v_1, v_2, v_3, v_4,v_5\}}$. 
In fact, $w$ cannot be in $Z_{v_5}$ since otherwise, the path $Q'_b\in \mathcal{Q}_1$ having $Q_b$ as a proper subpath contains a vertex of $Z_{v_5}$ as an internal vertex; 
violating the definition of $(Z_{v_5},Z_{v_{15}})$-path. 
Now observe that $Q_b$ contains, as a subpath, a $Z_{v_0}$-path $Q'$ having all internal vertices in $Z_{\{v_1, v_2, v_3, v_4\}}\setminus Z_{\{v_0, v_5\}}$. 
Let $x,y$ be the endpoints of $Q'$. Since $v_0$ is adjacent to $x,y$ but is not adjacent to any internal vertices of $Q'$,
$G[\{v_0\}\cup V(Q')]$ contains a hole by  Lemma~\ref{lem:twopaths}, a contradiction. A symmetric argument holds for the case 
$w\in Z_{\{v_{15}, v_{16}, v_{17}, v_{18},v_{19}\}}$. Therefore, $\{C_i:1\le i\le k+1\}$ 
is a collection of vertex-disjoint holes of $G$, which completes the proof.
\end{proof}

\begin{proof}[Proof of Proposition~\ref{prop:hitsunflower}]
Note that $\abs{T_{petal}}\le 19k$ and $\abs{V(C)}>\mu_k> 25\abs{T_{petal}}$ by assumption. Thus, there are $25$ consecutive vertices on $C$ having no vertices in $T_{petal}$.
We choose a subpath $v_{-2}v_{-1}v_0v_1 \cdots v_{20}v_{21}v_{22}$ of $C$ that contains no vertices in $T_{petal}$.
 Let $P_1$ be the $(v_0, v_{20})$-subpath of $C$ containing $v_1$ 
 and $P_2$ be the $(v_5, v_{15})$-subpath of $C$ containing $v_4$.

We apply Menger's Theorem for $(Z_{v_0},Z_{v_{20}})$-paths in $G[Z_{V(P_1)}]$,
and then for $(Z_{v_5},Z_{v_{15}})$-paths in $G[Z_{V(P_2)}]$.  
We have one of the following.
\begin{itemize}
\item The first application of Menger's Theorem outputs a vertex set $X$ with $\abs{X}\leq k+12$ hitting all $(Z_{v_0},Z_{v_{20}})$-paths in $G[Z_{V(P_1)}]$.
\item The second application of Menger's Theorem outputs a vertex set $X$ with $\abs{X}\leq 3k+14$ hitting all $(Z_{v_5},Z_{v_{15}})$-paths in $G[Z_{V(P_2)}]$.
\item The first algorithm outputs at least $k+13$ vertex-disjoint paths, and the second algorithm outputs at least $3k+15$ vertex-disjoint paths. 
\end{itemize} 

In the third case,  by Lemma~\ref{lem:connectingpaths}, we can construct $k+1$ vertex-disjoint holes in polynomial time.

Suppose we obtained a vertex set $X$ in the first case.  We claim that $T_{petal}\cup X$ hits all full sunflowers. Suppose that there is a full sunflower $H$ avoiding every vertex of $T_{petal}\cup X$. By definition, $\spp (H)=V(C)$. In particular, $H$ contains at least one vertex of $Z_{v_{10}}$, say $w$. 
Let $F$ be the connected component of the restriction of $H$ on $G[Z_{V(P_1)}]$ containing $w$. 
Clearly $F$ is a path and 
its endpoints are contained in $Z_{\{v_0, v_{20}\}}$ because of Lemma~\ref{lem:connectedsupport}.

Suppose that the endpoints of $F$ are contained in distinct sets of $Z_{v_0}$ and $Z_{v_{20}}$, respectively. 
Let $F'$ be a subpath of $F$ that is a $(Z_{v_0}, Z_{v_{20}})$-path. Note that $F'$ is a $(Z_{v_0}, Z_{v_{20}})$-path of $G[Z_{V(P_1)}]$ 
because $F$ is a path of $G[Z_{V(P_1)}]$.
But it contradicts the fact that $X$ hits all such paths. 

Suppose that both endpoints of $F$ are contained in one of $Z_{v_0}$ or $Z_{v_{20}}$, say $Z_{v_0}$.
Let $F_1$ and $F_2$ be the two subpaths of $F$ from $w$ to its endpoints.
Then by Lemma~\ref{lem:connectedsupport}, both $\spp(F_1)$ and $\spp(F_2)$ contain the $(v_0, v_{10})$-subpath of $P_1=v_0v_1 \cdots v_{20}$. 
This implies that $\spp(F_1-w)\cap \spp(F_2-w)$ contains  $\{v_0, v_1, v_2\}$.
Since there are no edges between $F_1-w$ and $F_2-w$, Lemma~\ref{lem:overlay} implies that 
there is a hole contained in  $Z_{\{v_0, v_1, v_2\}}$.
This is a contradiction because we assumed $\{v_{-2}, v_{-1}, \ldots, v_{22}\}\cap T_{petal}=\emptyset$ while 
$T_{petal}$ contains the support of every petal of $G$.
Therefore, $T_{petal}\cup X$ hits every full sunflower. The case when both endpoints of $F$ are contained on $Z_{v_{20}}$ 
follows from a symmetric argument.

The second case when we obtain the vertex set $X$ with $\abs{X}\leq 3k+14$ can be handled similarly. Hence, in the first or second case, we can output a required vertex set $T_{full}$ of size at most $3k+14$ hitting every full sunflower in polynomial time.
\end{proof}

\subsection{Hitting all $D$-traversing sunflowers.}\label{subsec:sunflowerwith}

Our proof builds on the observation that any $D$-traversing sunflower entails another $D$-traversing sunflower $H'$ where the support of the path $H'- D$ is `small'. Then we exploit the min-max duality of vertex cover and matching on bipartite graphs
in order to cover such $D$-traversing sunflowers with small support.

The following lemma describes how to obtain such a sunflower $H'$. 
We depict the setting of Lemma~\ref{lem:smallsunflower} in Figure~\ref{fig:traversingsunflower}.

\begin{LEM}\label{lem:smallsunflower}
Let $v_1v_2 \cdots v_5$ be a subpath of $C$ such that $\{v_1,\ldots , v_5\}\cap T_{petal}=\emptyset$ and let $P=p_1p_2 \cdots p_m$ be a path in $G[D\cup Z_{\{v_1, v_2, \ldots, v_5\}}]$ such that
\begin{enumerate}[(i)]
\item $p_1$ is a $C$-dominating vertex and $p_2=v_3$, 
\item $p_m\in Z_{\{v_1, v_5\}}\setminus \{v_1, v_5\}$, 
\item all internal vertices of $P$ are in $Z_{\{v_2, v_3, v_4\}}\setminus Z_{\{v_1, v_5\}}$, and 
\item $E(G[V(P)])\setminus E(P)\subseteq \{p_1p_m\}$; that is, $G[V(P)]$ is either an induced path $p_1p_2 \cdots p_m$ or an induced cycle $p_1p_2 \cdots p_mp_1$, 
\item if $G[V(P)]$ is an induced cycle, then $m\ge 4$.
\end{enumerate}
Then there exists a $D$-traversing sunflower $H$ containing $p_1$  and $p_2$ such that $V(H)\setminus \{p_1\}\subseteq  Z_{\{v_1, v_2, v_3\}} \cap (V(P)\cup \{v_1\})$
 or $ V(H)\setminus \{p_1\} \subseteq  Z_{\{v_3, v_4, v_5\}}\cap (V(P)\cup \{v_5\})$. 
\end{LEM}
\begin{proof}
We claim the following:
\begin{quote}
If $p_m\in Z_{v_1}$, then $P-p_1$ is contained in $Z_{\{v_1, v_2, v_3\}}$. Likewise, 
if $p_m\in Z_{v_5}$, then $P-p_1$ is contained in $Z_{\{v_3, v_4, v_5\}}$.
\end{quote} 

We only prove the first statement; the proof of the second statement will be symmetric.
Let us assume $p_m\in Z_{v_1}$. We observe that $P$ contains no vertex in $Z_{v_5}$ because 
all internal vertices of $P$ are in $Z_{\{v_2, v_3, v_4\}}\setminus Z_{\{v_1, v_5\}}$.

We first show that $P$ contains no vertex in $Z_{v_4}\setminus Z_{v_3}$.
Suppose the contrary and let $w\in V(P)\cap (Z_{v_4}\setminus Z_{v_3})$. 
Since both $C[\spp(p_2Pw)]$ and $C[\spp(wPp_m)]$ are connected by  Lemma~\ref{lem:connectedsupport}, 
$P$ contains a $Z_{v_3}$-subpath $P'$ whose internal vertices are all contained in $Z_{v_4}\setminus Z_{v_3}$.
Then $v_3$ is not adjacent to any internal vertex of $P'$, and by Lemma~\ref{lem:twopaths}, $G[V(P')\cup \{v_3\}]$ contains a hole, which is a petal.
This contradicts the fact that $v_3\notin T_{petal}$, because by the construction of $T_{petal}$ in Lemma~\ref{lem:petalcover}, $T_{petal}$ fully contains the support of every petal.
Hence, $P$ contains no vertex of $Z_{v_4}\setminus Z_{v_3}$ and we have $V(P)\setminus \{p_1\}\subseteq Z_{\{v_1, v_2, v_3\}}$.

We claim that there is a $D$-traversing sunflower as claimed. 
If $p_1p_m\in E(G)$, then $G[V(P)]$ is a hole as claimed by (iv) and due to the previous claim.
Hence, we may assume $p_1$ is not adjacent to $p_m$.
Observe that $v_1p_1p_2$ is an induced path.
Also, $p_2Pp_m\odot p_mv_1$ is a path from $v_1$ to $v_3$, and it does not contain $v_2$. Indeed, 
if the path $p_2Pp_m\odot p_mv_1$ contains $v_2$, then $\{p_1,p_2=v_3,v_2\}$ forms a triangle, a contradiction to (iv). 
Now, $p_1$ has no neighbors in $V(P)\setminus \{p_1, p_2\}$ as we assumed that $p_1$ is not adjacent to $p_m$.
Therefore, we can apply Lemma~\ref{lem:twopaths} with the induced path $v_1p_1p_2$ with $p_1$ as an internal vertex and 
the path $p_2Pp_m\odot p_mv_1$. It follows that there is a hole in $G[V(P)\cup \{v_1\}]$ containing $p_1$, $p_2$ 
such that $V(H)\setminus \{p_1\}\subseteq  Z_{\{v_1, v_2, v_3\}}$. The statement follows immediately.
\end{proof}

\begin{figure}
  \centering
  \begin{tikzpicture}[scale=0.7]
  \tikzstyle{w}=[circle,draw,fill=black!50,inner sep=0pt,minimum width=4pt]

   \draw (-2,0)--(11,0);
	\draw(-2, 0)--(-3,-0.5);
	\draw(11, 0)--(12,-0.5);
    \draw[dashed](13, -1)--(12,-0.5);
	\draw[dashed](-4,-1)--(-3,-0.5);

      \draw(5,3.5) [in=80,out=150] to (.5,0);
      \draw(.5,0)--(.5,1.5);
           \node at (0, 1.5) {$p_m$};
      
 	\draw (5, 3.5)--(4.5,0)--(4, 2)--(3,2.7)--(2, 2)--(1.5, 0.5)--(.5, 1.5);
        \draw (4,2) node [w] {};
        \draw (3,2.7) node [w] {};
        \draw (2,2) node [w] {};
        \draw (1.5,.5) node [w] {};
        \draw (.5,1.5) node [w] {};

 \node at (3, 1.5) {$P$};
 
 \foreach \y in {-1.5,0.5,  ..., 11}{
      \draw (\y,0) node [w] (a\y) {};
     }
     \node at (0.5, -.8) {$v_1$};
     \node at (2.5, -.8) {$v_2$};
     \node at (4.5, -.8) {$v_3=p_2$};
     \node at (6.5, -.8) {$v_4$};
     \node at (8.5, -.8) {$v_5$};

 \foreach \y in {0.5, 8.5}{
\draw[dashed, rounded corners] (\y, 3.3)--(\y+.5, 3.3)--(\y, -.5)--(\y-.5, 3.3)--(\y, 3.3);
}

     \node at (0.5, 4) {$Z_{v_1}$};
     \node at (8.5, 4) {$Z_{v_5}$};

\draw[rounded corners] (-3-.5+7,4)--(-3-.5+7,3)--(0-.5+7,3)--(0-.5+7,5)--(-3-.5+7,5)--(-3-.5+7,4);
       \draw (-2+7, 3.5) node [w] {};
     \node at (5.5, 3.5) {$p_1$};

     \node at (-1.5-.5+7, 4.5) {$D$};
     
  \end{tikzpicture}     \caption{Obtaining another sunflower in Lemma~\ref{lem:smallsunflower}.
  As $v_1p_1p_2$ is an induced path and $p_1$ has no neighbors in the set of internal vertices of $p_2Pp_m\odot p_mv_1$, $G[V(P)\cup \{v_1\}]$ contains a hole.}\label{fig:traversingsunflower}
\end{figure}
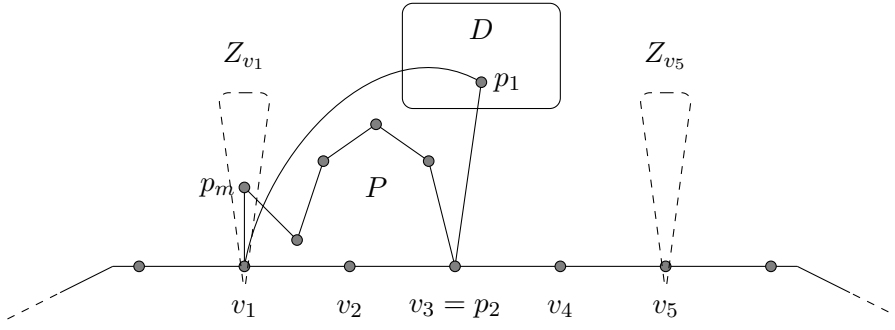

\begin{LEM}\label{lem:small}
Let $H$ be a $D$-traversing sunflower in $G-T_{petal}$ containing a $C$-dominating vertex $d$. Then there exist three consecutive vertices $x,y,z$ on $C$ and a $D$-traversing sunflower $H'$ containing $d$ such that $V(H)\cap \{x,y,z\}\neq \emptyset$ and $V(H')\setminus \{d\}\subseteq  Z_{\{x,y,z\}}$. 
\end{LEM}
\begin{proof}
By Lemma~\ref{lem:neighborofdominating}, $H$ contains at most $2$ vertices of $C$, and every vertex in $V(H)\cap V(C)$ is adjacent to the  
(unique) vertex of $V(H)\cap D$.
Let $P$ be the connected component of $H-(V(C)\cup \{d\})$.
By Lemma~\ref{lem:dominating}, every hole contains at most one vertex of $D$, and it implies that 
$P$ contains no vertices of $D$.
Let $a\in V(H)\cap V(C)$.
If the support of $P$ is contained in $N_C[a]$, then we are done as $\abs{N_C[a]}\le 3$ and we can take $H'=H$ and $\{x,y, z\}=N_C[a]$.
We may assume that the support of $P$ contains a vertex in $C$ whose distance to $a$ in $C$ is $2$ by Lemma~\ref{lem:connectedsupport}.
Let $v_1, v_2, v_3, v_4, v_5$ be the consecutive vertices of $C$ where $a=v_3$.
This assumption implies that $P$ contains a vertex in either $Z_{v_1}$ or $Z_{v_5}$.

Let $w$ be the vertex of $Z_{\{v_1, v_5\}}\cap V(H)$ that is closest to $a$ in $P$.
Let $Q$ be the $(d,w)$-subpath of $H$ containing $a$.
We verify the preconditions of Lemma~\ref{lem:smallsunflower} with $(p_1, p_2, P)=(d, a, Q)$.
The first condition is clear.
Note that $H$ contains neither $v_2$ nor $v_4$; otherwise, $dav_2$ or $dav_4$ is a triangle in $H$, contradicting the assumption that $H$ is a hole.
Thus $Q$ contains neither $v_2$ nor $v_4$.
If $w$ is $v_1$ or $v_5$, then the neighbor of $w$ in $Q$ is also in $Z_{\{v_1, v_5\}}$, contradicting the choice of $w$.
Thus, $w\in Z_{\{v_1, v_5\}}\setminus \{v_1, v_5\}$.
Clearly, all internal vertices of $Q$ are in $Z_{\{v_2, v_3, v_4\}}\setminus Z_{\{v_1, v_5\}}$; otherwise by Lemma~\ref{lem:connectedsupport}, 
$Q$ must contain an internal vertex from $Z_{\{v_1, v_5\}}$, contradicting the choice of $w$.
The last two conditions are satisfied because $H$ is a hole.
Then $(d,a,Q)$ meets the preconditions of Lemma~\ref{lem:smallsunflower}. 

Therefore, there exists a $D$-traversing sunflower $H'$ containing $d$ and $a$ such that $V(H')\setminus \{d\}\subseteq  Z_{\{v_1, v_2, v_3\}}$ or $V(H')\setminus \{d\}\subseteq  Z_{\{v_3, v_4, v_5\}}$.
As $a=v_3\in V(H)$, we have $V(H)\cap \{v_1, v_2, v_3\}\neq\emptyset$ or $V(H)\cap \{v_3, v_4, v_5\}\neq\emptyset$, respectively.
\end{proof}

Based on Lemma~\ref{lem:small}, we prove the following.

\begin{PROP}\label{prop:skew}
There is a polynomial-time algorithm which finds either $k+1$ vertex-disjoint holes in $G$ or a vertex set $T_{trav:sunf}\subseteq (D\cup V(C))\setminus T_{petal}$ of size at most $15k+9$ such that $T_{petal}\cup T_{trav:sunf}$ hits all $D$-traversing sunflowers.
\end{PROP}
\begin{proof}
Let $C=v_0v_1 \cdots v_{m-1}v_0$. All additions are taken modulo $m$. 
We create an auxiliary bipartite graph $\mathcal{G}_i=(D\uplus \mathcal{A}_i, \mathcal{E}_i)$ for each $0\leq i \leq 4$, such that
\begin{itemize}
\item $\mathcal{A}_i=\{\{v_{5j+i},v_{5j+i+1},v_{5j+i+2}\}: j=0, 1, \ldots ,\lfloor \frac{m}{5} \rfloor -1\}$,
\item there is an edge between $d\in D$ and $\{x,y,z\}\in \mathcal{A}_i$ if and only if 
there is a hole $H$ containing $d$ such that $V(H)\setminus \{d\}\subseteq Z_{\{x,y,z\}}$ (thus, $V(H)\cap \{x,y,z\}\neq \emptyset$).
\end{itemize}
Clearly, the auxiliary graph $\mathcal{G}_i$ can be constructed in polynomial time using Lemma~\ref{lem:detectinghole}. 

Now, we apply Theorem~\ref{thm:menger} to each $\mathcal{G}_i$ and output either a matching of size $k+1$ or a vertex cover of size at most $k$.

Suppose that there exists $i\in \{0,1,\ldots, 4\}$ such that $\mathcal{G}_i$ contains a matching $M$ of size at least $k+1$. 
We argue that there are $k+1$ vertex-disjoint holes in this case. Let $e=(d,\{x,y,z\})$ and $e'=(d',\{x',y',z'\})$ be two distinct edges of $M$. 
By construction, 
there exist two holes $H$ and $H'$ such that 
\begin{itemize}
\item $H$ contains $d$ and $V(H)\setminus \{d\} \subseteq Z_{\{x,y,z\}}$, and
\item $H'$ contains $d'$ and $V(H')\setminus \{d'\} \subseteq Z_{\{x',y',z'\}}$.
\end{itemize}
Recall that any vertex of $N(C)\setminus D$ has at most three neighbors on $C$, which are consecutive by Lemma~\ref{lem:consecutive}. On the other hand, the distance between $\{x,y,z\}$ and $\{x',y',z'\}$ on $C$ is at least three by the construction of the family $\mathcal{A}_i$. Therefore the two sets $Z_{\{x,y,z\}}$ and $Z_{\{x',y',z'\}}$ are disjoint, 
which implies  
$H$ and $H'$ are vertex-disjoint. We conclude that one can output $k+1$ vertex-disjoint holes when there is a matching $M$ of size $k+1$ in one of $\mathcal{G}_i$'s.

Consider the case when for every $0\le i\leq 4$, $\mathcal{G}_i$ admits a vertex cover $S_i$ of size at most $k$. For $S_i$, let $S^*_i$ be the vertex set 
\[ (S_i\cap D) \cup \bigcup_{\{x,y,z\}\in S_i\cap \mathcal{A}_i} \{x,y,z\} \]
and let $T_{trav:sunf}:=\left( \bigcup_{i=0}^4 S^*_i \right) \cup \{v_{5\lfloor \frac{m}{5} \rfloor+i}:-2\le i\le 6\}$. Notice that $\abs{T_{trav:sunf}}\leq 15k+9$. 

\begin{CLAIM}
The vertex set $T_{petal}\cup T_{trav:sunf}$ hits all $D$-traversing sunflowers. 
\end{CLAIM}
\begin{proofofclaim}
Suppose $G-(T_{petal}\cup T_{trav:sunf})$ contains a $D$-traversing sunflower $H$ having a vertex $d\in D$. 
By Lemma~\ref{lem:small}, 
there exist $x,y,z$ that are consecutive vertices on $C$ and a $D$-traversing sunflower $H'$ containing $d$ such that $V(H)\cap \{x,y,z\}\neq \emptyset$ and $V(H')\setminus \{d\}\subseteq  Z_{\{x,y,z\}}$. Clearly, we have either 
\begin{itemize}
\item $d$ is adjacent to $\{x,y,z\}$ in one of the bipartite graphs $\mathcal{G}_i$, or
\item $\{x,y,z\}\subseteq \{v_{5\lfloor \frac{m}{5} \rfloor+i}:-2\le i\le 6\}$.
\end{itemize}
In the first case, $S^*_i$ contains $\{d\}$ or $\{x,y,z\}$, as $S_i$ is a vertex cover of $\mathcal{G}_i$. Since $V(H)\cap \{x,y,z\}\neq \emptyset$ and $H$ contains $d$, $S^*_i$ contains a vertex of $H$, which contradicts the assumption that $H$ is a $D$-traversing sunflower in $G-(T_{petal}\cup T_{trav:sunf})$. In the second case, 
$T_{trav:sunf}$ contains $\{x,y,z\}$, which again contradicts that $H$ is a $D$-traversing sunflower in $G-(T_{petal}\cup T_{trav:sunf})$.
\end{proofofclaim}

This completes the proof. 
\end{proof}

\section{Hitting all tulips}\label{sec:tulip}

In this section, 
we show that one can find in polynomial time either $k+1$ vertex-disjoint holes or a vertex set hitting all tulips.
Again, we assume that $(G,k,C)$ is given as an input such that $C$ is a shortest hole of $G$ of length 
strictly greater than $\mu_k$ and $G-V(C)$ is chordal.

The first few sections will focus on $D$-avoiding tulips. In Subsection~\ref{subsec:Dtulip}, we settle the case of $D$-traversing tulips.
Subsection~\ref{subsec:final} will establish the main theorem for holes in general, Theorem~\ref{thm:core}.
For $D$-avoiding tulips, it is sufficient to consider the graph $G_{deldom}=G-D$.  

\subsection{Constructing a nested structure of partial tulips}\label{subsec:tuliphive}

We recursively construct a subgraph of $G_{deldom}$ in which all vertices have degree $2$ or $3$ and it contains $C$. 
A subgraph of $G$ is called a \emph{$(2,3)$-subgraph} if its all vertices have degree $2$ or $3$.
For a $(2,3)$-subgraph $F$, a vertex $v$ of degree $3$ in $F$ is called a \emph{branching point} in $F$,
 and other vertices are called  \emph{non-branching points}. 

Given a $(2,3)$-subgraph $F$ of $G_{deldom}$ containing $C$, an $(x,y)$-path $P$ of $G_{deldom}$ is a \emph{$F$-extension} if it satisfies the following.
\begin{enumerate}[(i)]
\item $x$ and $y$ are distinct non-branching points of $F$.
\item $\{x,y\}\cap V(C)\neq \emptyset$.
\item $P$ is a proper $V(F)$-path and $P-V(F)$ is an induced path of $G_{deldom}$.
\item There exists a vertex $v\in V(P)$ such that $\dist_P(v,\{x,y\}\cap V(C))=2$ and $v\notin N[F]$. 
\end{enumerate} 
Note that by  condition (iv), the length of an $F$-extension is at least $4$.

A cycle $H$ of $G_{deldom}$ is an \emph{almost $F$-extension} 
if it satisfies the following. 
\begin{enumerate}[(i)]
\item[(i')] $\abs{V(H)\cap V(C)}=1$ and the vertex in $V(H)\cap V(C)$ is a non-branching point of $F$ in $C$. 
\item[(ii')]  $H-V(C)$ is an induced path of $G_{deldom}$ and contains no vertex of $F$.
\item[(iii')] There exists a vertex $v\in V(H)$ such that $\dist_H(v,V(H)\cap V(C))=2$ and $v\notin N[F]$. 
\end{enumerate} 
We call the vertex in $V(H)\cap V(C)$ the \emph{root} of the almost $F$-extension $H$. 

It is not difficult to see that given a $(2,3)$-subgraph $F$ containing $C$, 
there is a polynomial-time algorithm to find a shortest $F$-extension $P$ or correctly decides that there is no $F$-extension. 
For this, we exhaustively guess five vertices $x,y,x',y',v$ such that 
\begin{itemize}
\item $x$ and $y$ are non-branching points of $F$ such that $x\in V(C)$,
\item $x'$ and $y'$ are neighbors of $x$ and $y$ in $V(G_{deldom})\setminus V(F)$, respectively, 
\item $v$ is a neighbor of $x'$ in $V(G)\setminus N[F]$.
\end{itemize}
Since we are looking for an $(x,y)$-path $P$ where $\mathring{x}P\mathring{y}$ is induced, 
we check whether there is a path from $v$ to $y'$ in $G_{deldom}-((V(F)\cup N[x'])\setminus \{v\})$. 
If there is such a path, then we find a shortest one $Q$. Then $xx'v\odot vQy'\odot y'y$ is an $F$-extension. 
Among all possible choices of five vertices $x,y,x',y',v$, we find a shortest $F$-extension using these five vertices.
Clearly if there is an $F$-extension, then we can find a tuple of such five vertices that outputs a shortest $F$-extension in the above procedure.

Throughout this section, we heavily rely on the structure of a maximal subgraph obtained by adding a sequence of $F$-extensions exhaustively. We additionally impose a tie break rule for the choice of $F$-extensions.

\begin{description}
\item [Initialize] $W_1=C$, $B_1=\emptyset$, and $i=1$.
\item [At step $i$]  We perform the following.
\begin{enumerate}
\item Find a shortest $W_i$-extension $P_{i}$ such that
\begin{itemize}
\item[] (\textbf{tie break rule}) $\abs{V(P_i)\cap V(C)}$ is maximum.
\end{itemize}
If no $W_i$-extension exists, then  terminate. Let $x_i$, $y_i$ be the endpoints of $P_{i}$ otherwise. 
\item Set $W_{i+1}:= (V(W_i)\cup V(P_{i}), E(W_{i})\cup E(P_{i}))$. 
\item Set $B_{i+1}:=B_{i} \cup \{x_i,y_i\}$ and increase $i$ by one.
\end{enumerate}
\end{description}
We remark that even with the tie break rule, the choice of a $W_i$-extension in (1) is not unique.

Notice that every vertex of $W_i$ has degree 2 or 3. Let $W_1, W_2, \ldots , W_{\ell}$ be 
the sequence of subgraphs constructed exhaustively until there is no $W_{\ell}$-extension. Let $W=W_{\ell}$  and $T_{branch}=B_{\ell}$. 
Throughout this section, we fix those sequences $W_1, W_2, \ldots , W_{\ell}=W$ and $P_1, P_2, \ldots, P_{\ell-1}$, and $B_1, B_2, \ldots, B_{\ell}=T_{branch}$. 
Clearly, the construction of $W$ requires at most $n$ iterations, and thus we can construct these sequences in polynomial time.

The first observation is that if $T_{branch}$ has size at least $s_{k+1}$, then $G[V(W)]$ contains $k+1$ vertex-disjoint holes. 
In fact, the construction of $W$ is calibrated so that every cycle of $W$ contains a hole of $G$. 
For this, the condition (iv) of $W$-extension is crucial. Due to the next lemma, we may assume that $\abs{T_{branch}}< s_{k+1}$.

\begin{LEM}\label{lem:manybranching}
If $W$ has at least $s_{k+1}$ branching points, then there are $k+1$ vertex-disjoint holes and they can be detected in polynomial time. 
\end{LEM}
\begin{proof}
By Theorem~\ref{thm:simonovitz}, $\abs{T_{branch}}\geq s_{k+1}$ implies that $W$ has at least $k+1$ vertex-disjoint cycles, and such a collection of cycles can be found in polynomial time.
 We shall prove that for each cycle $H$ of $W$, there is a hole in the subgraph of $G$ induced by $V(H)$. Clearly, this immediately establishes the statement.
 We fix a cycle $H$ of $W$. We may assume that $H\neq C$.
 Recall that for each $i$, $P_i$ is a $W_i$-extension added to $W_i$.

Let $i$ be the minimum integer such that $E(H)\subseteq E(W_{i+1})$. We claim that $P_{i}$ is entirely contained in $H$ as a subgraph. 
Notice that for any  $W_{j}$-extension $P_j$, every branching point $v\in T_{branch}$ which is an internal vertex of $P_j$ has been added at iteration $j'>j$. 
Therefore, if $P_{i}$ is not entirely contained in $H$ as a subgraph, then for some $i'>i$, there exists a subpath of $P_{i'}$ 
such that $E(P_{i'})\cap E(H)\neq \emptyset$, contradicting the choice of $i$. 

Let $x,y$ be the endpoints of $P_i$. 
Let $v$ be an internal vertex of $P_{i}$ that is not contained in $N[W_i]$.
Such a vertex exists by the condition (iv) of the definition of a $W$-extension. 
Let $Q:=H-v$. Since  $\mathring{x}P_i\mathring{y}$ is induced, the neighbors of $v$ in $P_i$ are not adjacent, and $v$ has no neighbors in the set of internal vertices of $Q$. 
Therefore, by Lemma~\ref{lem:twopaths}, $G[V(H)]$ contains a hole, as claimed.
\end{proof}

In the next step, we exhaustively find almost $W$-extensions and cover them if there are no $k+1$ vertex-disjoint holes. We show that if there are two almost $W$-extensions with roots $x_1$ and $x_2$ and $\dist_C(x_1, x_2)\ge 5$, then these two almost $W$-extensions do not intersect. This is because if they meet, then we can obtain a $W$-extension, contradicting the maximality of $W$. Using this, we can deduce that if there are $5k+5$ almost $W$-extensions with distinct roots, then there are $k+1$ vertex-disjoint holes. 
\begin{PROP}\label{prop:almostpacking}
There is a polynomial-time algorithm that
finds either $k+1$ vertex-disjoint holes or a vertex set $T_{almost}\subseteq V(C)\setminus T_{branch}$ of size at most $5k+4$ such that $T_{almost}\cup T_{branch}$ hits all almost $W$-extensions.  
\end{PROP}
\begin{proof}
Let $C=v_0v_1 \cdots v_{m-1}v_0$. All additions are taken modulo $m$. 
We greedily construct a collection of almost $W$-extensions $\mathcal{Y}=\{Y_1, Y_2, \ldots , Y_t\}$ (not necessarily vertex-disjoint) with distinct roots $v_{a_1}, v_{a_2}, \ldots, v_{a_t}\in V(C)\setminus T_{branch}$, and stop if $t$ reaches $5k+5$. 
To construct such a collection, we do the following for each vertex $v\in V(C)\setminus T_{branch}$:
\begin{enumerate}
\item Choose three vertices $w_1, w_2, w_3$ such that $w_1, w_3\in Z_v\setminus \{v\}$, $w_2\notin N[W]$, and $w_1$ is adjacent to $w_2$ but not adjacent to $w_3$.
\item Test whether there is a path from $w_2$ to $w_3$ in $G_{deldom}-((V(W)\cup N[w_1])\setminus \{w_2\})$.
If there is such a path $P$, then we add the cycle $H=w_1w_2\odot w_2Pw_3\odot w_3vw_1$ 
to $\mathcal{Y}$.
\end{enumerate}
It is not difficult to verify that there is an almost $W$-extension with root $v$ if and only if 
the algorithm outputs such a cycle $H$.

We claim that if $v_{a_p}$ and $v_{a_q}$ have distance at least $5$ in $C$, 
then $Y_p$ and $Y_q$ do not meet.

\begin{CLAIM}\label{claim:distancealmost}
Let $p,q\in \{1, 2, \ldots, t\}$. 
If $\dist_C(v_{a_p}, v_{a_q})\ge 5$, 
then $V(Y_p)\cap V(Y_q)=\emptyset$.
\end{CLAIM}
\begin{proofofclaim}
Suppose for contradiction that $Y_p$ and $Y_q$ meet at a vertex $z$.
Let $Y_p=p_1p_2 \cdots p_rv_{a_p}p_1$ and $Y_q=q_1q_2 \cdots q_sv_{a_q}q_1$.
For convenience let $p_0:=v_{a_p}$ and $q_0:=v_{a_q}$.
By the condition (iii') of an almost $W$-extension, 
we may assume that $p_2, q_2\notin N[W]$.
Let $t_1$ be the minimum integer such that $p_{t_1}$ has a neighbor in $Y_q$.
We choose a neighbor $q_{t_2}$ of $p_{t_1}$ in $Y_q$ with minimum $t_2$.
Let $R:=p_0Y_pp_{t_1}\odot p_{t_1}q_{t_2}\odot q_{t_2}Y_qq_0$.
It is not difficult to see that $R$ is an induced path.

Since $\dist_C(p_0, q_0)\ge 5$, the length of $R$ is at least $4$ by Lemma~\ref{lem:generalfarnonadj}. Therefore, 
$R$ contains either $p_2$ or $q_2$.
It implies that $R$ is a $W$-extension, contradicting the maximality of $W$.
\end{proofofclaim}

Suppose $t\geq 5k+5$. There exists $M=\{b_1, b_2, \ldots, b_{k+1}\}\subseteq \{a_1, a_2, \ldots, a_{t}\}\setminus \{v_i: m-4 \le i\le m-1\}$ 
such that for all $b_i, b_j\in M$, $b_i\equiv b_j\pmod 5$. 
As we exclude the vertices of $\{v_i: m-4 \le i\le m-1\}$, for every $b_i, b_j\in M$, $\dist_C(v_{b_i}, v_{b_j})\ge 5$.
By Claim~\ref{claim:distancealmost}, 
for $i, j\in \{1, 2, \ldots, k+1\}$, the corresponding almost $W$-extensions $Y_{b_{i}}$ and $Y_{b_{j}}$ are vertex-disjoint.
Thus, we can output $k+1$ vertex-disjoint holes in polynomial time. 
Otherwise, $\abs{\mathcal{Y}}\le 5k+4$, and thus the set of all roots $\{v_{a_1}, v_{a_2}, \ldots, v_{a_t}\} \subseteq V(C)\setminus T_{branch}$ 
of $\mathcal{Y}$ 
contain at most $5k+4$ vertices. Clearly, the set of all roots 
hits every element of $\mathcal{Y}$, that is, every almost $W$-extension.
\end{proof}

\subsection{$Q$-tunnels}\label{subsec:cfragment}

We define 
\begin{align*}
T_{ext}:=&T_{petal}\cup T_{full}\cup T_{trav:sunf}\cup T_{branch}\cup T_{almost}\cup\\
 &N_C^{20}[(T_{petal}\cup T_{full}\cup T_{trav:sunf}\cup T_{branch}\cup T_{almost})\cap V(C)].
\end{align*}
Note that 
\[ \abs{T_{ext}}\le 41(19k+(3k+14)+(15k+9)+(s_{k+1}-1)+(5k+4))\le 41(s_{k+1}+42k+26). \]
Since $\abs{T_{petal}\cup T_{full}\cup T_{trav:sunf} \cup T_{branch}\cup T_{almost}}\le s_{k+1}+42k+26$, $C-(T_{petal}\cup T_{full}\cup T_{trav:sunf} \cup T_{branch}\cup T_{almost})$ contains at most $s_{k+1}+42k+26$ connected components, so does $C-T_{ext}$.
Let $\mathcal{Q}$ be the set of connected components of $C-T_{ext}$ and we call each element of $\mathcal{Q}$ a \emph{$C$-fragment}.

 We want to show that for every $D$-avoiding tulip $H$ not hit by $T_{ext}$ and for every $Q\in \mathcal{Q}$, 
 if $H$ contains a vertex of $Q$ far from  the endpoints of $Q$, then $H$ must traverse the $Q$-tunnel from one entrance to the other entrance. 
 To argue this, we show that $H$ contains no edge $vw$ where $v\in Z_{V(Q)}$ and $w\notin N[C]$.
 We first need to show that in such a case, we have $w\notin N[W]$. 
  The next lemma states a useful distance property of $W$.
 See Figure~\ref{fig:distancelemma} for an illustration.

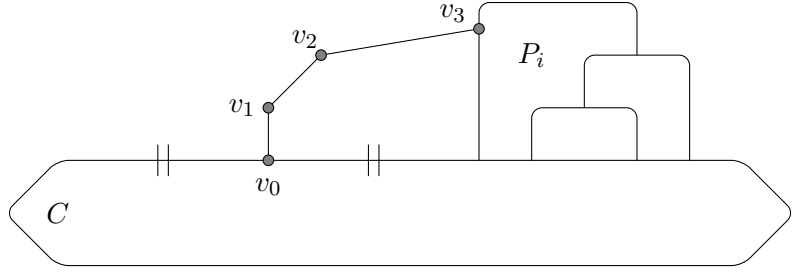
\begin{figure}
  \centering
  \begin{tikzpicture}[scale=0.7]
  \tikzstyle{w}=[circle,draw,fill=black!50,inner sep=0pt,minimum width=4pt]

   \draw[rounded corners] (6,0)--(11,0)--(12,-1)--(11,-2)--(-2,-2)--(-3,-1)--(-2,0)--(6,0);
   
   \draw[rounded corners] (7,0)--(7,1)--(9,1)--(9,0);
   \draw[rounded corners] (8,1)--(8,2)--(10,2)--(10,0);
   \draw[rounded corners] (6,0)--(6,3)--(9,3)--(9,2);
   
   \draw (-.1,0.3)--(-.1,-0.3);
   \draw (.1,0.3)--(.1,-0.3);
   \draw (4-.1,0.3)--(4-.1,-0.3);
   \draw (4.1,0.3)--(4.1,-0.3);
   
   \draw (2,0)--(2,1)--(3,2)--(6,2.5);
   \draw (2,0) node [w] {};
   \draw (2,1) node [w] {};
   \draw (3,2) node [w] {};
    \draw (6,2.5) node [w] {};
   
     \node at (7, 2) {$P_i$};

    \node at (2, -.5) {$v=v_0$};
    \node at (1.5, 1) {$v_1$};
    \node at (2.7, 2.3) {$v_2$};
   \node at (5, 2.8) {$u=v_3$};

     \node at (-2, -1) {$C$};

   \end{tikzpicture}     \caption{The setting in Lemma~\ref{lem:distance2} where $v$ is a vertex of $C-T_{ext}$ and there is a path of length $3$ from $v$ to $u\in V(W)\setminus V(C)$ whose internal vertices are in $V(G_{deldom})\setminus V(W)$. By Lemma~\ref{lem:distance2}, $v_2$ should have a neighbor in $C$. }\label{fig:distancelemma}
\end{figure}

\begin{LEM}\label{lem:distance2}
Let $Q\in \mathcal{Q}$ be a $C$-fragment, and 
$v\in V(Q)$ and $u\in V(W)$ with $v\neq u$. 
Then every $V(W)$-path $R=v_0v_1 \cdots v_{s}$ from $v_0=v$ to $v_s=u$ satisfies one of the following.
\begin{enumerate}[(1)]
\item $u$ is a vertex of $C$ such that $\dist_C(u, V(Q))\le 4$, 
\item $R$ has length $3$ and $v_2$ is adjacent to a vertex of $C$, 
\item $R$ has length at least 4.
\end{enumerate}
\end{LEM}
\begin{proof}
Recall that $W_1, W_2, \ldots , W_{\ell}=W$ is a sequence of subgraphs, $P_1, P_2, \ldots, P_{\ell-1}$ is a sequence of $W_i$-extensions, and $B_1, B_2, \ldots, B_{\ell}=T_{branch}$ is a sequence of branching points during the construction of $W$.

Suppose $u\in V(C)$. If a $V(W)$-path $R$ between $u$ and $v$ has length at most $3$, 
then there is an edge between $Z_u$ and $Z_v$ or we have $Z_u\cap Z_v\neq \emptyset$.
Then Lemma~\ref{lem:farnonadj} implies that $\dist_C(u,v)\le 3$, and $R$ satisfies (1).  
Therefore, we may assume $u\notin V(C)$. In particular, the following claim for every $i\in \{1, \ldots, \ell-1\}$  establishes the statement immediately. 
We prove it by induction on $i$: 
\begin{itemize}
\item[$(\ast)$] if $u$ is an internal vertex of $P_i$, then every $V(W)$-path $R$ between $v$ and $u$ satisfies (2) or (3).
\end{itemize}

Let $P_{i}=u_0u_1 \cdots u_p$. 
By definition of a $W_{i}$-extension, we may assume $u_0\in V(C)$ and $u_2\notin N[W_{i}]$.  
Suppose there exists a $V(W)$-path from $v$ to an internal vertex of $P_i$ violating  (3). 
Such a path has length at most $3$.
Let $s\in \{1,2,3\}$ be the minimum integer such that 
there is a $V(W)$-path of length $s$ between $v$ and an internal vertex of $P_i$. 
We choose the minimum integer $j\in \{1, 2, \ldots, p-1\}$ such that  
there is a $(v,u_j)$-path $R$ of length $s$.
Let $R:=v_0v_1 \cdots v_s$ with $v_0=v$ and $v_s=u_j$, and 
 $R_1=u_0P_iu_j\odot R$.  

We verify that $R_1$ is a $W_i$-extension. 
\begin{CLAIM}\label{claim:r1}
$R_1$ is a $W_i$-extension containing $u_2$. 
\end{CLAIM}
\begin{proofofclaim}
By the choice of $s$ and $u_j$, every vertex in $\mathring{v_0}R\mathring{u_j}$ has no neighbors in $\mathring{u_0}P_i\mathring{u_j}$. Therefore, $\mathring{u_0}R_1\mathring{v_0}$ is an induced path. Also, $v_0$ is a non-branching point of $W_i$. 
Hence, $R_1$ satisfies the conditions (i)-(iii) of $W_i$-extension. 
For (iv), it is sufficient to show that $j\ge 2$.
Suppose $j=1$. 
Then $R_1$ has length at most $4$, and by Lemma~\ref{lem:generalfarnonadj} with $m=2$, we have $\dist_C(v, u_0)\le 7$.
Since $u_0\in T_{branch}\cap V(C)$, this contradicts the fact that $v\in V(Q)$ and thus $\dist_C(v, T_{branch}\cap V(C))\ge 20$. 
We conclude that $j\ge 2$ and thus $R_1$ contains $u_2$. Since $P_i$ meets (iv) as a $W_i$-extension, we have $u_2\notin N[W_i]$.
Therefore, $R_1$ satisfies all four conditions for being a $W_i$-extension.
\end{proofofclaim}

Next, we show that $R$ has length exactly $3$. 
When $R$ has length $1$ or $2$, 
we derive a contradiction from the fact that $P_i$ is taken as a $i$-th $W_i$-extension.

\begin{CLAIM}\label{claim:length3}
$s=3$; that is, $R$ has length $3$.
\end{CLAIM}
\begin{proofofclaim}
First assume that $R$ has length $1$.
If $j<p-1$, then $R_1$ is shorter than $P_i$, which contradicts the fact that $P_i$ is taken as a shortest $W_i$-extension.
Thus we have $j=p-1$.
If $u_p$ is a vertex in $C$, observe that $u_{p-1}\in Z_{v_0}\cap Z_{u_p}$. Then, by Lemma~\ref{lem:farnonadj} we have $\dist_C(v_0,u_p)\leq3$. However, $u_p\in T_{branch}$ and $v_0\in V(Q)$ imply 
$\dist_C(v_0, u_p)\ge 20$, a contradiction. Hence, $u_p$ is not a vertex of $C$. This means that 
 $R_1$ should have been chosen as a $W_i$-extension instead of $P_i$ because of tie break rule, a contradiction.

Suppose now that $R$ has length $2$.
If $j<p-2$, then $R_1$ is shorter than $P_i$, which contradicts the fact that $P_i$ is taken as a shortest $W_i$-extension.
Thus we have $j=\{p-2, p-1\}$.
If $u_p$ is a vertex of $C$, then $\dist_C(v_0, u_p)\le 7$ by Lemma~\ref{lem:generalfarnonadj}.
This contradicts the fact that $\dist_C(v_0, u_p)\ge 20$.
Therefore, $u_p$ is not a vertex of $C$.
If $j=p-2$, then $R_1$ should be taken instead of $P_i$ because of tie break rule.

Hence, we may assume that $j=p-1$. Then $v_0v_1u_{p-1}u_p$ is a path of length $3$ from $v_0$ to $u_p\in V(W_i)$, 
and by induction hypothesis, $u_{p-1}$ has a neighbor in $C$.
Let $z$ be a neighbor of $u_{p-1}$ in $C$.

By Lemma~\ref{lem:generalfarnonadj}, we have $\dist_C(v_0, z)\le 3$. Therefore, 
\[z\notin T_{petal}\cup T_{branch}\cup T_{almost}\] since otherwise, $v_0$ is contained in the 20-neighborhood 
of $V(C)\cap (T_{petal}\cup T_{branch}\cup T_{almost})$ and thus, $v_0\in T_{ext}$. In particular, $z$ is a non-branching point of $W$.
We choose a neighbor $u_{j'}$ of $z$ such that $j'$ is minimum and let $R_2=zu_{j'}\odot u_{j'}P_iu_0$. 
Note that $j'\leq j=p-1$. 

It is easy to verify that $R_2$ is a $W_i$-extension similarly as in Claim~\ref{claim:r1}. 
Especially, $R_2$ meets condition (iv) because of $j'\geq 3$, which follows from the fact that 
$z$ has no neighbors in $\{u_1, u_2\}$ by  Lemma~\ref{lem:generalfarnonadj}.
If $j'<p-1$, then $R_2$ is shorter than $P_i$ contradicting the fact that $P_i$ is chosen as a shortest $W_i$-extension.
If $j'=p-1$, then since $u_p\notin V(C)$, $R_2$ should have been chosen instead of $P_i$ because of  tie break rule.
We conclude that $R$ cannot have length 2, which completes the proof of the claim.
\end{proofofclaim}

Now, we shall show that $v_2$ has a neighbor in $C$, thus $R$ satisfies (2). 
Suppose the contrary, and observe $v_2\notin N[C]$ because $R$ is a $V(W)$-path and $v_2\notin V(C)$.
If $v_2$ has a neighbor in $V(W_i)\setminus V(C)$, then by induction hypothesis, $v_2$ has a neighbor in $C$, a contradiction. 
Therefore, $v_2$ has no neighbors in $W_i$ and especially, $v_2\notin N[W_i]$.

If $j<p-3$, then $R_1$ is shorter than $P_i$, which contradicts that $P_i$ is taken as a shortest $W_i$-extension.
Thus we have $p-3\leq j \leq p-1$.
We choose a neighbor $u_{j'}$ of $v_2$ with maximum $j'$. 
By the choice of $j$ and from $v_2\notin N[W_i]$, 
we have $p-3\leq j\leq j' \leq p-1$.
Since $v_0$ or $v_1$ has no neighbors in $P_i$ by Claim~\ref{claim:length3}, 
 $v_0v_1v_2u_{j'}\odot u_{j'}P_iu_p$ is a $W_i$-extension in which  $v_2\notin N[W_i]$.

If $j\ge 4$, then  $v_0v_1v_2u_{j'}\odot u_{j'}P_iu_p$ is shorter than $P_i$, contradicting the fact that 
$P_i$ is chosen as a shortest $W_i$-extension.  If $j\le 3$, then $R_1$ has length at most $6$, and by Lemma~\ref{lem:generalfarnonadj}, we have $\dist_C(v_0, u_0)\le 15$, a contradiction to the fact that $\dist_C(v_0, u_0)\ge 20$.

Therefore, we conclude that $v_2$ is adjacent to a vertex of $C$. This proves the claim $(\ast)$, which completes the proof.
\end{proof}

The following is a simple, but important observation.
See Figure~\ref{fig:qtunnel} for an illustration.

\begin{LEM}\label{lem:tunnellemma1}
Let $Q\in \mathcal{Q}$ be a $C$-fragment, and let $H$ be a $D$-avoiding tulip in $G_{deldom}-T_{ext}$. 
Then $H$ contains no two adjacent vertices $v$ and $w$ such that $v$ is in the $Q$-tunnel and $w\in V(G_{deldom})\setminus N[C]$.
\end{LEM}
\begin{proof}
Suppose $H$ contains two adjacent vertices $v$ and $w$ such that $v\in Z_{V(Q)}$ and $w\in V(G_{deldom})\setminus N[C]$.
Since $v\in Z_{V(Q)}$ and $w\notin N[C]$, we have $v\notin V(C)$ and $v$ has a neighbor in $Q$. 
Let $z$ be a neighbor of $v$ in $Q$. 
We prove that $w$ has no neighbors in $W$.
\begin{CLAIM}\label{claim:notin}
$w\notin N[W]$.
\end{CLAIM}
\begin{proofofclaim}
By Lemma~\ref{lem:distance2}, $w\notin V(W)$.
Suppose $w$ has a neighbor in $W$. Since $w\notin N[C]$, $w$ has a neighbor in $V(W)\setminus   N[C]$. Let $u$ be such a neighbor. Since $zvwu$ is a path of length $3$, by Lemma~\ref{lem:distance2}, $w$ has a neighbor in $C$, a contradiction. Therefore, we conclude that $w$ has no neighbors in $W$. Thus, we have $w\notin N[W]$.
\end{proofofclaim}

Let $H=v_1v_2 \cdots v_mv_1$ where $v_1=v$ and $v_2=w$. 
Note that $H$ contains at least one non-branching point of $W$ since $(V(C)\setminus T_{ext})\cap V(H)\neq \emptyset$. We choose a minimum integer $i>2$ such that $v_i$ has a neighbor that is a non-branching point of $W$. Clearly $2<i\leq  m$ from the fact that $v_1$ is not in $C$ and by Lemma~\ref{lem:distance2}. We also observe that $v_2=w\notin N[W]$ by Claim~\ref{claim:notin} and that $v_1v_2 \cdots v_i$ is an induced path. 
Let $z'$ be a neighbor of $v_i$ which is a non-branching point of $W$. Then $zv_1v_2 \cdots v_iz'$ is a $W$-extension or an almost $W$-extension depending on whether $z=z'$ or not. It contradicts either the maximality of $W$ or that $T_{ext}$ hits all almost $W$-extensions. This completes the proof.
\end{proof}

Next, we prove that
if a $D$-avoiding tulip contains a vertex of a $C$-fragment $Q$ that is far from its endpoints $v$ and $w$, 
then its restriction on the $Q$-tunnel should be some path from $Z_v$ to $Z_w$. 
Since we will add all vertices of $C$-fragments having at most $35$ vertices to the deletion set for remaining $D$-avoiding tulips, we focus on $C$-fragments $Q$ with at least $36$ vertices.

\begin{LEM}\label{lem:tunnellemma3}
Let $Q=q_1q_2 \cdots q_m\in \mathcal{Q}$ be a $C$-fragment with at least $36$ vertices and let $R$ be the $Q$-tunnel.
Let $H$ be a $D$-avoiding tulip in $G_{deldom}-T_{ext}$ such that 
\begin{itemize}
\item $H$ contains no vertices in $\{q_i:1\le i\le 5, m-4\le i\le m\}$,
\item $H$ contains a vertex $v$ in $V(Q)$.
\end{itemize}
Then the connected component of the restriction of $H$ on $R$ containing $v$ is a path from $Z_{q_1}$ to $Z_{q_m}$.
\end{LEM}
\begin{proof}
Since $H$ is a tulip, $H$ is not fully contained in $R$.
Therefore, the component of the restriction of $H$ on $R$ containing $v$ is a path. 
Let $P$ be such a path. 

We claim that both endpoints of $P$ are contained in $Z_{\{q_1, q_m\}}$.
Suppose the contrary, and let $w$ be an endpoint of $P$ such that $w\in (\bigcup_{2\leq i\leq m-1} Z_{q_i})\setminus Z_{\{q_1, q_m\}}$. 
Let $\overline{Q}$ be the $(q_1, q_m)$-subpath of $C$ that does not contain $v_2$.

Let $w'\in N_H(w)\setminus V(P)$. Since $w$ is a vertex of $R$, Lemma~\ref{lem:tunnellemma1} implies that $w'\in N[C]$. 
Suppose $w'\in Z_{V(\overline{Q})}\setminus Z_{V(Q)}$.
We choose $y\in \spp(G[\{w\}])$ and $y'\in \spp(G[\{w'\}])$ so as to minimize $\dist_C(y,y')$. 
Then $\dist_C(y,y')\le 3$ by Lemma~\ref{lem:farnonadj}.
We also have $\dist_C(y,y')\ge 2$ because $y\in V(Q)\setminus \{q_1,q_m\}$ and $y'\in V(\overline{Q})$. 
Then Lemma~\ref{lem:threeconsecutive} implies that $w$ is not adjacent to $w'$, a contradiction.
Therefore, each endpoint of $P$ is contained in $Z_{\{q_1, q_m\}}$.

Now, we claim that the endpoints of $P$ are contained in distinct sets of $Z_{v_1}$ and $Z_{v_m}$.
Suppose for contradiction that both endpoints of $P$ are contained in the same set of $Z_{v_1}$ or $Z_{v_m}$. Without loss of generality, they are contained in $Z_{v_1}$.
Since $P$ contains no vertices in $\{q_i:1\le i\le 5\}$, 
$P$ contains two subpaths from $Z_{q_1}\setminus \{q_1\}$ to $Z_{q_5}\setminus \{q_5\}$.
But the supports of those two paths share three vertices $q_1, q_2, q_3$, and by Lemma~\ref{lem:overlay}, 
there is a petal contained in $Z_{\{q_1, q_2, q_3\}}$. This contradicts the assumption that $T_{petal}\subseteq T_{ext}$ contains the support of every petal.

This implies that one endpoint is in $Z_{q_1}$ and the other endpoint is in $Z_{q_m}$, as required. 
\end{proof}

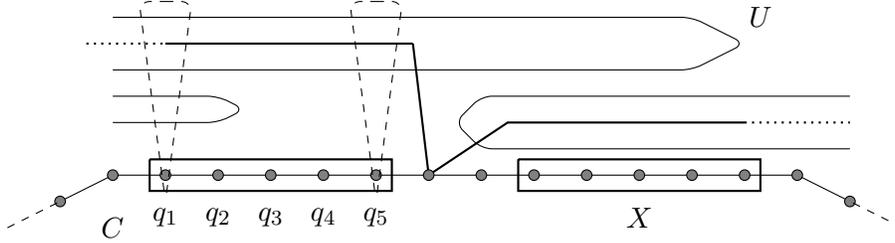
\begin{figure}
  \centering
  \begin{tikzpicture}[scale=0.7]
  \tikzstyle{w}=[circle,draw,fill=black!50,inner sep=0pt,minimum width=4pt]

	\draw[thick, dotted] (-2.5,2.5)--(-1,2.5);
	\draw[thick] (-1, 2.5)--(3.7,2.5);
	\draw[thick] (3.7,2.5)--(4,0)--(5.5,1)--(10,1);
	\draw[thick, dotted] (10, 1)--(12, 1);

   \draw (-2,0)--(11,0);
	\draw(-2, 0)--(-3,-0.5);
	\draw(11, 0)--(12,-0.5);
      \draw (-3,-.5) node [w] {};
       \draw (12,-.5) node [w] {};
 	\draw[dashed](13, -1)--(12,-0.5);
	\draw[dashed](-4,-1)--(-3,-0.5);
 
 \foreach \y in {-2,...,11}{
      \draw (\y,0) node [w] (a\y) {};
     }

     \node at (-2, -1) {$C$};
     \node at (10.3, 3) {$U_1$};
     \node at (-1, -.8) {$q_1$};
     \node at (0, -.8) {$q_2$};
     \node at (1, -.8) {$q_3$};
	\node at (2, -.8) {$q_4$};
	\node at (3, -.8) {$q_5$};
  
  	\draw[thick] (-1-.3,0.3)--(3+.3,0.3)--(3.3,-.3)--(-1.3,-.3)--(-1.3,0.3);
  	\draw[thick] (6-.3,0.3)--(10.3,0.3)--(10.3,-.3)--(5.7,-.3)--(5.7,0.3);
  
	  \node at (5, -1) {$P$};
         \path[->] (5, -.7) edge [bend right] (4.5, .3);

   \node at (8, -.8) {$X$};
    \draw[rounded corners] (12,1.5)--(5,1.5)--(4.5,1)--(5,.5)--(12,.5);
     \draw[rounded corners] (-2,3)--(9,3)--(10,2.5)--(9,2)--(-2,2);
     \draw[rounded corners] (-2,1+.5)--(0,1+.5)--(.5,0.75+.5)--(0,0.5+.5)--(-2,.5+.5);
    \foreach \y in {-1, 3}{
\draw[dashed, rounded corners] (\y, 3.3)--(\y+.5, 3.3)--(\y, -.5)--(\y-.5, 3.3)--(\y, 3.3);
}
 
   \end{tikzpicture}     
   \caption{An illustration of the set $X$ defined in Lemma~\ref{lem:tunnellemma2}. Since $X$ consists of the last $5$ vertices of the support of the component $U_1$, every path from $Z_{q_1}$ to $Z_{q_m}$ should move to a vertex of $Q$ appearing before $X$, and pass through another connected component of $R-(T_{ext}\cup V(Q))$ to reach $Z_{q_m}$.
   But this will lead to a petal whose support is near to $X$, contradicting the choice of the set $T_{ext}$. }\label{fig:distance2}
\end{figure}

Due to Lemma~\ref{lem:tunnellemma3}, we know that for any $D$-avoiding tulip $H$ in $G_{deldom}-T_{ext}$, 
there is a subpath $P$ of $H$ and a $Q$-tunnel $R$ such that $P$ is a path from one entrance of $R$ to the other entrance. 
The next  lemma describes how to find a hitting set for such path $P$ when its two endpoints belong to distinct connected components of $R-(T_{ext}\cup V(Q))$.
See Figure~\ref{fig:distance2} for an illustration.

\begin{LEM}\label{lem:tunnellemma2}
Let $Q=q_1q_2 \cdots q_m\in \mathcal{Q}$ be a $C$-fragment with at least $36$ vertices and let $R$ be the $Q$-tunnel.
One can find in polynomial time a vertex set $X\subseteq V(Q)\setminus \{q_i:1\le i\le 5, m-4\le i\le m\}$ of size at most $5$
hitting every path $P$ from $Z_{q_1}$ to $Z_{q_m}$ in $R-T_{ext}$ such that 
\begin{itemize}
\item $P$ contains no vertices in $\{q_i:1\le i\le 5, m-4\le i\le m\}$, 
\item the endpoints of $P$ are contained in distinct connected components of $R-(T_{ext}\cup V(Q))$.
\end{itemize} 
\end{LEM}
\begin{proof}
We begin with the following claim.

\begin{CLAIM}\label{claim:startingcomponent}
There is exactly one connected component of $R-(T_{ext}\cup V(Q))$ intersecting both $Z_{q_1}\setminus \{q_1\}$ and $Z_{q_5}\setminus \{q_5\}$. 
\end{CLAIM}
\begin{proofofclaim}
Since $P$ contains no vertices in $\{q_1,q_2, \ldots, q_5\}$, there is at least one component of $R-(T_{ext}\cup V(Q))$ intersecting both $Z_{q_1}\setminus \{q_1\}$ and $Z_{q_5}\setminus \{q_5\}$.
Suppose there are two such components $C_1$ and $C_2$.
For each $C_i$, we find a path $P_i$ form $Z_{q_1}\setminus \{q_1\}$ and $Z_{q_5}\setminus \{q_5\}$. Clearly $P_1$ and $P_2$ are vertex-disjoint, and there are no edges between $P_1$ and $P_2$.
As $\spp(P_1)$ and $\spp(P_2)$ share $3$ vertices $q_3, q_4, q_5$, 
by Lemma~\ref{lem:overlay}, 
$Z_{\{q_3, q_4, q_5\}}$ contains a petal. This contradicts the assumption that $T_{petal}\subseteq T_{ext}$ contains the support of every petal.
\end{proofofclaim}

Let $U_1$ be the connected component of $R-(T_{ext}\cup V(Q))$ intersecting both $Z_{q_1}\setminus \{q_1\}$ and $Z_{q_5}\setminus \{q_5\}$. 
Likewise, let $U_2$ be the unique connected component of $R-(T_{ext}\cup V(Q))$ intersecting both $Z_{q_{m-4}}\setminus \{q_{m-4}\}$ and $Z_{q_m}\setminus \{q_m\}$. 
Note that $U_1$ and $U_2$ are distinct since $P$ must intersect $V(U_1)\cap Z_{q_1}$ and $V(U_2)\cap  Z_{q_m}$.
Let $x$ be the maximum integer such that $Z_{q_x}\cap V(U_1)\neq \emptyset$ and let $X:=\{q_i:x-4\le i\le x\}$.

Now, we show that $X$ hits $P$. Suppose this is not the case. The choice of $x$ implies that $\{q_1,\ldots , q_x\}$ is a separator in $R-T_{ext}$ 
between $V(U_1)$ and $V(U_2)$. Since $P$ intersects both $V(U_1)$ and $V(U_2)$ while avoiding $\{q_1,\ldots , q_5\}\cup X$, 
it must contain a vertex of $\{q_6,\ldots , q_{x-5}\}$. (Especially, we have $x-5\ge 6$.) Let $P'_3$ be a $(\{q_6,\ldots , q_{x-5}\}, Z_{q_m})$-path which is a subpath of $P$. 
As a subpath of $P'_3$, we can choose  $(Z_{q_{x-4}},Z_{q_{x}})$-path $P_3$. 
Let $P_4$ be a path from $Z_{q_{x-4}}$ to $Z_{q_x}$ in $U_1$. 

Since no internal vertex of $P'_3$ belongs to $\{q_6,\ldots , q_{x-5}\}$, $P'_3$ and thus $P_3$ does not contain 
a vertex of $U_1$.   
Hence, $P_3$ and $P_4$ are disjoint. 
Notice that $P_3$ is contained 
in $Z_{X}\setminus X$ due to Lemma~\ref{lem:pathsupport} and the assumption $V(P)\cap X=\emptyset$. 
Therefore, 
$P_3$ is a subpath of $P$  that is contained in some component of $R-(T_{ext}\cup V(Q))$ different from $U_1$.  
Therefore,  there is no edge between $V(P_3)$ and $V(P_4)$.
As $\spp(P_3)$ and $\spp(P_4)$ share three vertices $q_{x-4}, q_{x-3}, q_{x-2}$, 
by Lemma~\ref{lem:overlay}, 
$Z_{\{q_{x-4}, q_{x-3}, q_{x-2}\}}$ contains a petal. 
However, this contradicts the assumption that $T_{petal}\subseteq T_{ext}$ contains the support of every petal. 
We conclude that $X$ hits $P$.
\end{proof}

The path meeting the conditions of Lemma~\ref{lem:tunnellemma3} can be hit by a vertex set obtained in Lemma~\ref{lem:tunnellemma2}, unless its endpoints are contained in the same connected component of $R-(T_{ext}\cup V(Q))$.
If the endpoints are contained in the same component of $R-(T_{ext}\cup V(Q))$, then clearly the endpoints can be connected via a path of $R$ traversing no vertices of $Q$. We prove that if its endpoints are contained in the same component of $R-(T_{ext}\cup V(Q))$ and the path contains a vertex of $Q$, 
then we could reroute to find another $D$-avoiding tulip with less vertices of $C$.

\begin{LEM}\label{lem:tunnellemma4}
Let $Q=q_1q_2 \cdots q_m\in \mathcal{Q}$ be a $C$-fragment of at least $36$ vertices and let $R$ be the $Q$-tunnel.
Let $H$ be a $D$-avoiding tulip in $G_{deldom}-T_{ext}$ such that 
\begin{itemize}
\item $H$ contains no vertices in $\{q_i:1\le i\le 15, m-14\le i\le m\}$,
\item $H$ contains a vertex $v$ in $V(Q)$, 
\item the endpoints of the restriction of $H$ on $R$ containing $v$ are contained in the same connected component of $R-(T_{ext}\cup V(Q))$.
\end{itemize}
Then there is a $D$-avoiding tulip $H'$ in $G_{deldom}-T_{ext}$ such that $V(C)\cap V(H')\subsetneq V(C)\cap V(H)$.
\end{LEM}
\begin{proof}
By Lemma~\ref{lem:tunnellemma3}, 
the connected component of the restriction of $H$ on $R$ is a path from $Z_{q_1}$ to $Z_{q_m}$.
Let $P=p_1p_2 \cdots p_n$ be the path such that $p_1\in Z_{q_1}$ and $p_n\in Z_{q_m}$.
We choose the minimum integer $x$ such that $p_x$ is contained in $Z_{q_{15}}$
and choose the maximum integer $y$ such that $p_y$ is contained in $Z_{q_{m-14}}$.
Let $U$ be the connected component of $R-(T_{ext}\cup V(Q))$ containing $p_1$ and $p_n$.
Then $p_x$ is a vertex of $U$ since otherwise $p_1Pp_x$  traverses a vertex of $Q$, 
which must be in $\{q_{16},\ldots, q_{m-15}\}$; this means that $p_1Pp_x$ contains a vertex of 
$Z_{q_{15}}$ as an internal vertex by Lemma~\ref{lem:connectedsupport}, contradicting the choice of $x$. 
Similarly, $p_y$ is in $U$. 

Let $J$ be a shortest path from $p_x$ to $p_y$ in $U$.
We want to show that $p_2$ has no neighbors in $J$.
To show this, we claim that $J$ does not intersect $Z_{q_9}$.

\begin{CLAIM}\label{claim:nine}
$J$ does not intersect $Z_{q_9}$.
\end{CLAIM}
\begin{proofofclaim}
Suppose for contradiction that $J$ contains a vertex $r\in Z_{q_9}$.
Let $J_1$ and $J_2$ be the two components of $J-r$.
Since $J$ is induced, there are no edges between $J_1$ and $J_2$.
Also, by Lemma~\ref{lem:farnonadj}, 
for each $i\in \{1, 2\}$, the endpoint of $J_i$ adjacent to $r$ 
should have a neighbor in $C$ which has distance at most $4$ from $q_9$.
This implies that $\spp(J_1)$ and $\spp(J_2)$ both contain $\{q_{13}, q_{14}, q_{15}\}$.
Then by Lemma~\ref{lem:overlay}, we can find a petal  contained in $Z_{\{q_{13}, q_{14}, q_{15}\}}$.
This contradicts that $T_{petal}\subseteq T_{ext}$ contains the support of all petals.
\end{proofofclaim}

Suppose that $p_2$ has a neighbor in $J$. Consequently, there is a $V(C)$-path from $q_1$ to a vertex of $\spp(J)$ 
traversing $p_1$ and $p_2$ whose length is $4$. By Lemma~\ref{lem:generalfarnonadj}, the endpoint of this path contained in $\spp(J)$ 
is within distance at most 7 in $C$. Thus, by Lemma~\ref{lem:connectedsupport}, $\spp(J)$ contains $q_9$ and this contradicts Claim~\ref{claim:nine}. 
Therefore, $p_2$ does not have a neighbor in $J$. 

Let $P_{rem}$ be the subpath of $H$ from $p_1$ to $p_m$ not containing $p_2$.
Observe that 
\[P_{new}=p_3Pp_x\odot p_xJp_y\odot p_yPp_m\odot p_mP_{rem}p_1\]
is a walk from $p_3$ to $p_1$ in $G_{deldom}$.
Observe that the only vertices of $P_{new}$ adjacent to  $p_2$ are $p_1$ and $p_3$.
Also, $p_1$ is not adjacent to $p_3$.
Thus by applying Lemma~\ref{lem:twopaths} to $p_1p_2p_3$ and a shortest path from $p_1$ to $p_3$ in $P_{new}$, we derive that 
$G_{deldom}[\{p_2\}\cup V(P_{new})]$ contains a hole $H'$.
Clearly, $H'$ is a $D$-avoiding hole. If it is a sunflower, then $H'$ is hit by $T_{full}$, a contradiction.
Thus, $H'$ is a $D$-avoiding tulip.
Clearly, $V(H')\cap V(C)\subsetneq V(H)\cap V(C)$ because the restriction of $H'$ on $R$ containing $p_2$ does not contain a vertex of $Q$.
\end{proof}

We are ready to construct a hitting set for all $D$-avoiding tulips. 
In the hitting set, we impose an additional condition that will be used for hitting $D$-traversing tulips.

\begin{PROP}\label{prop:Davoid}
There is a polynomial-time algorithm which finds either $k+1$ vertex-disjoint holes in $G$ or a vertex set $T_{avoid:tulip}\subseteq V(G)\setminus T_{ext}$ of at most $35(s_{k+1}+42k+26)$ vertices such that 
\begin{itemize}
\item $T_{ext}\cup T_{avoid:tulip}$ contains $N^{15}_C[T_{ext}\cap V(C)]$, and
\item $T_{ext}\cup T_{avoid:tulip}$ hits all $D$-avoiding tulips.  
\end{itemize}
\end{PROP}
\begin{proof}
We construct a set $T_{avoid:tulip}$ as follows:
\begin{enumerate}[(1)]
\item for each component of $C-T_{ext}$ having at most $35$ vertices, we add all vertices to $T_{avoid:tulip}$,
\item for each component $Q=q_1q_2 \cdots q_m$ of $C-T_{ext}$ with $m\ge 36$, 
we add $\{q_i:1\le i\le 15, m-14\le i\le m\}$, and the set obtained in Lemma~\ref{lem:tunnellemma2} to $T_{avoid:tulip}$.
\end{enumerate}
Since $C-T_{ext}$ contains at most $s_{k+1}+42k+26$ connected components, 
we have $\abs{T_{avoid:tulip}}\le 35(s_{k+1}+42k+26)$. 
Furthermore, $T_{ext}\cup T_{avoid:tulip}$ contains the $15$-neighborhood of $T_{ext}\cap V(C)$ in $C$.
We claim that $T_{ext}\cup T_{avoid:tulip}$ hits all $D$-avoiding tulips.

Suppose for contradiction that there is a $D$-avoiding tulip $H$ in $G_{deldom}-(T_{ext}\cup T_{avoid:tulip})$.
We choose such a tulip with minimum $\abs{V(C)\cap V(H)}$.
Since $H$ is a hole, it contains a vertex in $V(C)\setminus T_{ext}$, say $v$.
Let $Q=q_1q_2 \cdots q_m$ be a connected component of $C-T_{ext}$ containing $v$.
By (1), we have $m\ge 36$.

As $\{q_i:1\le i\le 15, m-14\le i\le m\}$ was added to $T_{avoid:tulip}$, 
$v$ is a vertex in $V(Q)\setminus \{q_i:1\le i\le 15, m-14\le i\le m\}$. 
Let $R$ be the $Q$-tunnel, and let $P$ be the restriction of $H$ on the $Q$-tunnel containing $v$.

By Lemma~\ref{lem:tunnellemma3}, 
$P$ is a path from $Z_{q_1}$ to $Z_{q_m}$.
If the endpoints of $P$ are contained in the distinct components of $R-(T_{ext}\cup V(Q))$, then 
$T_{avoid:tulip}$ hits this path by Lemma~\ref{lem:tunnellemma2}, a contradiction.
Assume the endpoints of $P$ are contained in the same component of $R-(T_{ext}\cup V(Q))$.
Then by Lemma~\ref{lem:tunnellemma4}, there is a $D$-avoiding tulip $H'$ with $V(H')\cap V(C)\subsetneq V(H)\cap V(C)$, 
contradicting the minimality of $H$.

Therefore, $T_{ext}\cup T_{avoid:tulip}$ intersects all $D$-avoiding tulips.
\end{proof}

\subsection{Handling $D$-traversing tulips} \label{subsec:Dtulip}

By Lemmas~\ref{lem:dominating} and~\ref{lem:neighborofdominating}, 
every $D$-traversing tulip $H$ contains precisely one $C$-dominating vertex and it contains one or two vertices of $C$. 
Also, the vertices in $V(H)\cap V(C)$ are adjacent to the unique $C$-dominating vertex in $H$.

Here, we use a technique similar to the one in Theorem~\ref{prop:skew}. That is,  we construct auxiliary bipartite graphs, and we will find either $k+1$ vertex-disjoint holes, or a set covering all such tulips.

\begin{LEM}\label{lem:dominatingtulip2}
There is a polynomial-time algorithm which finds either $k+1$ vertex-disjoint holes in $G$ or a vertex set 
$T_{trav:tulip}\subseteq (D\cup V(C))\setminus T_{ext}\setminus T_{avoid:tulip}$ of size at most $25k+9$ such that $T_{ext}\cup T_{avoid:tulip}\cup T_{trav:tulip}$ hits every $D$-traversing tulip.  
\end{LEM}

\begin{proof}
Let $C=v_0v_1 \cdots v_{m-1}v_0$. All additions are taken modulo $m$. 
We create an auxiliary bipartite graph $\mathcal{G}_i=(D\uplus \mathcal{A}_i, \mathcal{E}_i)$ for each $0\leq i \leq 4$, such that
\begin{itemize}
\item $\mathcal{A}_i=\{v_{5j+i}: j=0, \ldots ,\lfloor \frac{m}{5} \rfloor -1\}$,
\item there is an edge between $d\in D$ and $x\in \mathcal{A}_i$ if and only if  
there is an  $(x,d)$-path $P$ in $G- ((D\cup V(C))\setminus \{d,x\})-T_{ext}-T_{avoid:tulip}$ such that 
\begin{itemize}
\item $G[V(P)]$ is a hole, 
\item the second neighbor of $x$ in $P$ is not in $N[C]$.
\end{itemize}
\end{itemize}

It is not difficult to see that each auxiliary graph $\mathcal{G}_i$ can be constructed in polynomial time. 
For a pair of $d\in D$ and $x\in V(C)$, 
we first ensure that $dx\in E(G)$ and $\{d,x\}\cap T_{ext}=\emptyset$. 
By guessing the vertices  $y,z$ such that $z\notin N[C]$ and $dxyz$ forms an induced path,  and then 
computing a shortest $(z,d)$-path, 
we can find an $(x,d)$-path $P$ satisfying the two conditions above. 
On the other hand, if there is an $(x,d)$-path meeting the above conditions, then we can find such a path $P$ with the corresponding choice of $y$ and $z$.

Suppose there exists $i\in \{0,1,\ldots, 4\}$ such that $\mathcal{G}_i$ contains a matching $M$ of size at least $k+1$. 
We argue that there are $k+1$ holes in this case. 

Let $e_1=(d_1,x_1)$ and $e_2=(d_2,x_2)$ be two distinct edges of $M$. 
By construction, we have $\dist_C(x_1,x_2)\ge 5$, and 
for each $i\in \{1,2\}$,  
there is an $(x_i,d_i)$-path $P_i$ in  $G-((D\cup V(C))\setminus \{d_i,x_i\})-T_{ext}-T_{avoid:tulip}$ such that 
$G[V(P_i)]$ is a hole and
the second neighbor of $x_i$ in $P_i$ is not in $N[C]$.
Let $y_i$ and $z_i$ be the first and second neighbors of $x_i$ in $P_i$ respectively, for each $i$.

First, we show that $z_i\notin N[W]$. Notice that $x_i\notin T_{ext}$ and $x_i$ is a vertex of a $C$-fragment.  
If $y_i\in V(W)$ or $z_i\in V(W)$, then Lemma~\ref{lem:distance2} implies that $y_i\in V(C)$ or $z_i\in V(C)$, 
which is not possible since $P_i$ is a path with $V(P_i)\cap V(C)=\{x_i\}$.
Hence, we have $y_i,z_i\notin V(W)$. Suppose that $z_i$ has a neighbor $z_i'\in V(W)$.
Since $x_iy_iz_iz_i'$ is a $V(W)$-path of length $3$, 
$z_i$ has a neighbor in $C$ by Lemma~\ref{lem:distance2}. 
This contradicts the assumption that $z_i\notin N[C]$.
We conclude that $z_i\notin N[W]$.

Next, we show that if $P_i$ contains a non-branching point of $W$ other than $x_i$, 
then there is a $W$-extension.
\begin{CLAIM}\label{claim:wextension1}
$P_i$ contains no point of $W$ other than $x_i$.
\end{CLAIM}
\begin{proofofclaim}
We prove for $P_1$; the same proof holds for $P_2$.
Suppose the claim does not hold. Recall that $V(P_1)\cap T_{branch}\subseteq V(P_1)\cap T_{ext}=\emptyset$.  
Hence, we may assume that $P_1$ contains a non-branching point of $W$.
Let $P_1:=p_1p_2 \cdots p_m$ where $p_1=x_1$ and $p_m=d_1$, 
and let $j\in \{1,2, \ldots, m\}$ be the minimum integer such that $p_j$ is a non-branching point of $W$ other than $x_1=p_1$.
Clearly, $p_j\notin V(C)$, as $p_1$ is the unique vertex of $P_i$ contained in $C$.

If $j\le 4$, 
then $p_1P_1p_j$ is a path of length at most $3$, and 
since $p_1\notin T_{ext}$, by Lemma~\ref{lem:distance2},
$j=4$ and $p_{j-1}$ has a neighbor in $C$. 
However, it contradicts the choice of $P_1$ that the second neighbor $y_1=p_3$ of $x_1$ is not in $N[W]$.
Therefore, $j\ge 5$.
It means that $p_1P_1p_5$ is a $W$-extension containing $p_3\notin N[W]$.
This contradicts the maximality of $W$.
\end{proofofclaim}

We further claim that if $P_1$ and $P_2$ intersect, then there is a $W$-extension.
\begin{CLAIM}\label{claim:wextension2}
$P_1$ and $P_2$ do not share a vertex.
\end{CLAIM}
\begin{proofofclaim}
Suppose that $P_1$ and $P_2$ intersect. 
We choose $f_1\in V(P_1)$ having a neighbor in $P_2$  so that $\dist_{P_1}(f_1,x_1)$ is minimized.
Among neighbors of $f_1$ in $P_2$, we choose $f_2$ that is closest to $x_2$ in $P_2$.
By the choice of $f_1$ and $f_2$, $x_1P_1f_1\odot f_1f_2\odot f_2P_2x_{2}$ is an induced path. 
Note that there are no edges between $Z_{x_1}$ and $Z_{x_2}$ because $\dist_C(x_1, x_2)\ge 5$.
Therefore, $x_1P_1f_1\odot f_1f_2\odot f_2P_2x_{2}$ contains at least one of $z_1$ and $z_2$, which are not in $N[W]$.
By Claim~\ref{claim:wextension1}, $x_1P_1f_1\odot f_1f_2\odot f_2P_2x_{2}$ is a $V(W)$-path.
It implies that $x_1P_1f_1\odot f_1f_2\odot f_2P_2x_{2}$ is a $W$-extension, contradicting the maximality of $W$.
Therefore, $P_1$ and $P_2$ do not share a vertex.
\end{proofofclaim}

Claim~\ref{claim:wextension2} implies that if a bipartite graph $\mathcal{G}_i$ contains a matching of size $k+1$, then 
we can construct $k+1$ vertex-disjoint holes in polynomial time. 
Consider the case when for every $0\le i\leq 4$, $\mathcal{G}_i$ admits a vertex cover $S_i$ of size at most $k$. For $S_i$, let $S^*_i$ be the vertex set 
\[ (S_i\cap D) \cup \bigcup_{x\in S_i\cap \mathcal{A}_i} N^2_C[x] \]
and let $T_{trav:tulip}:=\left( \bigcup_{i=0}^4 S^*_i \right) \cup \{v_{5\lfloor \frac{m}{5} \rfloor+i}:-2\le i\le 6\}$. Notice that $\abs{T_{trav:tulip}}\leq 25k+9$.

In what follows, we will prove that 
the vertex set $T_{ext}\cup T_{avoid:tulip}\cup T_{trav:tulip}$ indeed hits all $D$-traversing tulips. 
Suppose that there is a $D$-traversing tulip $H$ containing $d\in D$ and $x\in V(C)$ in $G-(T_{ext}\cup T_{avoid:tulip}\cup T_{trav:tulip})$.
Clearly $P:=H-dx$ is an $(x,d)$-path such that $G[V(P)]$ is a hole, but it may not certify the existence of the edge $dx$ in the auxiliary bipartite graph, 
because the second neighbor of $x$ in $P$ can be in $N[C]$.
Let $P=p_1p_2 \cdots p_m$ with $p_1=x$ and $p_m=d$.
Let $w_1, w_2, \ldots, w_5$ be the consecutive vertices on $C$ such that $w_3=x$.
Let $i$ be the minimum integer such that $p_{i+1}\notin N[C]$. Such $p_{i+1}$ exists because $H$ is a tulip.

\begin{CLAIM}\label{claim:earlyleft}
We have $\{p_1, \ldots, p_i\}\subseteq Z_{\{w_2, w_3, w_4\}}\setminus Z_{\{w_1, w_5\}}$. 
\end{CLAIM}
\begin{proofofclaim}
Suppose not. By Lemma~\ref{lem:connectedsupport}, $p_1Pp_i$ contains an $(x,Z_{\{w_1, w_5\}})$-subpath $Q$.
Then it is easy to see that $dx\circ Q$
meets the preconditions of Lemma~\ref{lem:smallsunflower}; especially every internal vertex of $Q$ 
is in $Z_{\{w_2, w_3, w_4\}}\setminus Z_{\{w_1, w_5\}}$ by Lemma~\ref{lem:pathsupport}.
Therefore, Lemma~\ref{lem:smallsunflower} implies that 
there exists a $D$-traversing sunflower $H'$ containing $v$ and $d$ such that 
$V(H')\setminus \{d\}$ is contained in either $Z_{\{w_1, w_2, w_3\}}\cap (V(Q)\cup \{w_1\})$ or $Z_{\{w_3, w_4, w_5\}}\cap (V(Q)\cup \{w_5\})$. 
Since $T_{petal}\cup T_{trav:sunf}$ hits $H'$  by Proposition~\ref{prop:skew} while $(T_{petal}\cup T_{trav:sunf})\cap  (V(Q)\cup \{d\})
\subseteq T_{ext}\cap V(H)=\emptyset$, 
either $w_1$ or $w_5$ must be contained in $T_{petal}\cup T_{trav:sunf}$. However, by the construction of $T_{ext}$, 
this implies $x=w_3\in T_{ext}$, a contradiction.
This proves the claim.
\end{proofofclaim}

By Claim~\ref{claim:earlyleft}, $p_i$ has a neighbor in $\{w_2, w_3, w_4\}$. 
Let $w$ be a neighbor of $p_i$ in $\{w_2, w_3, w_4\}$. 
Note that $\{w_2,w_3,w_4\}\cap T_{ext}=\emptyset$ since otherwise, we have $x=w_3\in T_{ext}\cup T_{avoid:tulip}$ by the construction in Proposition~\ref{prop:Davoid}, 
contradicting the assumption that $V(H)\cap (T_{ext}\cup T_{avoid:tulip})=\emptyset$.

\begin{CLAIM}\label{claim:secondvertex}
We have $p_{i+1}\notin N[W]$. 
\end{CLAIM}
\begin{proofofclaim}
By Lemma~\ref{lem:distance2}, $p_{i+1}$ cannot be in $W$. 
We show that $p_{i+1}\notin N(W)$.
Suppose that 
$p_{i+1}$ has a neighbor $p_{i+1}'$ in $W$.
Then $wp_ip_{i+1}p'_{i+1}$ is a $V(W)$-path and $p'_{i+1}\notin V(C)$ because of $p_{i+1}\notin N[C]$.
Recall that $w\notin T_{ext}$.
Now, Lemma~\ref{lem:distance2} applies and we have $p_{i+1}\in N(C)$. This contradicts the choice of $i$ and $p_{i+1}\notin N[C]$.
It follows that $p_{i+1}\notin N(W)$.
\end{proofofclaim}

\begin{CLAIM}\label{claim:nopoint}
$P$ contains no point of $W$ other than $p_1$. 
\end{CLAIM}
\begin{proofofclaim}
Suppose the contrary. Clearly, $P$ does not contain a branching-point of $W$ 
because $V(P)\cap T_{branch}\subseteq V(H)\cap T_{ext}=\emptyset$. 
Therefore, we may assume that $P$ contains a non-branching point of $W$ other than $p_1$. 
Let $j\in \{2, \ldots, m\}$ be the minimum integer such that $p_j$ is a non-branching point of $W$. 

Suppose $j\leq i$. By Claim~\ref{claim:earlyleft}, we have $p_j\in Z_{\{w_2, w_3, w_4\}}\setminus Z_{\{w_1, w_5\}}$. 
Let $p'_j\in \{w_2,w_3,w_4\}$ be a neighbor of $p_j$. Then $p'_jp_j$ is a $V(W)$-path and $p'_j\notin T_{ext}$. 
Therefore, Lemma~\ref{lem:distance2} applies to $p'_jp_j$ and we have $p_j\in V(C)$, thus $p_j\in \{w_2,w_3,w_4\}$. 
Note that $P$ contains no vertices in $\{w_2, w_4\}$ as $H$ contains no triangles. Hence, it follows $p_j=w_3=p_1$. 
This contradicts the assumption that $p_j\neq p_1$. 

Therefore, we  assume $j\ge i+1$. 
Notice that $wp_i\odot p_iPp_j$ is a $V(W)$-path. We want to argue that this is a $W$-extension, deriving a contradiction.
If $j\in \{i+1, i+2\}$, then from $w\notin T_{ext}$ and by Lemma~\ref{lem:distance2} we know that 
$j=i+2$ and $p_{i+1}$ has a neighbor in $C$. 
However, this again contradicts $p_{i+1}\notin N[C]$.
Therefore, we have $j\ge i+3$.
It means that $wp_i\odot p_iPp_j$ has length at least $4$, and thus $p_{i+1}\notin N[W]$, which makes $wp_i\odot p_iPp_j$ 
qualify as a $W$-extension.
This contradicts the maximality of $W$. We conclude that $H$ contains no non-branching point of $W$ other than $p_1$.
\end{proofofclaim}

Let $\ell\geq i+1$ be the minimum integer such that $p_{\ell}$ is a neighbor of $w$. Since $p_m=d$ is a neighbor of $w$, 
such $\ell$ exists. Furthermore, $\ell> i+1$ because we have $p_{i+1}\notin N[C]$ due to  the choice of $i$. 
Observe that $p_{\ell}wp_i$ is an induced path with $w$ as an internal vertex and $w$ is not adjacent to any internal vertex 
of $p_iPp_{\ell}$. Now Lemma~\ref{lem:twopaths} applies, implying that $G[V(p_iPp_{\ell})\cup \{w\}]$ has a hole $H'$ containing $w$. 
By Claim~\ref{claim:nopoint}, $H'$ contains no point of $W$ other than $w$. 

Observe that $H'$ qualifies as an almost $W$-extension if $p_{\ell}\neq d$; especially we have $p_{i+1}\notin N[W]$ by Claim~\ref{claim:secondvertex}.
Therefore $T_{almost}$ hits $H'$. On the other hand, $T_{almost}\cap (V(H')\setminus \{w\})\subseteq T_{ext}\cap  V(H)=\emptyset$, which 
implies $w\in T_{almost}$. Then by the construction of $T_{ext}$, we have $x\in T_{ext}$, a contradiction. 
If $p_{\ell}=d$, then $H'-dw$ is a path certifying an edge in an auxiliary bipartite graph. Therefore either one of $\{d,w\}$ is contained in the vertex cover 
or $w=v_{5\lfloor{\frac{m}{5}}\rfloor+a}$ with $0\leq a \leq 4$. In both cases, $x$ is included in $T_{trav:tulip}$, a contradiction.
This completes the proof. 
\end{proof}

\subsection{Proof of our main result}\label{subsec:final}

We prove Theorem~\ref{thm:core}. 

We apply Lemma~\ref{lem:petalcover}, Proposition~\ref{prop:hitsunflower}, and Proposition~\ref{prop:skew}.
Over all, we can in polynomial time either output $k+1$ vertex-disjoint holes or vertex sets $T_{petal}, T_{full}, T_{trav:sunf}$ hitting petals, full sunflowers, and $D$-avoiding sunflowers, respectively.

We construct $W$ with the set $T_{branch}$ of branching points as described in Subsection~\ref{subsec:tuliphive}.
By Lemma~\ref{lem:manybranching}, 
if $W$ has at least $s_{k+1}$ branching points, then there are $k+1$ vertex-disjoint holes and they can be detected in polynomial time. 
We apply Proposition~\ref{prop:almostpacking}.
If it outputs 
$k+1$ vertex-disjoint holes in $G$, then we are done.
We may assume it outputs
a vertex set $T_{almost}$ of at most $5k+4$ vertices where $T_{almost}$ hits all almost $W$-extensions.

Let $T_{ext}$ be the union of $T_{petal}\cup T_{full}\cup T_{trav:sunf}\cup T_{branch}\cup T_{almost}$ and the $20$-neighborhood of $V(C)\cap (T_{petal}\cup T_{full} \cup T_{trav:sunf} \cup T_{branch}\cup T_{almost})$.

By Proposition~\ref{prop:Davoid}, we can in polynomial time either find $k+1$ vertex-disjoint holes or find a set 
$T_{avoid:tulip}\subseteq V(G)\setminus T_{ext}$ of at most $35(s_{k+1}+42k+26)$ vertices such that $T_{ext}\cup T_{avoid:tulip}$ hits all $D$-avoiding tulips.
By Lemma~\ref{lem:dominatingtulip2}, we can either find $k+1$ holes 
or find a set $T_{trav:tulip}\subseteq V(G)\setminus (T_{ext}\cup T_{avoid:tulip})$ of size $25k+9$ such that 
$T_{ext}\cup T_{avoid:tulip}\cup T_{trav:tulip}$ hits all $D$-traversing tulips. 
Therefore, we can either find $k+1$ vertex-disjoint holes, or output a vertex set with at most 
\begin{align*}
&\abs{T_{ext}\cup T_{avoid:tulip}\cup T_{trav:tulip}} \\
&\le 41(s_{k+1}+42k+26)+ 35(s_{k+1}+42k+26)+25k+9 \\
										&\le 76s_{k+1}+3217k+1985
										\end{align*}
						vertices hitting all holes.
This completes the proof of Theorem~\ref{thm:core}.

\section{Cycles of length at least $5$ do not have the Erd\H{o}s-P\'osa property under the induced subgraph relation}\label{sec:lowerbound}

In this section, we show that the class of cycles of length at least $\ell$ for every fixed $\ell\ge 5$ 
has no Erd\H{o}s-P\'osa property under the induced subgraph relation.

A \emph{hypergraph} is a pair $(X,\mathcal{E})$ such that $X$ is a set and $\mathcal{E}$ is a family of non-empty subsets of $X$, called \emph{hyperedges}.
A subset $Y$ of $X$ is called a \emph{hitting set} if for every $F\in \mathcal{E}$, $Y\cap F\neq \emptyset$.
For positive integers $a,b$ with $a\geq b$, 
let an \emph{$(a,b)$-uniform hypergraph}, denote it by $U_{a,b}$, be the hypergraph $(X, \mathcal{E})$ such that 
$\abs{X}=a$ and $\mathcal{E}$ is the set of all subsets of $X$ of size $b$.
It is not hard to observe that in $U_{2k-1, k}$, every two hyperedges intersect and
the minimum size of a hitting set of $U_{2k-1, k}$ is precisely $k$.

\begin{THMMAIN2}
Let $\ell\ge 5$ be a positive integer. 
Then the class of cycles of length at least $\ell$ has no Erd\H{o}s-P\'osa property under the induced subgraph relation.
\end{THMMAIN2}

\begin{proof}
Suppose for contradiction that there is a function $f_{\ell}:\mathbb{N}\rightarrow \mathbb{N}$ such that
for every graph $G$ and a positive integer $k$, either
\begin{itemize}
\item $G$ contains $k+1$ pairwise vertex-disjoint holes of length at least $\ell$ or
\item there exists $T\subseteq V(G)$ with $\abs{T}\le f_{\ell}(k)$ such that $G- T$ contains no holes of length at least $\ell$.
\end{itemize} 
Let $x=\max \{f_{\ell}(1)+1, \ell\}$. From the hypergraph $U_{2x-1, x}=(X,\mathcal{E})$,  
we construct a graph $G$ on the vertex set $S\uplus \bigcup_{F\in \mathcal{E}}Y_F$, where
\begin{itemize}
\item $S=\{s_v:v\in X\}$ is an independent set of size $\abs{X}$,
\item $Y_F=\{y_v:v\in F\}$ is an independent set of size $x$ for each $F\in \mathcal{E}$.
\end{itemize}
The edge set of $G$ is created as follows.
\begin{itemize}
\item For each hyperedge $F\in \mathcal{E}$ with $F=\{v_i:1\le i\le x\}$,  we add the edge set
\[\{y_{v_1}s_{v_1},s_{v_1}y_{v_2},\ldots , y_{v_x}s_{v_x},s_{v_x}y_{v_1}\}.\]
\item For each pair of two distinct hyperedges $F_1, F_2\in  \mathcal{E}$, we add all possible edges between $Y_{F_1}$ and $Y_{F_2}$.
\end{itemize}
Note that for each $F\in \mathcal{E}$, $G[Y_F\cup S]$ contains precisely one hole, which has length $2x(\ge \ell)$. We denote this hole as $C_F$.
Figure~\ref{fig:construction} depicts the construction.

We verify that every hole of length at least $\ell$ is one of the holes in $\{C_F:F\in \mathcal{E}\}$.

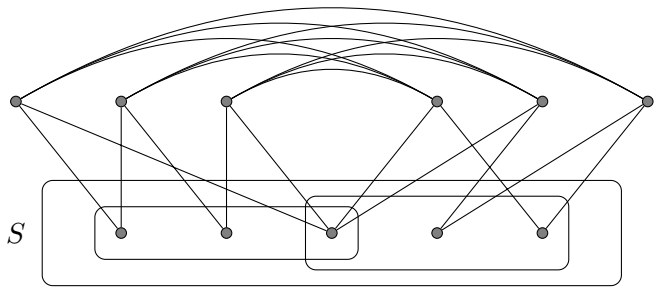
\begin{figure}
  \centering
  \begin{tikzpicture}[scale=0.7]
  \tikzstyle{w}=[circle,draw,fill=black!50,inner sep=0pt,minimum width=4pt]

 \foreach \y in {0, 2, 4, 6, 8}{
      \draw (\y,0) node [w] (a\y) {};
     }

 \foreach \y in {0, 2, 4}{
      \draw (\y-2,2.5) node [w] (b\y) {};
     }

 \foreach \y in {4, 6, 8}{
      \draw (\y+2,2.5) node [w] (c\y) {};
     }
     
     \draw (b0)--(a0)--(b2)--(a2)--(b4)--(a4)--(b0);
     \draw (c4)--(a4)--(c6)--(a6)--(c8)--(a8)--(c4);

	     \draw(b0) [in=150,out=30] to (c4); 
	     \draw(b0) [in=150,out=30] to (c6); 
	     \draw(b0) [in=150,out=30] to (c8); 
	     \draw(b2) [in=150,out=30] to (c4); 
	     \draw(b2) [in=150,out=30] to (c6); 
	     \draw(b2) [in=150,out=30] to (c8); 
	     \draw(b4) [in=150,out=30] to (c4); 
	     \draw(b4) [in=150,out=30] to (c6); 
	     \draw(b4) [in=150,out=30] to (c8); 
 
\draw[rounded corners] (0,1)--(9.5,1)--(9.5,-1)--(-1.5,-1)--(-1.5,1)--(0,1);

\draw[rounded corners] (0,.5)--(4.5,.5)--(4.5,-.5)--(-.5,-.5)--(-.5,.5)--(0,.5);
\draw[rounded corners] (4,.7)--(8.5,.7)--(8.5,-.7)--(4-.5,-.7)--(4-.5,.7)--(4,.7);

     \node at (-2, 0) {$S$};
 
   \end{tikzpicture}     \caption{An illustration of two holes constructed from two hyperedges.}\label{fig:construction}
\end{figure}

\begin{CLAIM}\label{claim:chordlesscycle}
Every hole of length at least $\ell$ is exactly one of the holes in $\{C_F:F\in \mathcal{E}\}$.
\end{CLAIM}
\begin{proofofclaim}
Suppose $C$ is a hole of length at least $\ell\ge 5$.
We show that $V(C)\subseteq V(C_F)$ for some $F\in \mathcal{E}$. Clearly, it implies the claim as each $C_F$ is a hole.

Suppose for contradiction that $C$ is not contained in one of $\{C_F:F\in \mathcal{E}\}$. Then there are two distinct 
hyperedges $F, F'\in \mathcal{E}$ such that $V(C)\cap Y_F\neq \emptyset$ and $V(C)\cap Y_{F'}\neq \emptyset$. 
Let $v\in V(C)\cap Y_F$ and $v' \in V(C)\cap Y_{F'}$. Due to the construction of $G$, we have $vv'\in E(G)$. Furthermore, 
this also implies that for every $F''\in \mathcal{E}\setminus \{F,F'\}$, we have $V(C)\cap Y_{F''}=\emptyset$. 

Since $S$ is independent, among the vertices of $V(C)\setminus \{v,v'\}$ there are at least $\lfloor (\abs{V(C)}-2)/2 \rfloor$ vertices 
of $Y_{F}\cup Y_{F'}$. Suppose $V(C)\setminus \{v,v'\} \setminus S$ has two vertices $w$ and $w'$. 
If both of $w$ and $w'$ are in $Y_F$, then $v'$ is adjacent to at least three vertices of $C$, a contradiction. 
Therefore, we may assume that $w\in Y_F$ and $w'\in Y_{F'}$. Then $G[\{v,v',w,w'\}]$ is a cycle of length $4$, 
contradicting the assumption that $C$ is a hole of length at least $\ell(\ge 5)$. If $V(C)\setminus \{v,v'\} \setminus S$ 
contains a unique vertex, say $w\in Y_F$, observe that $\abs{V(C)}=5$ and $wv'$ is a chord of $C$, a contradiction.
\end{proofofclaim}

By Claim~\ref{claim:chordlesscycle}, $\{C_F:F\in \mathcal{E}\}$ is precisely the set of all holes of length at least $\ell$ in $G$.
One can observe that two holes in $\{C_F:F\in \mathcal{E}\}$ intersect because $(X,\mathcal{E})$ is the hypergraph $U_{2x-1, x}$, 
in which every two hyperedges intersect.
Therefore, by the property of the function $f_{\ell}$, 
there exists a vertex subset $T\subseteq V(G)$ with $\abs{T}\le f_{\ell}(1)<x$ 
such that $G- T$ contains no holes of length at least $\ell$.
We may assume that $T$ is a subset of $S$; a vertex of $Y_F$ 
hits no hole of $G$ other than  the hole $C_F$, which can be hit by choosing 
the corresponding vertex of $S$ instead.
However, there is always a hyperedge avoiding a set of $x-1$ elements in the hypergraph $U_{2x-1,x}$, 
and it implies that $G-T$ contains a hole in $\{C_F:F\in \mathcal{E}\}$.
This is a contradiction.
We conclude that such a function $f_{\ell}$ does not exist.
\end{proof}

For an integer $\alpha\ge 2$, a graph class $\mathcal{C}$ has the \emph{$1/ \alpha$-integral Erd\H{o}s-P\'osa property} under a graph containment relation $\le_{\star}$ 
if there exists a function $f: \mathbb{N} \rightarrow \mathbb{N}$ such that
for every graph $G$ and a positive integer $k$, $G$ contains either
\begin{itemize}
\item $k+1$ pairwise distinct subsets $Z_1,\ldots , Z_k$ such that each subgraph of $G$ induced by $Z_i$ contains a member of $\mathcal{C}$ under $\le_{\star}$ and each vertex of $G$ is contained in at most $\alpha$ sets of $Z_1, \ldots, Z_k$, or
\item a vertex set $T$ of $G$ such that $\abs{T}\le f(k)$  and $G- T$ contains no member of $\mathcal{C}$ under $\le_{\star}$.
\end{itemize} 
Sometimes, a class of graphs that does not have the Erd\H{o}s-P\'osa property 
has the $1/2$-integral Erd\H{o}s-P\'osa property.
For example, the class of odd cycles has the $1/2$-integral Erd\H{o}s-P\'osa property (under the subgraph relation), while it has no Erd\H{o}s-P\'osa property~\cite{Reed1999}.

Simply modifying the proof of Theorem~\ref{thm:main2}, we can show that for any fixed integers $\alpha\ge 2$ and $\ell \ge 5$, the class of cycles of length at least $\ell$ does not have the $1/ \alpha$-integral Erd\H{o}s-P\'osa property under the induced subgraph relation. 
The main idea we used in Theorem~\ref{thm:main2} is that the set of hyperedges in the uniform hypergraph $U_{2x-1, x}$ satisfies that two hyperedges always intersect and the size of a hitting set for $U_{2x-1, x}$ is at least $x$.

Let us consider the uniform hypergraph $U_{(\alpha+1) x-1, \alpha x}$. We claim that any tuple of $\alpha+1$ hyperedges in $U_{(\alpha+1) x-1, \alpha x}$ has a common intersection. Inductively, one can verify that for all integers $2\le t\le \alpha+1$, any tuple of $t$ hyperedges has at least $(\alpha+1-t)x +(t-1)$ common intersections. Thus, any tuple of $\alpha+1$ hyperedges has a non-empty intersection.  But the size of a hitting set for $U_{(\alpha+1) x-1, \alpha x}$ is at least $x$. Thus, by taking $x = \max\{f_{\ell}(\alpha) + 1, \ell\}$ and replacing $U_{2x-1, x}$ with $U_{(\alpha+1) x-1, \alpha x}$ in the proof of Theorem~\ref{thm:main2}, we obtain that the class of cycles of length at least $\ell$ does not have the $1/ \alpha$-integral Erd\H{o}s-P\'osa property under the induced subgraph relation.

\section{Applications of the Erd\"os-P\'osa property for holes}\label{sec:applications}

\subsection{Packing and covering weighted cycles}

We show the weighted version of Erd\"os-P\'osa property of cycles.
We recall that for a graph $G$ and a non-negative weight function $w:V(G)\rightarrow \mathbb{N}\cup \{0\}$, 
let $\pack(G, w)$ be the maximum number of cycles  (repetition is allowed)  such that each vertex $v$ is used  at most $w(v)$ times, and
let $\cover(G, w)$ be the minimum value $\sum_{v\in X} w(v)$ where $X$ hits all cycles in $G$.

\begin{COR}
For a graph $G$ and a non-negative weight function $w:V(G)\rightarrow \mathbb{N}\cup \{0\}$, 
$\cover(G,w)\le O(k^2\log k)$ where $k=\pack(G,w)$.
\end{COR}

\begin{proof}
We may assume that $w(v)$ is positive for every $v\in V(G)$. Let $k=\pack(G,w)$.
We construct a new graph $H$ from $G$ as follows:
\begin{enumerate}
\item We obtain a graph $G'$ from $G$ by subdividing each edge $uv$ into an induced path $u-e_{uv}- v$, and give the weight $w(e_{uv}):=\min \{w(u), w(v)\}$.
\item For each vertex $v$ in $G'$, let $Q_v$ be a complete graph on $w(v)$ vertices.
\item Let $H$ be the graph obtained from the vertex-disjoint union of all graphs in $\{Q_v:v\in V(G')\}$ by 
adding all edges between $Q_v$ and $Q_x$ for each edge $vx$ in $G'$.
\end{enumerate}
We will show that 
\begin{itemize}
\item $\pack (G, w)$ is the same as the number of maximum pairwise vertex-disjoint holes in $H$, and 
\item $\cover(G, w)$ is the same as the size of a minimum vertex set $S$ in $H$ such that $H-S$ has no holes.  
\end{itemize}
By Theorem~\ref{thm:main}, this implies $\cover(G,w)\le O(k^2\log k)$.

First, each hole 
$C$ of $H$ intersects a complete subgraph $Q_v$ at most once for every $v\in V(G')$; otherwise, $H$  contains a triangle.
Also, if a hole $C$ intersects a complete subgraph $Q_v$ for some vertex $v$ in $G'$, 
then $C$ traverses precisely two complete subgraphs among $\{Q_x:x\in N_{G'}(v)\}$.
This means that each hole $C$ of $H$ corresponds to a cycle of $G'$, and thus to a cycle of $G$.
It easily follows that $\pack(G, w)$ is as large  as the number of maximum pairwise vertex-disjoint holes in $H$.
Conversely, let $\mathcal{P}$ be a packing of cycles of $G$ in which every vertex $v\in V(G)$ is used at most $w(v)$ times. 
Clearly, each cycle of $G$ yields a canonical hole of $H$ due to the subdivision in the intermediate graph $G'$.
It is easy to see that one can build a packing $\mathcal{P}'$ of vertex-disjoint holes of $H$ from $\mathcal{P}$ 
such that $\abs{\mathcal{P}'}=\abs{\mathcal{P}}$. Therefore, $\pack(G,w)$ equals the number of maximum pairwise vertex-disjoint holes in $H$. 

Let $S$ be a minimum-sized vertex set in $H$ such that $H-S$ has no holes.
We observe that if $S$ contains a vertex of $Q_v$ for some $v\in V(G')$, 
then $V(Q_v)\subseteq S$ because of the minimality of $S$ and the fact that all vertices of $Q_v$ have the same neighborhood.

If $S$ contains  $Q_{e_{xy}}$ for a subdividing vertex $e_{xy}$ of $G'$ such that $w(e_{xy})=w(x)$, then
the set $(S\cup V(Q_x))\setminus V(Q_{e_{xy}})$ hits every hole of $H$.
Hence, we can assume that $S$ contains only vertices of $Q_v$ for $v\in V(G)$. 
Now, the vertex subset $S':=\{v\in V(G): V(Q_v)\subseteq S\}$ of $V(G)$ hits every cycle of $G$ and $\sum_{v\in S'}w(v)=\abs{S}$. 
This implies that $\cover(G, w)\le \abs{S}$.
Conversely, if $G$ contains a solution $S$, then we can simply take $\bigcup_{v\in S}Q_v$ to hit every hole of $H$. 
Therefore, 
$\cover(G, w)$ equals the size of a minimum vertex set $S$ of $H$ such that $H-S$ has no holes.
\end{proof}

\subsection{Approximations for \textsc{Chordal Vertex Deletion}}

Theorem~\ref{thm:main} can be converted into an approximation algorithm of factor $O({\sf opt}\log {\sf opt})$ for {\sc Chordal Vertex Deletion}, where 
${\sf opt}$ is the minimum number of vertices whose deletion makes the input graph $G$ chordal.
We may assume that the input graph $G$ is not chordal. 
The approximation algorithm works as follows. 
We first greedily construct a maximal packing of $p$ vertex-disjoint holes. For $k=p,\ldots$, we apply the algorithm of Theorem~\ref{thm:main}.
If it outputs $k+1$ holes, then we increase $k$ by one and recurse. If a hitting set $X$ of size $O(k^2\log k)$ is returned, then 
we return $X$ as an approximate solution for {\sc Chordal Vertex Deletion} and terminate the algorithm.
To see that this procedure achieves the claimed approximation factor, notice that performing the algorithm of Theorem~\ref{thm:main} for $k$ means that in the previous step (whether it was the greedy packing step or the run of  the 
algorithm of Theorem~\ref{thm:main}) found $k$ vertex-disjoint holes. Therefore, we have ${\sf opt}\geq k$. 
In particular, when a hitting set $X$ is returned, we have $\abs{X}\leq c_1k^2\log k +c_2\leq c_1 {\sf opt}^2\log {\sf opt} + c_2$. Here $c_1, c_2$ are the constants as in Theorem~\ref{thm:main}.
As we run the polynomial-time algorithm of Theorem~\ref{thm:main} at most $n$ times, clearly this approximation runs in polynomial-time.

The following statement summarizes the result.

\begin{THM}[Restatement of Theorem~\ref{thm:main3}]
There is an approximation algorithm of factor $O({\sf opt}\log {\sf opt})$ for {\sc Chordal Vertex Deletion}. 
\end{THM}

\section{Concluding remarks}

We show that the class of cycles of length at least $4$ has the  Erd\H{o}s-P\'osa property under the induced subgraph relation, with a gap function $f(k)=c_1k^2\log k+c_2$ for some constants $c_1$ and $c_2$. 
A natural question is whether an improved bound can be obtained.
A lower bound $c'k\log k$ for some constant $c'$ is known for the Erd\H{o}s-P\'osa property on cycles~\cite{ErdosP1965}. 
The following reduction shows that this is also a lower bound on a gap function for holes. Given a graph $G$, let $G'$ be a graph obtained by subdividing each edge of $G$ once. 
Then the girth of $G'$ is at least $4$, and there is an obvious one-to-one correspondence between cycles of $G$ and holes of $G'$.
Since we may assume that a minimum-sized vertex set hitting every hole of $G'$ contains no subdividing vertex, the packing and covering 
numbers for cycles of $G$ equals the packing and covering numbers for holes of $G'$, respectively.

One might ask whether or not variants of cycles satisfy the Erd\H{o}s-P\'osa property under the induced subgraph relation.
We list some of open problems.
\begin{itemize}
\item It was shown in~\cite{KakimuraKM2011, PontecorviW2012} 
that any graph with a vertex set $S$ contains either $k+1$ vertex-disjoint $S$-cycles or a vertex set of size $O(k\log k)$ hitting all $S$-cycles, where $S$-cycles are cycles intersecting $S$.  As a generalization, Bruhn, Joos, and Schaudt~\cite{BruhnJS2017} showed that the class of long $S$-cycles also has the Erd\H{o}s-P\'osa property. It is easy to see that the class of $S$-cycles has the Erd\H{o}s-P\'osa property under the induced subgraph relation, because every $S$-cycle contains an induced $S$-cycle.
Determining whether or not the class of $S$-cycles of length at least $4$ has the Erd\H{o}s-P\'osa property under the induced subgraph relation is an open problem.
\item Huynh, Joos, and Wollan~\cite{HuyneJW2017} generalized the  results of~\cite{KakimuraKM2011, PontecorviW2012}  to $(S_1, S_2)$-cycles, which intersect two prescribed vertex sets $S_1$ and $S_2$.
The same question can be asked for the cycles of length at least $4$ intersecting two prescribed sets $S_1$ and $S_2$.
\item Thomassen~\cite{Thomassen1988} proved that 
the class of even cycles has the Erd\H{o}s-P\'osa property. 
One might ask whether the class of even cycles has the Erd\H{o}s-P\'osa property under the induced subgraph relation.
\end{itemize}

Our construction for no Erd\H{o}s-P\'osa property in Section~\ref{sec:lowerbound} 
creates many holes of length exactly $4$, and we are not aware of any construction which does not feature (many) holes of length $4$. 
Does the class of cycles of length at least $\ell$, for fixed $\ell\ge 6$, has the Erd\H{o}s-P\'osa property on $C_4$-free graphs under the induced subgraph relation? 
The answer is not clear to us. When $\ell=5$, the Erd\H{o}s-P\'osa property holds as an immediate consequence of our result.

Our result can be reformulated as the Erd\H{o}s-P\'osa property for the class of $C_4$-subdivisions under the induced subgraph relation. 
Investigating the Erd\H{o}s-P\'osa property  of $H$-subdivisions under the induced subgraph relation for other graphs $H$, and the computational aspect of related covering/packing problems can be a fruitful research direction. 
Recently, the second author and Raymond~\cite{KwonR18} determined, for various graphs $H$, whether the class of $H$-subdivisions has the Erd\H{o}s-P\'osa property under the induced subgraph relation or not. 

\section*{Acknowledgment}
We thank D\'aniel Marx and Eduard Eiben for helpful discussions at an early stage of this work, and also the anonymous reviewers for comments to improve the original manuscript.

\end{document}